\documentclass[12pt]{article}
\usepackage{fullpage}
\usepackage{amssymb}
\usepackage{ntheorem}
\usepackage{graphics} 
\usepackage{amsmath}
\usepackage[color=yellow]{todonotes}
\usepackage{url}
\usepackage{xcolor}
\usepackage{longtable}
\usepackage[hidelinks]{hyperref}
\usepackage{bbm}

\usepackage{natbib}

\newenvironment {proof}{{\noindent\bf Proof }}{\hfill $\Box$ \medskip}

\newtheorem{theorem}{Theorem}[section]
\newtheorem{lemma}[theorem]{Lemma}

\newtheorem{proposition}[theorem]{Proposition}
\newtheorem{remark}[theorem]{Remark}

\newtheorem{corollary}[theorem]{Corollary}

\newtheorem{definition}[theorem]{Definition}
\newtheorem{notation}[theorem]{Notation}
\newtheorem{assumptions}[theorem]{Assumptions}

\newtheorem*{terminology-non}[theorem]{Terminology}

\newcommand{\revpoint}[2]{}
\newcommand{\llabel}[1]{}

\def \hat{\widehat}
\def \tilde{\widetilde}
\def \bar{\overline}
\newcommand{\IP}{\mathbb P}

\newcommand{\IE}{\mathbb E}
\newcommand{\IR}{\mathbb R}

\newcommand{\IN}{\mathbb N}

\newcommand{\ind}{\mathbf{1}}
\newcommand{\bigO}{\mathcal{O}}
\newcommand{\grad}{\nabla}

\newcommand{\DG}{\mathcal{B}}  
\newcommand{\DD}{\mathcal{D}}  
\newcommand{\meanq}{\vec b}    
\newcommand{\covq}{\mathbf{C}}     
\newcommand{\kernel}{\rho}  
\newcommand{\smooth}[1]{\kernel_{#1} \! * \!}  
\newcommand{\wavespeed}{\mathfrak{c}}    
\newcommand{\Lgen}{\mathcal{L}}    
\newcommand{\Pgen}{\mathcal{P}}    
\newcommand{\lp}{\xi}              
\newcommand{\labelspace}{\mathcal{I}} 
\newcommand{\concat}{\oplus}   
\newcommand{\measures}{\mathcal{M}_F(\IR^d)} 
\newcommand{\cmeasures}{\mathcal{M}_F(\overline{\IR}^d)} 
\newcommand{\lpmeasures}{\mathcal{M}(\overline{\IR^d} \times [0,\infty))} 
\newcommand{\flpmeasures}{\mathcal{M}(\IR^d \times [0,\infty))} 

\numberwithin{equation}{section}

\begin{document}

\title{\large{\bf
Looking forwards and backwards: dynamics and genealogies of locally regulated populations
}}

\author{ \begin{small}
\begin{tabular}{ll}                              
Alison M. Etheridge 
 & Thomas G. Kurtz \\   
Department of Statistics & Departments of Mathematics and Statistics\\       
Oxford University & University of Wisconsin - Madison \\                   
24-29 St Giles & 480 Lincoln Drive\\                                                         
Oxford OX1 3LB & Madison, WI  53706-1388\\
UK & USA \\                        
etheridg@stats.ox.ac.uk & kurtz@math.wisc.edu     \\
\url{http://www.stats.ox.ac.uk/~etheridg/} & 
\url{http://www.math.wisc.edu/~kurtz/}  \\       \\
\\
Ian Letter&  Peter L. Ralph 
\\   
Department of Statistics & Departments of Mathematics and Biology\\
Oxford University &University of Oregon\\                   
24-29 St Giles & Fenton Hall\\
Oxford OX1 3LB & Eugene, OR 97403-1222\\
UK & USA \\
restucci@stats.ox.ac.uk  & plr@oregon.edu \\
\url{https://www.stats.ox.ac.uk/~restucci/}&
\url{https://math.uoregon.edu/profile/plr} \\
\\
Terence Tsui Ho Lung
 &  \\   
Department of Statistics & \\
Oxford University & \\                   
24-29 St Giles& \\
Oxford OX1 3LB & \\
UK & \\
terence.tsui@sjc.ox.ac.uk &      \\
\url{https://www.maths.ox.ac.uk/people/terence.tsui}&  \\
\end{tabular}
\end{small}}

\date{\today}
\maketitle

\begin{abstract}
We introduce a broad class of mechanistic spatial models to describe how spatially heterogeneous populations live, die, and reproduce.  Individuals are represented by points of a point measure, whose birth and death rates can depend both on spatial position and local population density, defined at a location to be the convolution of the point measure with a suitable non-negative integrable kernel centred on that location. We pass to three different scaling limits: an interacting superprocess, a nonlocal partial differential equation (PDE), and a classical PDE.  The classical PDE is obtained both by a two-step convergence argument, in which we first scale time and population size and pass to the nonlocal PDE, and then scale the kernel that determines local population density; and in the important special case in which the limit is a reaction-diffusion equation, directly by simultaneously scaling the kernel width, timescale and population size in our individual based model.

A novelty of our model is that we explicitly model a juvenile phase.  The number of juveniles produced by an individual depends on local population density at the location of the parent; these juvenile offspring are thrown off in a (possibly heterogeneous, anisotropic) Gaussian distribution around the location of the parent; they then reach (instant) maturity with a probability that can depend on the local population density at the location at which they land.  Although we only record mature individuals, a trace of this two-step description remains in our population models, resulting in novel limits in which the spatial dynamics are governed by a nonlinear diffusion. 

Using a lookdown representation, we are able to retain information about genealogies relating individuals in our population and, in the case of deterministic limiting models, we use this to deduce the backwards in time motion of the ancestral lineage of an individual sampled from the population. We observe that knowing the history of the population density is not enough to determine the motion of ancestral lineages in our model. We also investigate (and contrast) the behaviour of lineages for three different deterministic models of a population expanding its range as a travelling wave: the Fisher-KPP equation, the Allen-Cahn equation, and a porous medium equation with logistic growth.

\vspace{.1in}


\vspace{.1in}

\noindent {\bf MSC 20}10 {\bf Subject Classification:}  Primary:  
\\Secondary:   
 
\end{abstract}
\tableofcontents
\newpage

\section{Introduction}
\label{introduction}

As one takes a journey, long or short, the landscape changes:
forests thicken or thin or change their composition;
even in flat plains,
springtime prairies host intergrading mosaics of different types of flowers.
The aim of this paper is to introduce and study
a broad class of mechanistic spatial models
that might describe how spatially heterogeneous populations live, die, and reproduce.
Questions that we (start to) address include:
How does population density change across space and time?
How might we learn about the underlying dynamics from genealogical or genetic data?
And, how does genetic ancestry spread across geography
when looking back through time in these populations?

Reproduction of individuals naturally leads to spatial branching process models,
including branching random walk, 
branching Brownian motion, and the Dawson-Watanabe superprocesses.  
However, as a result of the branching assumption (once born, individuals behave 
independently of one another), 
a population evolving according to any of these models will
either die out or grow without bound and, in so doing, can
develop clumps of arbitrarily large density and extent. 
Our starting point here is an individual-based model of a single species 
in continuous space in which
birth, death, and establishment may all depend on local population density
as well as on spatial location,
allowing for stable populations through density-dependent feedback.
The model generalizes those introduced to the ecology literature by
\citet{bolker/pacala:1997} and \citet{law/murrell/dieckmann:2003},
and our work follows various others in the mathematical literature
(e.g., \citet{etheridge:2004,fournier2004microscopic}).

Although it is often mathematically convenient to assume that individuals follow
Brownian motion during their lifetime, 
in our model, offspring are thrown off according to some 
spatial distribution centred on the location of the parent
and do not subsequently move.
This is particularly appropriate for modelling plant populations, in which
this dispersal of offspring around the parent is the only source of
spatial motion.

Often models do not distinguish between juveniles and adults, so,
for example, the number of adults produced by a single parent is determined only
by the degree of crowding at the location of the parent.
Although we shall similarly only follow the adult population, in formulating the dynamics of the
models we shall distinguish
between production of juveniles, which will depend upon the location of 
the adult, and their successful establishment, which will depend on the
location in which a juvenile lands. The result is that not only the absolute 
number, but also the spatial distribution
around their parent, 
of those offspring that survive to adulthood
will depend upon the local population 
density. 

We shall consider three different classes of scaling limits for our model.
The first yields a class of (generalised) superprocesses in which coefficients 
governing both the spatial motion and the branching components of the process can depend
on local population density; the second is a corresponding class of deterministic
non-local differential equations; and the third are classical PDEs.
We measure local population density around a point by convolving with
a smooth kernel $\rho(\cdot)$, which may differ
for the two stages of reproduction. 
When the limiting population process is deterministic,
it is a (weak) solution of an equation of the form
\begin{equation}
	\label{general deterministic limit}
        \partial_t \varphi_t(x)
        =
        r\left(x, \varphi_t \right)
        \DG^* \left[
            \varphi_t(\cdot)
            \gamma\big(\cdot, \varphi_t \big)
        \right](x)
        +
        \varphi_t(x)
        F\left(x, \varphi_t \right)
        ,
\end{equation}
where
$\varphi_t(x)$ can be thought of as the population density at $x$
(although the limit may be a measure without a density), and
$\DG^*$ is (the adjoint of) a strictly uniformly elliptic second order differential operator, typically the Laplacian.
The dependence of each of the terms $r$, $\gamma$, and $F$ on $\varphi$ is only through the local density at $x$,
e.g., $F(x, \varphi) = F(x, \smooth{} \varphi(x))$.
We shall be more specific about the parameters below. 

By replacing $\rho$ by
$\rho^\epsilon(\cdot)=\rho(\cdot/\epsilon)/\epsilon^d$, 
we can also scale the ``width'' of the 
region over which we measure local population density. 
When the population follows~(\ref{general deterministic limit}), 
we expect that if we take a second limit of $\epsilon \to 0$,
thus scaling the kernels appearing in $r$, $\gamma$, and $F$ and making interactions pointwise,
we should recover a nonlinear PDE.
We verify that this is indeed the case in two important examples:
a special case of the porous medium equation with a logistic growth term, in which
the limiting equation takes the form 
\begin{equation}
	\label{PME1}
	\partial_t \varphi = \Delta (\varphi^2)+\varphi(1-\varphi);
\end{equation}
and a wide class of semi-linear PDEs of the form 
\begin{equation}
\label{semilinear PDE}
\partial_t \varphi = \DG^*\varphi+ \varphi F(\varphi),
\end{equation}
which includes the Fisher-KPP equation and the Allen-Cahn equation.
Equations of this form have been studied extensively
in the context of spatial ecology
(see for instance \citet{lam2023introduction} and \citet{cantrell2004spatial})
and in many other fields;
for instance, \citet{ghosh2022emergent}
derive a stochastic version of~\eqref{semilinear PDE} to describe abundances of mutant bacteria strains
along the human gut,
\citet{li/buenzli/simpson:2022} study the effects of nonlinear diffusion
on long-term survival of a lattice-based interacting particle system,
and \citet{birzu/hallatschek/korolev:2017} describe genetic variation in expanding waves
using both forwards and backwards-time arguments.
Most of this work is theoretical; for empirical studies see for instance
\citet{adler2018interspecific} or \citet{zhu2023density}.
We do not study the effect of movement of adults,
which can additionally affect the limiting equations:
see for instance \citet{holmes/lewis/banks/veit:1994} or \citet{potts/borger:2023},
especially if movement depends on population density
(as in \citet{perkins:1992} and \citet{birzu2019genetic}).

It is of interest to understand under what conditions we can replace the two-step limiting
process described above by one in which we simultaneously scale the 
kernels
and the other parameters in 
our population model to arrive at the PDE limit.
This is mathematically much more challenging, but we establish such 
one-step convergence in cases for which the limit is a classical 
reaction-diffusion equation of the form~(\ref{semilinear PDE}) with 
$\DG=\Delta$,
and $\rho$ is a Gaussian density. We allow a wide class of reaction terms, $F$, so
that the Fisher-KPP equation 
(that is equation~(\ref{semilinear PDE}) with $\DG=\Delta$ and $F(\varphi)=1-\varphi$)
emerges as a special case.

Such results on (one-step) convergence to reaction-diffusion 
equation limits have been
achieved for a variety of interacting particle systems. 
Following the now classical contributions 
of~\citet{demasi/ferrari/lebowitz:1986, demasi/presutti:1991, oelschlaeger:1985},
much of this work has focused on lattice based models with one particle
per site, or on systems with a fixed number, $N$,
of interacting diffusions as $N\to\infty$.
For systems of proliferating particles,
as considered for example by 
\citet{oelschlaeger:1989, flandoli/leimbach/olivera:2019, flandoli/huang:2021}, 
an additional challenge (also apparent in our models), is 
the control of concentration of particles. 
We follow 
\citet{oelschlaeger:1989, flandoli/leimbach/olivera:2019} in considering
`moderate interactions', meaning that
the number of individuals in the neighbourhood over which we measure local
population density tends to infinity, whereas 
\citet{flandoli/huang:2021} also consider the situation in which
that number remains finite. 
We refer to \citet{flandoli/huang:2021} for a more thorough 
literature review, 
but note that both our model and scaling differ from those considered
in the body of work discussed there: whereas in those settings, the only 
scalings are the number of particles in the system and the size of the
neighbourhood over which individuals interact with one another, in
keeping with the vast literature on continuous state branching models, we 
also scale time and so must ensure that births are 
adequately compensated by deaths to prevent the population from exploding. 

The history of a natural population is often only accessible indirectly, 
through patterns of genetic diversity:
from genetic data, one can try to infer the genealogical trees that describe how
individuals in a sample from the population are related, and these have been shaped 
by its history (see e.g., \citet{neigel1993application,kelleher2019inferring}).
It is therefore of interest to establish information about the distribution of genealogical trees 
under our population model, which we do with a lookdown construction. 
Lookdown constructions were first introduced in~\cite{donnelly/kurtz:1996} to provide 
a mechanism for retaining information about genealogical relationships between individuals
sampled from a population evolving according to the Moran model when passing to the 
infinite population limit.
Since then, they have been extended to a wide range of models. Of particular
relevance to our work here are the papers ~\cite{kurtz/rodrigues:2011} 
and~\cite{etheridge/kurtz:2019}, in which
lookdown constructions are provided for a wide variety of population models, including
spatially structured branching processes.

In general, even armed with a lookdown construction, calculation of relevant statistics 
of the genealogy remains a difficult question. 
However, in special
circumstances, some progress can be made. As an illustration, 
we shall consider a scenario that
has received a great deal of attention in recent years, in which 
a population is expanding into new territory as a travelling wave. 
In Section~\ref{ancestral lineages for travelling waves}
we shall describe the motion of a single ancestral lineage relative to
three different (deterministic) wavefronts across $\IR^1$.

Most work on the topic of ``waves'' of expanding populations
has focused on models that caricature the classical Fisher-KPP equation with a 
stochastic term, i.e.
$$dw=\big(\Delta w +sw(1-w)\big)dt +\sqrt{\frac{\alpha(w)}{N}}W(dt,dx),$$
where $W$ is space-time white noise, and $N$ is a measure of the 
local population density. The coefficient
$\alpha(w)$ is generally taken to be either
$w$, corresponding to a superprocess limit, or $w(1-w)$ giving a 
spatial analogue of a Wright-Fisher diffusion. 
Starting with the pioneering work of~\cite{brunet/derrida/mueller/munier:2006},
a considerable body of evidence has been amassed to underpin the
conjecture that for this, and a wide class of related models, 
genealogies converge on suitable timescales in the infinite density limit
to a Bolthausen-Sznitman coalescent. 
This reflects the fact that, for this equation, 
ancestral lineages become trapped in 
the wavefront, where the growth rate of the population is highest. 
Once there, they will experience 
rapid periods of coalescence 
corresponding to significant proportions of individuals in the front being 
descended from particularly reproductively successful ancestors. 

If one replaces the logistic growth term of the classical
Fisher-KPP equation with a nonlinearity that reflects cooperative
behaviour in the population, such as
\begin{equation}
	\label{Birzu nonlinearity}
	wF(w)=w(1-w)(Cw-1),
\end{equation}
then, for sufficiently large $C$ (strong cooperation),
the nature of the deterministic
wave changes from ``pulled'' to ``pushed'',
\citep{birzu/hallatschek/korolev:2017, birzu/hallatschek/korolev:2021},
and so the genealogies will be quite different from the Fisher-KPP case. 
For example, \citet{etheridge/penington:2022}
show that for a discrete space model corresponding to this 
nonlinearity with $C>2$, after suitable scaling, the genealogy of a
sample converges not to a Bolthausen-Sznitman coalescent, but to
a Kingman coalescent. 
The reason, roughly, is that ancestral lineages
settle to a stationary distribution relative to the position of the 
wavefront which puts very little weight close to the `tip' of the wave, so
that when ancestral lineages meet
it is typically at a location in which population density is
high, where no single ancestor produces a disproportionately large number of 
descendants in a short space of time. 

The shape of the wave is not determined solely by the reaction term.
For example, as a result of the nonlinear diffusion, 
for suitable initial conditions, the solution to the one-dimensional 
porous medium equation with logistic growth~(\ref{PME1}) converges to a travelling
wave with a sharp cut-off; i.e., in contrast to the classical 
Fisher KPP equation, the solution at time $t$
vanishes beyond $x=x_0+ct$ for some constant wavespeed $c>0$ \citep{kamin/rosenau:2004}.
As a first step towards understanding what we should expect in models with
nonlinear diffusion, one can ask about the position of an ancestral lineage
relative to the wavefront in the deterministic models. In 
Section~\ref{ancestral lineages for travelling waves}
we shall see that in our framework, even with logistic growth, the nonlinear diffusion 
corresponding to the porous medium equation results
in a stationary distribution for the ancestral lineage
that is concentrated behind the wavefront, leading 
us to conjecture that in the stochastic equation
the cooperative behaviour captured by the nonlinear diffusion will also 
result in a qualitatively different pattern of coalescence to that seen under the 
stochastic Fisher-KPP equation. Indeed, we believe that it should be feasible to 
show that in an appropriate limit one recovers a Kingman coalescent.

\paragraph{Structure of the paper}

In this paper we study scaling limits of spatial population models,
obtaining convergence of both the population process
(i.e., the population density as a function of time,
although strictly speaking it is a measure that may not have a density)
and of lineages traced back through such a population.
We retain information about lineages as we pass to the scaling limit
by means of a lookdown construction.

In what follows we first study various scaling limits of the spatial population process,
and then turn our attention to lineages traced back through these populations.
First, in Section~\ref{sec: Model and main results},
we describe the model and the main results,
Theorems~\ref{thm:nonlocal_convergence}, \ref{thm:local_convergence}, and~\ref{thm:lineages}.
Next, in Section~\ref{sec:applications}, we discuss a few striking consequences of these results
regarding the behavior of genealogies in travelling waves,
the appearance of periodic ``clumps'' in seemingly homogeneous population models,
and identifiability of the underlying dynamics from a stationary population profile.
In Section~\ref{sec:heuristics}, we provide heuristic explanations
of why the theorems ought to be true,
and some key ideas behind them,
and in Section~\ref{sec:lookdown} we define and discuss the lookdown construction.
Proofs of the results begin in Section~\ref{sec:proofs},
which proves results for population models with nonlocal interactions,
while Section~\ref{subsec:one step convergence proof} gives the more difficult proof
for the case when interaction distances also go to zero in the limit.
Finally, Section~\ref{sec:lookdown_proofs} gives proofs for convergence of the lookdown process
and the associated results for the motion of lineages.
The Appendix contains a few more technical and less central lemmas.
The results are illustrated in a few places with individual-based simulations,
made using SLiM \citep{haller_slim_2019},
but these are provided for visualization and we do not embark on numerical study.

\section{Model and main results}
    \label{sec: Model and main results}

Our model is one of individuals distributed across a continuous space 
which we shall take to be $\IR^d$. 
For applications, $d=1$ or $d=2$
(or even $d=3$ for cells within the body),
but our main results apply more generally.
At time zero, the population is distributed over a bounded region, with 
$\bigO(N)$ individuals per unit area in that region,
so the total number of individuals will also be $\bigO(N)$.
The population changes in continuous time, and
we encode the state of the population at time $t$ by a counting measure $X(t)$,
which assigns one unit of mass to the location of each individual.

Population dynamics are controlled by three quantities,
birth ($\gamma$), establishment ($r$), and death ($\mu$),
each of which can depend on spatial location and local population density
in a way specified below.
Each individual gives birth at rate $\gamma$ to a single (juvenile) offspring,
which is dispersed according to a kernel $q(x, \cdot)$ away from the location $x$ of the parent.
We assume that $q$ is the density of a multivariate Gaussian,
allowing a nonzero mean and anisotropic variance.
Both the mean and covariance of $q$ can change across space,
but do not depend on population density.
The offspring does not necessarily survive to be counted in the population:
it ``establishes'' with probability $r$,
or else it dies immediately.
Independently, each individual dies with rate $\mu$.

We aim to capture universal behaviour by passing to a scaling limit. Specifically, 
we shall take the ``density'', $N$, to infinity,
and also scale time by a factor of $\theta = \theta(N)$,
in such a way that defining $\eta^N(t) = X(\theta t) / N$,
the process $\{\eta^N(t)\}_{t \ge 0}$
will converge to a suitable measure-valued process
as $N$ and $\theta$ tend to infinity,
with the nature of the limit depending on how they tend to infinity together.
Evidently, we also need to scale the dispersal kernel if we are to obtain a
nontrivial limit, for which we use $q_\theta(x,\cdot)$,
the density of the multivariate Gaussian obtained by
multiplying the mean and variance components of $q(x,\cdot)$ by $1/\theta$.

Birth, establishment, and death can depend on the location of the individual
and the local population density.
Since we would like the population density to scale with $N$,
these are functions of $X/N$, i.e.,
the counting measure with mass $1/N$ placed at the location of each individual.
First consider birth rates, defined by
a nonnegative function $\gamma(x, m) : \IR^d \times \IR_{\ge 0} \to \IR_{\ge 0}$
of location $x$ 
and local population density $m$.
Local population density is defined as the convolution of $X/N$ with 
a smooth (non-negative, integrable) kernel $\rho_\gamma(\cdot)$.
We write this convolution as $\smooth{\gamma} X/N$.
Then, when the state of the population is $X$, an individual at location $x$ 
gives birth to a single juvenile offspring at rate
$\gamma(x, \smooth{\gamma} X(x) / N)$.
Similarly, the establishment probability of an offspring at location $y$ is
is $r(y, \smooth{r} X(y) / N)$,
where $r(y, m) : \IR^d \times \IR_{\ge 0} \to [0, 1]$
and again $\smooth{r} X$ is the convolution of 
$X/N$ with the smooth kernel $\rho_r$.

We shall write $\mu_\theta(x, X/N)$ for the per-capita death rate of a mature individual
at $x$ in the population.
In order for the 
population density to change over timescales of order $\theta$, we should like the net 
per capita reproductive rate 
to scale as $1/\theta$. 
In classical models, in which $r$, $\gamma$, and $\mu$ are constant, 
this quantity is simply $r\gamma -\mu$. 
Here, because production of juveniles and their establishment are
mediated by population density measured relative to different points, 
the net reproductive rate will take a more complicated form. 
In particular, the total rate of production of mature offspring by an individual at $x$ will be
\begin{equation} \label{eqn:total_repro}
	\gamma\big(x,\smooth{\gamma}X(x)/N\big)\int r\big(y, \smooth{r}X(y)/N\big)q_{\theta}(x,dy).
\end{equation}
Nonetheless, it will be convenient to define the death rate $\mu_\theta$ in terms of 
its deviation from $r\gamma$. 
To this end, we define
the death rate of an individual at $x$,  \revpoint{1}{4}
using a function $F(x, m) : \IR^d \times \IR_{\ge 0} \to \IR$,
as
\begin{align} \label{eqn:mu_defn}
    \mu_\theta(x, X/N)
    =
        r(x, \smooth{r} X(x) / N) \gamma(x, \smooth{\gamma} X(x) / N)
        - \frac{1}{\theta} F(x, \smooth{F} X(x) / N)
    ,
\end{align}
where $\rho_F$ is again a smooth kernel. 
(We will also assume that parameters are chosen so that this is always nonnegative,
a point we return to later.)
The function $F$ is nearly
the net per capita reproductive rate, scaled by $\theta$,
and would be equal to it in a nonspatial model;
but, as can be seen from \eqref{eqn:total_repro},
differs because an offspring's establishment probability is measured at their new location
rather than that of their parent.
For the most part, we work with $F$ instead of $\mu_\theta$.

So, each of the three demographic parameters $r$, $\gamma$, and $F$,
depends on local density, measured by convolution with a smooth kernel,
each of which can be different.
As a result, death rate depends (in principle) on population densities measured in 
three different ways, 
so that we could write 
$\mu_\theta(x) = \mu_\theta(x, \smooth{\gamma} X(x) / N, \smooth{r} X(x) / N, \smooth{F} X(x) / N)$.
This may seem unnecessarily complex.
However, not only is it natural from a 
biological perspective, it also turns out to be convenient for
capturing nontrivial examples in the scaling limit.

\begin{remark}
Although this model allows fairly general birth and death mechanisms,
there are a number of limitations.
Perhaps most obviously, to simplify the notation
individuals give birth to only one offspring at a time,
although this restriction could be easily lifted
(as in Section 3.4 of \cite{etheridge/kurtz:2019}).
Furthermore, individuals do not move during their lifetime,
and the age of an individual does not affect its fecundity or death rate.
Finally, there is no notion of mating
(although limitations on reproduction due to availability of mates can be incorporated into the birth rate, $\gamma$),
so the lineages we follow will be uniparental.
For these reasons, the model is most obviously applicable to bacterial populations or selfing plants,
although we do not anticipate that incorporation of these complications
will change the general picture.
\end{remark}

For each $N$ and $\theta$, we study primarily the process
with mass scaled by $N$ and time scaled by $\theta$,
$$ \big(\eta^N_t\big)_{t\geq 0} := \big(X(\theta t)/N\big)_{t\geq 0} , $$ 
which takes values in the space of c\`adl\`ag paths in
$\measures$ (the space of finite measures on $\IR^d$
endowed with the weak topology). In fact $\eta_t^N$ will be a purely atomic
measure comprised of atoms of mass $1/N$.

\begin{notation}
Expressions like $\gamma(x, \smooth{\gamma} \eta(x))$ will appear repeatedly in what follows.
To make formulae more readable, we overload notation to define
$$
    \gamma(x, \eta) := \gamma(x, \smooth{\gamma} \eta(x)) ,
$$
and similarly write $r(x, \eta)$ for $r(x, \smooth{r} \eta(x))$,
$F(x, \eta)$ for $F(x, \smooth{F} \eta(x))$,
and $\mu_\theta(x, \eta)$ for the expression of equation~\eqref{eqn:mu_defn}.
When convenient, we may also suppress the arguments completely,
writing simply $\gamma$, $r$, $F$, and $\mu_\theta$ for these quantities.
\end{notation}

\begin{remark}
In our prelimiting model,
the population is represented by a point measure in which each
individual is assigned a mass $1/N$. We use the term ``population density'' for
this process, as it is supposed to measure population size relative to a 
nominal occupancy of $N$ individuals per unit area. There is no implication 
that the measure representing the population is absolutely continuous with
respect to Lebesgue measure; indeed in the prelimit it is certainly not.
\end{remark} 

In summary, at each time $t$, $\eta^N_t$ is purely atomic, consisting of atoms of mass $1/N$
(which are the individuals).
At instantaneous rate $\theta \gamma(x, \eta^N_t) N \eta^N_t(dx)$ an offspring of mass $1/N$ is 
produced at location $x$, which
immediately disperses to a location $y$ offset from $x$
by an independent Gaussian random variable,
and once there establishes instantaneously \revpoint{1}{5}
with probability $r(y, \eta^N_t)$, or else dies. 
The distribution of the dispersal displacement (i.e., $y-x$) may depend on $x$,
and is specified by functions defining the mean $\meanq(x) / \theta$ \revpoint{1}{6}
and covariance matrix $\covq(x) / \theta$.
At instantaneous rate $\theta \mu_\theta(x, \eta^N_t) N \eta^N_t(dx)$ an individual at location 
$x$ dies.
Note that the process 
$\left(\eta^N_t\right)_{t\ge 0}$, which
records numbers and locations of adult individuals, is just a scaled spatial birth and 
death process. If, for example, we insist that $\gamma(x,m)$ is bounded, then 
existence (and in particular non-explosion) is guaranteed by comparison with
a pure birth process. We do not dwell on this, as we shall require more stringent conditions
if we are to pass to the limit as $\theta$ and $N$ tend to infinity.

It is convenient to characterise the process as a solution to a martingale problem.
We write $C_b^\infty(\IR^d)$ for the space of bounded smooth functions on $\IR^d$, and, where
convenient, we
write $\langle f, \eta \rangle = \int_{\IR^d} f(x) \eta(dx)$.

\begin{definition}[Martingale Problem Characterisation]
    \label{defn:mgale_construction}
	For each value of $N$ and $\theta$, and each purely atomic 
$\eta_0^N\in\measures$ with atoms of mass $1/N$,
$(\eta^N_t)_{t \geq 0}$ is the (scaled) empirical measure of a birth-death 
process with c\`adl\`ag paths in $\measures$ for which,
for all $f \in C^{\infty}_{b}(\IR^d)$,
writing $q_\theta(x, dy)$ for the Gaussian kernel
with mean $x + \meanq(x)/\theta$ and covariance $\covq(x) / \theta$,
\begin{equation}
    \label{eqn:eta_martingale}
\begin{aligned}
M^N_t(f)
&:=  \langle f, \eta^N_t \rangle
        -\langle f, \eta^N_0 \rangle
 \\ &\qquad {}
 -  \int_{0}^{t}\bigg\{
        \bigg\langle \left( \int \theta
             \left(
                f(z)r(z, \eta^N_{s})
                - f(x)r(x, \eta^N_{s})
            \right)
        q_\theta(x,dz) \right)
         \gamma(x, \eta^N_{s}), \eta^N_{s}(dx) \bigg\rangle
\\ & \qquad \qquad \qquad \qquad \qquad {}
    + \bigg\langle f(x) F(x, \eta^N_{s}),
    \eta^N_{s}(dx) \bigg\rangle
    \bigg\} ds
\end{aligned}    
\end{equation}
is a martingale (with respect to the natural filtration), 
with angle bracket process
    \begin{equation} \label{eqn:prelimit_martingale_variation}
\begin{aligned} \relax
\left\langle M^N(f) \right\rangle_t =& 
\frac{\theta}{N} \int_{0}^{t}\bigg\{
    \Big\langle \gamma(x, \eta^N_{s})
        \int f^2(z)r(z, \eta^N_{s}) q_\theta(x,dz) 
    , \eta^N_{s}(dx) \Big\rangle \\
& \qquad \qquad {}
    + \Big\langle \mu_\theta(x, \eta^N_{s}) f^2(x) 
    , \eta^N_{s}(dx)\Big\rangle
    \bigg\}ds. 
\end{aligned}    
\end{equation}
\end{definition}

The angle bracket process (or, ``conditional quadratic variation'')
is the unique previsible process making $(M^N(f)_t)^2 - \langle M^N(f) \rangle_t$ a martingale
with respect to the natural filtration.
It differs from the usual quadratic variation
(usually denoted $[M^N(f)]_t$)
because the process has jumps;
for the (continuous) limit the two notions will coincide.
The use of angle brackets for both integrals and this process
is unfortunately standard but should not cause confusion,
since the angle bracket process always carries a subscript for time.

The form of~(\ref{eqn:eta_martingale}) and~(\ref{eqn:prelimit_martingale_variation})
is explained in Section~\ref{sec:heuristics}.
Note that since (juvenile) individuals are produced at rate $N \gamma \eta$, but each
has mass $1/N$,
these factors of $N$ cancel in~\eqref{eqn:eta_martingale}. 
Under our scaling, $N$ and $\theta=\theta(N)$ will tend to infinity in such a 
way that $\alpha:=\lim_{N\to\infty}\theta(N)/N$ exists and is finite.
From the expression~(\ref{eqn:prelimit_martingale_variation}) it is easy to guess that 
whether the limiting processes will be deterministic or stochastic
is determined by whether $\alpha$ is zero or nonzero.

It is convenient to record some notation for the 
generator of the diffusion 
limit of a random walk with jump distribution determined by $q_\theta(x,dy)$.

\begin{definition}[Dispersal generator]
    \label{def:dispersal_generator}
    As above, we define the dispersal kernel,
    $q_\theta(x, dy)$,
    to be the density of a multivariate Gaussian
    with mean $\meanq(x)/\theta$ and covariance matrix $\covq(x)/\theta$
    (although often we omit the dependence of $\meanq$ and $\covq$ on $x$).
    Furthermore, we define for $f \in C_b^\infty(\IR^d)$,
    \begin{align}
    \DG f(x)
        =
        \frac{1}{2} \sum_{ij} \covq(x)_{ij} \partial_{x_i} \partial_{x_j} f(x)
        + \sum_i \meanq(x)_i \partial_{x_i} f(x)
    \end{align}
    and denote the adjoint of $\DG$ by
    \begin{align*}
    \DG^* f(x)
        &=
        \frac{1}{2} \sum_{ij} \partial_{x_i} \partial_{x_j} (\covq(x)_{ij} f(x))
        - \sum_i \partial_{x_i} (f(x) \meanq(x)_i) 
        \\
        &=
        \frac{1}{2} \sum_{ij} C_{ij}(x) \partial_{x_i} \partial_{x_j} f(x)
        + \sum_i \left(
            \frac{1}{2} \sum_j \partial_{x_j} C_{ij}(x) - \meanq_i(x)
        \right) \partial_{x_i} f(x)
        \\ & \qquad {}
        + \left(
            \frac{1}{2} \sum_{ij} \partial_{x_i} \partial_{x_j} C_{ij}(x)
            -
            \sum_i \partial_{x_i} \meanq_i(x)
        \right) f(x) .
    \end{align*}
\end{definition}

\begin{remark}
    \label{rem:DG_limit}
    $\DG$ is defined so that
    \begin{align*}
        \theta \int \left(
            f(y) - f(x)
        \right) q_\theta(x, dy)
    \to \DG f(x) 
        \qquad \text{as } \theta \to \infty .
    \end{align*}
\end{remark}

\begin{remark}
\label{ancestral lineages: first guess}
An equivalent way to describe the model would be to say that
when the state of the population is $\eta$,
an individual at $x$ gives birth at rate
$$
    \theta \gamma(x, \eta)
    \int r(y, \eta) q(x, dy) ,
$$
and that offspring disperse according to the kernel
$q_\theta^\mathfrak{m}$ (with superscript $\mathfrak{m}$ because it is post-\textbf{m}ortality),
defined by: \revpoint{1}{7}
$$
    q_\theta^\mathfrak{m}(x,\eta,dy)
    :=
    \frac{
        r(y, \eta) q_\theta(x, dy)
    }{
        \int r(z, \eta) q_\theta(x, dz)
    } .
$$
Clearly, the random walk driven by this dispersal kernel
is biased towards regions of higher establishment probability.
For comparison with future results,
it is interesting to write down the limiting generator:
\begin{align}
	\label{weighted dispersal}
    \lim_{\theta \to \infty}
    \theta \int (f(y) - f(x)) q_\theta^\mathfrak{m}(x, \eta, dy)
    &=
    \frac{
        \DG\left[ f(\cdot) r(\cdot, \eta) \right](x)
        - 
        f(x) \DG\left[ r(\cdot, \eta) \right](x)
    }{
        r(x, \eta)
    } .
\end{align}
In the simplest case of unbiased isotropic dispersal
(i.e.,~$\meanq = 0$ and $\covq = I$), $\DG = \Delta/2$,
and so~(\ref{weighted dispersal}) is equal to
\begin{align*}
    \frac{1}{2} \Delta f(x) + \grad f(x) \cdot \grad \log r(\cdot, \smooth{r} \eta(\cdot))(x) .
\end{align*}
One might guess that the spatial motion described by following  
the ancestral lineage of an individual 
back through time
would be described (in the limit) by the adjoint of this generator.
However, we will see in Section~\ref{sec:ancestral lineages}
that this is not in fact the case.
\end{remark}

In order to pass to a scaling limit, we will need to impose some 
conditions on the parameters of our model.

\begin{assumptions}
\label{def:model_setup}
We shall make the following assumptions on the parameters of our model.

{\bf Dispersal generator:}
We assume that
\begin{enumerate}
    \item $\meanq(x)$ and $\covq(x)$ are
        $\beta$-H\"older continuous for some $\beta \in (0, 1]$
        and uniformly bounded
        in each component, and
    \item the operator $\DG$ is uniformly strictly elliptic,
        i.e., $\inf_x \inf_{y:\|y\|=1} \sum_{ij} y_i C(x)_{ij} y_j > 0$.
\end{enumerate}

{\bf Reproduction parameters:}
We assume that 
\begin{enumerate}
            \addtocounter{enumi}{2}
\item The function $F(x,m)$ satisfies 
    \begin{enumerate}
    \item $F(x,m)$ is locally Lipschitz in $m$;
    \item $F(x,m)$ is uniformly bounded above (but not necessarily below);
    \item for each fixed $m$, $\sup_{x\in\IR^d}\sup_{k\leq m}|F(x,k)|<\infty$;
    \end{enumerate}
\item The functions $r(x,m)$, $\gamma(x,m)$ have bounded first and second 
    derivatives in both arguments;
    \label{r_gamma_derivs_condition}
\item $\gamma(x,m)$ is uniformly bounded;
    \label{gamma_bounded_condition}
\item For each $f \in C_b^2(\IR^d)$,
    there is a $C_f$ such that
    $$
        |\gamma(x, \eta) \theta \int 
        (r(y,\eta) f(y) - r(x, \eta) f(x))
        q_\theta(x, dy)| \le C_f (1 + |f(x)|)
    $$
    for all $x \in \IR^d$ and $\eta \in \measures$.
    Furthermore, $C_f$ only depends on the norm of the first two derivatives of $f$,
    i.e., 
    \[
        C_f = C(\sup_x \sup_{\|z\|=1} \max( \sum_i z_i \partial_{x_i} f(x), \sum_{ij} z_i z_j \partial_{x_i x_j} f(x))) .
    \] \label{gamma_B_condition} \revpoint{1}{9}
\item To keep expressions manageable, we shall also assume that  \revpoint{1}{1}
    the death rate (as defined in~\eqref{eqn:mu_defn}) is nonnegative, i.e., that
    \[
    \mu_\theta(x) = r(x,\eta) \gamma(x,\eta) - \frac{1}{\theta} F(x,\eta) \ge 0 .
    \]
\end{enumerate}
\end{assumptions}

Since $F$ is bounded above,
the final assumption that $\mu_\theta \ge 0$ --
or, equivalently, that $F(x,\eta) \le \theta r(x,\eta) \gamma(x,\eta)$ --
will always be true for large enough $\theta$ 
as long as $r$ and $\gamma$ are bounded away from zero.

Since we take bounded $f$,
for most situations the bound $C_f(1+|f(x)|)$ in Condition~\ref{gamma_B_condition} above
can be safely replaced simply by $C_f$;
however, this will be useful
in certain situations where we consider a sequence of $f$ with increasing upper bounds.
We now give two concrete situations in which Condition~\ref{gamma_B_condition} is satisfied.
The proof is in Section~\ref{sec:preliminary_proofs}.

\begin{lemma}
    \label{lem:conditions on r}
    Assume that Conditions~\ref{def:model_setup} are satisfied,
    except for Condition~\ref{gamma_B_condition}.
    If either
\begin{enumerate}
\item
\label{control through r} 
    $|\partial_{x_i}r(x,\eta)|$ and 
    $|\partial_{x_ix_j}r(x,\eta)|$
    are uniformly bounded for $x\in\IR^d$, $\eta\in\measures$;
\item 
\label{control through gamma}
    or, $m^2\gamma(x,m)$ is uniformly bounded and there exists $C<\infty$
    such that for $\theta$ sufficiently large, and all $x\in\IR^d$, $\eta\in\measures$,
    \[
    \theta\int\big(\smooth{r}\eta(y)-\smooth{r}\eta(x)\big)q_\theta(x,dy)\leq C\smooth{\gamma}\eta(x),
    \]
    and
    \[\theta\int\big(\smooth{r}\eta(y)-\smooth{r}\eta(x)\big)^2q_\theta(x,dy)
    \leq C(\smooth{\gamma}\eta(x))^2 ,
    \]
\end{enumerate} 
    then Condition~\ref{gamma_B_condition} is also satisfied.
\end{lemma}

The purpose of the conditions that we have placed on the reproduction parameters is to 
ensure that the net per capita reproduction rate (before time scaling)
is order $1/\theta$. As remarked above, 
because of the non-local reproduction mechanism, it no longer suffices \revpoint{1}{10}
to assume that 
$r(x,\eta)\gamma(x,\eta)-\mu_\theta(x)$ is of order $1/\theta$.
Perhaps the simplest example in which we can see
that non-local reproduction can lead to rapid growth even when $r\gamma=\mu$
is where $\gamma \equiv 1$ and $F \equiv 0$, so that $\mu_\theta = r$,
and $\eta = \delta_x$
(i.e., the population has all individuals at a single location),
so that $\smooth{r} \eta(y) = \rho_r(y)$.
In this case, the mean rate of change of the total population size
is $\int (r(y, \rho_r(y) ) - r(x, \rho_r(x))) q_\theta(x, dy)$;
the first condition of Lemma~\ref{lem:conditions on r} would ensure this is of order $1/\theta$.

If $r(x,m)$ is independent of $m$, then the conditions are easy to 
satisfy; they just require some regularity of $r$ as a function of $x$. 
Condition~\ref{control through r} of Lemma~\ref{lem:conditions on r} is also satisfied if for example
$\|\nabla\rho_r\|\leq C\rho_r$ 
and $m\partial_mr(x,m)$, $m^2\partial_{mm}r(x,m)$ are bounded.
This is the case, for instance, if $\rho_r$ decays exponentially.
On the other hand, it might seem more natural to take $\rho_r$ to be a Gaussian 
density with parameter $\sigma_r$, say. Then, 
as we check in Lemma~\ref{lem:gamma_bound},
Condition~\ref{control through gamma} of Lemma~\ref{lem:conditions on r}
is satisfied if $\rho_\gamma$ is also
Gaussian with parameter $\sigma_\gamma$ and $\sigma_\gamma>\sigma_r$. For large enough
$\theta$, this condition guarantees that 
$\sigma_r+1/\theta <\sigma_\gamma$, so that the establishment probability of a juvenile
is controlled by individuals that are already `felt' by the fecundity-regulating kernel $\rho_\gamma$
at the location of their parent.

\subsection{Scaling limits of the population process}

Our main results depend on two dichotomies:
Is the limiting process deterministic or a (generalized) superprocess?
And, are interactions pointwise in the limit or nonlocal?
See Figure \ref{fig:super_vs_det_2d} for snapshots of the population
from direct simulation of the process using SLiM \citep{haller_slim_2019}
illustrating this first dichotomy.
Below we have results for deterministic limits with pointwise and nonlocal interactions,
and for superprocess limits with nonlocal interactions.

\begin{figure}
    \begin{center}
        \includegraphics[width=2.5in]{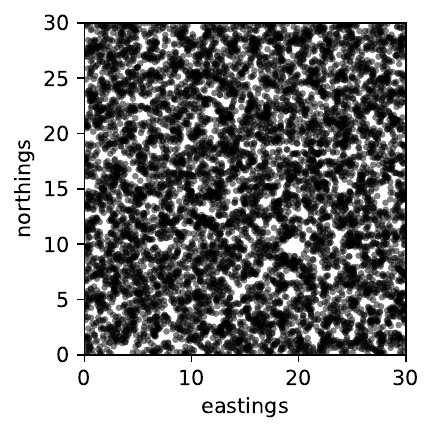}
        \includegraphics[width=2.5in]{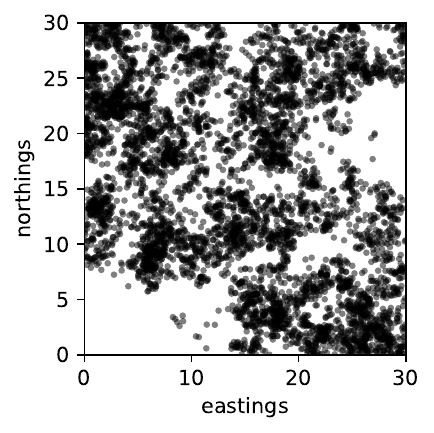}
    \end{center}
    \caption{
        Snapshots of two simulations, with small $\alpha=\theta/N$ (left) and large 
	$\alpha = \theta/N$ (right).
        Simulations are run with a Fisher-KPP-like parameterization:
        birth and establishment are constant, while death increases linearly with density,
        at slope $1/\theta$.
        Left: $\alpha=0.1$. Right: $\alpha=10$.
        Other parameters were the same:
        dispersal ($q_\theta$) and interactions (here, only $\rho_F$)
        are Gaussian with standard deviation 1,
        and the equilibrium density ($N$) is 10 individuals per unit area.
        The remaining parameters are constant: $r \equiv \gamma \equiv 1$.
        \label{fig:super_vs_det_2d}
    }
\end{figure}

\paragraph{Scaling limits with nonlocal interactions:}
Recall that the process $(\eta_t^N)_{t\geq 0}$ takes its values in the space
${\mathcal D}_{[0,\infty)}(\measures)$ of c\`adl\`ag
paths on $\measures$. We endow $\measures$ with the topology of weak convergence
and 
${\mathcal D}_{[0,\infty)}(\measures)$ with the Skorohod topology.
A sequence of processes taking values in 
${\mathcal D}_{[0,\infty)}(\measures)$ is said to be tight if the corresponding
sequence of distributions is tight,
i.e., if any infinite subsequence has a weakly convergent subsubsequence.
Our first main result establishes tightness of our rescaled population processes in 
the case in which interactions remain nonlocal under the scaling, and 
characterises limit points as solutions to a martingale problem.

\begin{theorem} \label{thm:nonlocal_convergence}
    Let $(\eta^N_t)_{t \geq 0}$
    be as defined in Definition~\ref{defn:mgale_construction}
	and assume that as $N \to \infty$, $\theta(N) \to \infty$
	in such a way that $\theta(N)/N \to \alpha$.
    (However, the kernels $\rho_r$, $\rho_\gamma$, and $\rho_F$
    remain fixed.) 
    Suppose that Assumptions~\ref{def:model_setup} hold and, 
    further, that $\{\eta_0^N\}_{N\geq 1}$ is a sequence of 
    purely atomic measures, with $\eta_0^N$ comprised of atoms of mass $1/N$,
    which is tight in $\measures$.
    Also assume there exists a nonnegative $f_0 \in C(\IR^d)$
    with uniformly bounded first and second derivatives
    (i.e., with $\sup_x \sup_{\|z\|=1} \sum_i \partial_{x_i} f_0(x) z_i$
    and $\sup_x \sup_{\|z\|=1} \sum_{ij} \partial_{x_i x_j} f_0(x) z_i z_j$ both finite)
    and $f_0(x) \to \infty$ as $|x| \to \infty$ for which
    $\langle f_0(x), \eta_0^N(dx)\rangle<C<\infty$
    for some $C$ independent of $N$.
    Then the sequence of processes
    $(\eta^N_t)_{t \geq 0}$ is tight, and for any limit point
    $(\eta_t)_{t \geq 0}$, for every $f \in C^\infty_b(\IR^d)$,
    \begin{align} \label{eqn:limiting_mgale_problem}
        \begin{split}
        M_t(f)
            &:=
            \langle f(x), \eta_t(dx) \rangle
            -
            \langle f(x), \eta_0(dx) \rangle
            \\ & \qquad
            -
            \int_0^t \big\langle
                \gamma(x, \eta_s)
                \mathcal{B}\left(
                    f(\cdot) r(\cdot, \eta_s)
                \right)(x)
                +
                f(x)
                F(x, \eta_s),
                \eta_{s}(dx)
            \big\rangle ds
        \end{split}
    \end{align}
    is a martingale (with respect to the natural filtration),
    with angle bracket process
    \begin{align} \label{eqn:limiting_mgale_variation}
        \left\langle M(f) \right\rangle_t
        =
        \alpha
        \int_0^t
        \big\langle
	    2\gamma\left( x, \eta_{s} \right)
            r\left(x, \eta_{s} \right)
            f^2(x),
            \eta_{s} (dx)
        \big\rangle ds. 
    \end{align}
    If $\alpha = 0$ the limit is deterministic.
\end{theorem}

Recall when interpreting~\eqref{eqn:limiting_mgale_problem}
that, for instance, $r(x, \eta_s) = r(x, \smooth{r} \eta_s(x))$,
and so $\DG(fr)(x) = \DG(f(\cdot) r(\cdot, \smooth{r} \eta_s(\cdot)))(x)$.
The proof of this theorem appears in Section~\ref{sec:population_density_proof}.

Theorem \ref{thm:nonlocal_convergence} provides tightness of the
rescaled processes. If the limit points are unique, then this
is enough to guarantee convergence.

\begin{corollary} \label{cor:superprocess_uniqueness}
    Under the assumptions of Theorem~\ref{thm:nonlocal_convergence},
    if the martingale problem
    defined by equations~\eqref{eqn:limiting_mgale_problem} 
and~\eqref{eqn:limiting_mgale_variation}
    has a unique solution,
    then $(\eta^N_t)_{t \ge 0}$ converges weakly
to that solution
    as $N \to \infty$.
\end{corollary}

When $\alpha>0$, the limit points can be thought of as interacting superprocesses. For
example, when $r$ and $\gamma$ are constant, and 
$F(x,\eta_s)=1-\smooth{F} \eta_s(x)$,
we recover a superprocess with nonlinear death rates corresponding to logistic growth
\citep{etheridge:2004}
that is a continuous limit of the Bolker-Pacala model \citep{bolker/pacala:1997,bolker/pacala:1999}.
We are not aware of a general result to determine when we will have uniqueness of 
solutions to the martingale problem of Theorem~\ref{thm:nonlocal_convergence} when $\alpha>0$.
We do not address the question of uniqueness here,
but (as stated below in Proposition~\ref{thm:mmt_application}),
a consequence of the Markov Mapping Theorem (Theorem~\ref{thm:mmt}) is that 
uniqueness of the martingale problem for the (yet to be defined) lookdown process would imply
uniqueness of solutions to this process.

Alternatively, the Dawson--Girsanov transform could be used to show uniqueness \revpoint{1}{12}
in the special case of a superprocess with nonlinear death rates:
if $r$ and $\gamma$ only depend on $x$ (not $\eta$),
then the process with $F=0$ is a heterogeneous branching superprocess
and hence the corresponding martingale problem is unique (see Section 4.3 of \cite{dawson:1993}).
Then, the Dawson--Girsanov transform (Theorem 7.2.2 of \cite{dawson:1993},
extended to measures with nonconstant mass as in Section 10.1.2 of \cite{dawson:1993})
would provide the Radon-Nikodym derivative of the law of the process with more general $F$
relative to the law of the process with $F=0$
under suitable conditions.

In a different but related setting, the Perkins stochastic calculus
(and its adaptation to a lookdown setting \cite{donnelly/kurtz:1999})
provides uniqueness for a different but related class of processes,
in which interactions affect the dispersal mechanism
(rather than reproduction) of the superprocess \cite{perkins:1992}.

For the deterministic case of $\alpha=0$,
the limiting process is a weak solution to a nonlocal PDE.
We next describe some situations in which more is known
about uniqueness and whether the solution is close to the corresponding local PDE.
First, recall the following notion of solution to a PDE.
\begin{definition}[Weak solutions]
    \label{defn:weak_solutions}
    We say that $(\eta_t)_{t \ge 0}$, with $\eta_t \in \measures$,
    is a \emph{weak solution} to the PDE
    \begin{equation} \label{eqn:nonlocal_pde}
        \partial_t \varphi = r \DG^*(\gamma \varphi) + \varphi F
    \end{equation}
	(where $r$, $\gamma$ and $F$ can all be functions of $\varphi$)
    if, for all $f \in C_b^\infty(\IR^d)$,
    \[
        \frac{d}{dt} \langle f, \eta_t \rangle
        =
        \langle
	    \gamma \DG(rf) + f F,
            \eta_t
        \rangle .
    \]
\end{definition}
The notation $\varphi$ is meant to be suggestive of a density,
and recall that equation~\eqref{eqn:nonlocal_pde} has made dependencies on $x$ and $\varphi$ implicit;
written out more explicitly, \eqref{eqn:nonlocal_pde} is
\[
    \partial_t \varphi_t(x)
    =
    r\left(x, \smooth{r} \varphi_t(x) \right)
    \DG^* \left[
        \varphi_t(\cdot)
        \gamma\big( x, \smooth{\gamma} \varphi_t(\cdot) \big)
    \right](x)
    +
    \varphi_t(x)
    F\left(x, \smooth{F} \varphi_t(x) \right)
    .
\]
Because Theorem~\ref{thm:nonlocal_convergence} only tells us about 
weak convergence, in the case $\alpha=0$ we can only deduce that 
any limit point $\eta_t$ is a weak solution to this nonlocal PDE.

Specialising the results of 
\citet{kurtz/xiong:1999} to the deterministic setting
provides general conditions under which we
have existence and uniqueness of solutions to~(\ref{eqn:nonlocal_pde}) which
have an $L^2$-density with respect to Lebesgue measure. 
Recall that the Wasserstein metric, defined by 
\[\rho (\nu_1,\nu_2)=\sup\Big\{\Big|\int fd\nu_1-\int fd\nu_2\Big|:\sup_x|f(x
)|\leq 1,|f(x)-f(y)|\leq \|x-y\|\Big\},\]
determines the topology of weak convergence on $\measures$.
We write 
$r(x,\eta) \gamma(x,\eta) \covq(x) = J(x,\eta) J(x,\eta)^T$, and 
$\beta(x,\eta) = r(x,\eta) \gamma(x,\eta) \big(\meanq(x)+\covq(x)\nabla\log r(x,\eta)\big)$
(quantities that will appear in Proposition~\ref{prop:limiting_construction}).
If $J$, $\beta$, and $F$ are bounded and Lipschitz in the sense 
that
\begin{equation}
\label{lipschitz for kurtz xiong}
|J (x_1, \nu_1)-J (x_2, \nu_2)|,|\beta (x_1,\nu_1)-\beta 
	(x_2,\nu_2)|, |F(x_1,\nu_1)-F(x_2,\nu_2)|
	\leq C(\|x_1-x_2\|+\rho (\nu_ 1,\nu_2))
\end{equation}
for some $C>0$,   
the methods of \citet{kurtz/xiong:1999} show that 
if the initial condition $\eta_0$ for our 
population process has
an $L^2$ density, then so does $\eta_t$ for $t>0$. 
Although the necessary estimates (for which we refer to the original paper)
are highly nontrivial, 
the idea of the proof is simple. Take a solution to the
equation and use it to calculate the coefficients $r$, $\gamma$ and $F$ that depend
on local population density. Then $\eta$ solves the 
{\em linear} equation obtained by regarding those values of $r$, $\gamma$ and $F$ as given.
It remains to prove that the solution to the linear equation has a density
which is achieved by obtaining
$L^2$ bounds on its convolution with the heat semigroup at time $\delta$ and letting 
$\delta\to 0$.
We also have the following uniqueness result.
\begin{theorem}[Special case of \citet{kurtz/xiong:1999}, Theorem~3.5]
	\label{thm:kurtzxiong}
Suppose  $J$, $\beta$, and $F$ are bounded and Lipschitz in the sense 
of~(\ref{lipschitz for kurtz xiong}).
If $\eta_0$ has an $L^2({\mathbb R}^d)$-density, then there 
exists a unique $L^2({\mathbb R}^d)$-valued  
	solution of (\ref{eqn:nonlocal_pde}) in the sense of 
	Definition~\ref{defn:weak_solutions}.
\end{theorem}
\begin{remark}
\citet{kurtz/xiong:1999} considers an infinite system of stochastic differential equations
for the locations and weights of a collection of particles that interact through their 
weighted empirical measure, which is shown to be the unique solution to a stochastic
	PDE. As we shall see through our lookdown representation in 
Section~\ref{sec:lookdown}, the
	solution to our
	deterministic equation can be seen as the empirical measure of a 
	countable number of particles (all with the same weight)
which, in the notation above, evolve according
	to 
	\[X(t)=X(0)+\int_0^t\beta\big(X(s), \eta_s\big)ds
+\int_0^tJ\big(X(s), \eta_s\big)dW(s)\]
	(with an independent Brownian motion $W$ for each particle).
\end{remark}


\paragraph{Two-step convergence to PDE:}
Although the coefficients at $x$ in \eqref{eqn:nonlocal_pde} are nonlocal,
we can choose our kernels $\rho_\gamma$, $\rho_r$, and $\rho_F$
in such a way that they depend only on the population in a region close to $x$,
and so we expect that under rather general conditions
solutions of the nonlocal PDE will be close to the corresponding 
classical PDE.
The following propositions provide two concrete situations in which this is true.
In the first, the PDE is a reaction-diffusion equation, and
in the proof in Section~\ref{two-step convergence to FKPP} we borrow
an idea from~\cite{penington:2017} to express the solutions to both the nonlocal
equation and the classical PDE through a Feynman-Kac formula. 

\begin{proposition}
    \label{prop:nonlocal_to_local}
Let $\rho^\epsilon_F(x)=\rho_F\big(x/\epsilon)/\epsilon^d$.
Assume $\varphi_0\in L^2(\IR^d)$ is a positive, uniformly Lipschitz, and uniformly bounded function. 
Suppose that $\varphi^\epsilon\in L^2(\IR^d)$ is a weak solution to the equation
\begin{equation}
\label{nonlocalPDEv1} 
\partial_t \varphi^\epsilon = \DG^* \varphi^\epsilon + 
\varphi^\epsilon F(\rho_F^\epsilon*\varphi^\epsilon),  
\qquad x \in \mathbb{R}^d,\, t >0, 
\end{equation}
with initial condition $\varphi_0(\cdot)$,
and that $\varphi$ is a weak solution to the equation
 \begin{equation}
\label{localPDE} 
\partial_t \varphi = \DG^* \varphi + \varphi F(\varphi),  
\qquad x \in \mathbb{R}^d,\, t >0, 
\end{equation}
also with initial condition $\varphi_0(\cdot)$.
Suppose further that $F$ is a Lipschitz function which is bounded above,
and that $\meanq(x)$ and $\covq(x)$, the drift and covariance matrix of $\DG$,
satisfy the conditions of Assumptions~\ref{def:model_setup}
and are such that $\DG^* 1 = 0$ (see Definition~\ref{def:dispersal_generator}).
Then, for all $T>0$ there exists a constant $K=K(T, \Vert \varphi_0 \Vert_\infty) < \infty$
and a function $\delta(\epsilon)$ (dependent on $\rho_F$) with $\delta(\epsilon)\to 0$ 
as $\epsilon\to 0$,
such that, for all $0 \leq t \leq T$, and $\epsilon$ small enough,
\[ 
\Vert \varphi_t(\cdot) - \varphi_t^\epsilon(\cdot) \Vert_\infty\leq K\delta(\epsilon). 
\]
In particular, as $\epsilon\to 0$, we have that $\varphi^\epsilon$ converges 
uniformly in compact intervals of time to $\varphi$.
\end{proposition}

\begin{remark}
Note that Theorem~\ref{thm:kurtzxiong} guarantees uniqueness of solutions
to equation~(\ref{nonlocalPDEv1}). 
\end{remark}

\begin{remark}
    Instead of putting the fairly strong constraint that $\DG^* 1(x) = 0$ for all $x$,
    it would be enough to assume instead that $\DG^* 1(x)$ is uniformly bounded,
    so that $f \mapsto \DG^*f - f \DG^* 1$ is the generator of a conservative diffusion.
    Since then $\varphi$ solves
    $$ \partial_t \varphi = (\DG^* \varphi - \varphi \DG^* 1) + \varphi (F(\varphi) + \DG^* 1) , $$
    the proof goes through essentially unchanged,
    with only $F(\varphi(x))$ replaced with $F(\varphi(x)) + \DG^*1(x)$,
    a bounded perturbation.
\end{remark}

Our second example 
in which we know solutions to the nonlocal PDE converge to
solutions of the local PDE as interaction distances go to zero
is a nonlocal version of a porous medium equation with logistic growth.
That is, we consider non-negative solutions to the equation
\begin{equation}
\label{mollified equation}
\partial_t \psi^\epsilon=
\Delta\left(\psi^\epsilon \, \rho^{\epsilon}_\gamma*\psi^\epsilon\right)
+\psi^\epsilon\left(1-\rho^{\epsilon}_\gamma*\psi^\epsilon\right).
\end{equation}
The case without the reaction term (and with $\IR^d$ replaced by a torus) is considered 
by~\cite{lions/mas-gallic:2001} who use
it as a basis for a particle method for numerical solution of the porous medium equation.
Of course this does not quite fit into our framework, since in the notation of our 
population models this would necessitate $\gamma(x,m)=\smooth{\epsilon}m$ which is
not bounded. However, this can be overcome by an additional layer of approximation
(c.f.~our numerical experiments of Section~\ref{beyond linear diffusion})
and we do not allow this to detain us here. Existence and uniqueness of solutions 
to~(\ref{mollified equation}) can be obtained using the approach
of~\cite{lions/mas-gallic:2001}, so
we should like to prove that as $\epsilon\to 0$ we have
convergence to the solution to the porous medium equation with
logistic growth:
\begin{equation}
\label{PME}
\partial_t\psi=
\Delta\left(\psi^2\right)
+\psi\left(1-\psi\right).
\end{equation}

\begin{notation}
We use $\rightharpoonup$ to denote weak convergence in the sense of analysts;
that is, $\psi^\epsilon\rightharpoonup \psi$ in $L^1$ means 
$\int \psi^\epsilon v d x
\rightarrow\int \psi v d x$ for all $v\in L^\infty$.

We write $L_t^2(H^1)$ for functions for which the $H^1$ norm in
space is in $L^2$ with respect to time, i.e.
$$\int_0^T\int \left\{\psi_t(x)^2+\left\| \nabla \psi_t(x)\right\|^2\right\}
dx dt <\infty,$$
and $C_t(L^1)$
will denote functions for which the $L^1$ norm in space is continuous in time.
\end{notation}

\begin{proposition}
	\label{nonlocalPME to PME}
Suppose that
we can write $\rho_\gamma=\zeta*\check{\zeta}$, where $\check{\zeta}(x)=\zeta(-x)$ and
$\zeta\in\mathcal{S}(\IR^d)$ (the Schwartz space of rapidly decreasing functions).
Furthermore, suppose that 
$\psi_0^\epsilon\geq 0$ is such that
there exists $\lambda\in (0,1)$ for which
\begin{equation*}
\sup_\epsilon \int\exp(\lambda \|x\|)\psi_0^\epsilon(x)dx < \infty,\quad\mbox{ and }
\sup_\epsilon\int \psi_0^\epsilon|\log \psi_0^\epsilon|dx < \infty,
\end{equation*}
with
$\psi_0^\epsilon\rightharpoonup \psi_0$ as $\epsilon\to 0$. Then
writing $\psi^\epsilon$ for the solution to~(\ref{mollified equation})
on $[0,T]\times \IR^d$ with
initial condition $\psi_0^\epsilon$,
$\psi^\epsilon\rightharpoonup \psi$ as $\epsilon\to 0$ where
$\psi\in L_t^2(H^1)\cap C_t(L^1)$, $\int \psi|\log \psi| dx<\infty$, and
$\psi$ solves~(\ref{PME}) on $[0,T]\times \IR^d$. 
\end{proposition}

The example that we have in mind for the kernel $\rho_\gamma$ is a Gaussian kernel.
For the proof, see Section~\ref{sec:pme}.

\begin{remark} \label{remark_on_nonlocal_to_local}
Although it seems hard to formulate an all-encompassing result,
Propositions~\ref{prop:nonlocal_to_local} and~\ref{nonlocalPME to PME} are by no
means exhaustive. When the scaling limit is deterministic, one can 
expect analogous results under rather general conditions. However, when the limit
points are stochastic, they resemble ``nonlinear superprocesses'' and so one cannot
expect a density with respect to Lebesgue measure in $d\geq 2$. It is then not
reasonable to 
expect to be able to make sense of the limit if we scale the kernels in this way.
Moreover, in one dimension, where the classical superprocess does have a density with 
respect to Lebesgue measure, the form of~(\ref{eqn:limiting_mgale_problem})
suggests that even if one can remove the local averaging from $\gamma$, it
will be necessary \revpoint{1}{14}
to retain averaging of $r$ in order to obtain a well-defined limit 
(since otherwise the term $\DG(f(\cdot)r(\cdot,\eta))(x)$
may not be well-defined).
\end{remark}

\paragraph{One-step convergence to PDE:}
Theorem \ref{thm:nonlocal_convergence},
combined with Proposition \ref{prop:nonlocal_to_local} or~\ref{nonlocalPME to PME}
implies that we can take the limit $N \to \infty$
followed by the limit $\epsilon \to 0$
to obtain solutions to the PDE \eqref{localPDE}.
However, it is of substantial interest to know whether
we can take those two limits simultaneously.
The general case seems difficult,
but we prove such ``diagonal'' convergence in the following situation.
The proof is provided in Section~\ref{subsec:one step convergence proof}.

\begin{theorem}[Convergence to a PDE]
    \label{thm:local_convergence}
    Let $(\eta^N_t)_{t \geq 0}$
    be as defined in Definition~\ref{defn:mgale_construction} with
$r(x,m)\equiv 1\equiv \gamma(r,m)$, $F(x,m)\equiv F(m)$, 
$\rho^\epsilon_F$ a symmetric Gaussian density with variance
parameter $\epsilon^2$,
and $\DG=\Delta/2$. 
Further suppose that $F(m)$ is a polynomial with $F(m)\ind_{m\geq 0}$ bounded above.
Assume that $\langle 1,\eta_0^N\rangle$ is uniformly bounded, and that for all 
$x\in\IR^d$ and $k\in \IN$,
	\[
		\limsup_{\epsilon\to 0} 
		\IE\big[\rho_F^\epsilon*\eta_0(x)^k\big]<\infty,
	\]
	and 
\[
    \limsup_{\epsilon \to 0} \int \IE\big[\rho_F^\epsilon*\eta_0(x)^k\big] dx<\infty.
\]
Finally assume that $N \to \infty$, $\theta \to \infty$
and $\epsilon \to 0$ in such a way that
\begin{equation}
\frac{1}{\theta\epsilon^2}+\frac{\theta}{N\epsilon^d}\to 0.
\end{equation}
Then the sequence of 
    ${\mathcal D}_{[0,\infty)}(\measures)$-valued
stochastic processes $\big(\rho_F^\epsilon*\eta^N_t(x)dx\big)_{t \ge 0}$
converges weakly
to a measure-valued process with a density $\varphi(t, x)$
that solves
\begin{align}
	\label{target equation in local convergence}
        \partial_t \varphi(t,x) = \frac{1}{2} \Delta \varphi(t,x) +\varphi(t,x) F(\varphi(t,x)).
\end{align}
\end{theorem}

\begin{remark}
In fact,
our proof goes through without significant change under the conditions that
$F(m)\ind_{m\geq 0}$ is bounded above (but not necessarily below), and that
for all $m,n\in[0,\infty)$
\[
|F(m)|\leq\sum_{j=1}^ka_jm^j, \quad\mbox{ and }
|F(n)-F(m)|
\leq|n-m|\sum_{j=1}^{k'}b_j\Big(n^j+m^j\Big),
\]
for some non-negative constants $\{a_j\}_{j=0}^k$, $\{b_j\}_{j=0}^{k'}$.
We take $F$ to be polynomial to somewhat simplify notation in the proof.
\end{remark}

\subsection{Ancestral lineages in the scaling limit}
\label{sec:ancestral lineages}

Now that we have established what we can say about how 
population density changes with time,
we turn to results on ancestral lineages,
i.e., how genealogical ancestry can be traced back across the landscape.
Informally,
a \emph{lineage} $(L_t^N)_{t \ge 0}$,
begun at a spatial location $L_0^N = x$
where there is a focal individual in the present day,
can be obtained for each time $t$ by
setting $L_t^N$ to be the spatial location of the individual alive at time $t$
before the present from
whom the focal individual is descended.
Since in our model individuals have only one parent, this is unambiguous.
Although we did not explicitly retain such information,
it is clear that for finite $N$, since individuals are born one at a time, 
one could construct the lineage $(L_t^N)_{t=0}^T$
given the history of the population $(\eta^N_t)_{t = 0}^T$,
for each starting location to which $\eta^N_T$ assigns positive mass.
It is less clear, however, how to rigorously retain such information when we 
pass to the scaling limit.

However, the \emph{lookdown construction} in Section~\ref{sec:lookdown}
does just this -- \llabel{whatis_lineage}
the construction enables us to recover information about ancestry
in the infinite population limit,
and thus gives a concrete meaning to $(L_t)_{t \ge 0}$. \llabel{Linf_mention}
Roughly speaking,
each particle is assigned a unique ``level'' from $[0,\infty)$
that functions as a label and thus allows reconstruction of lineages.
The key to the approach is that levels are assigned in such a way as to be exchangeable, 
so that sampling a finite number, $k$ say, of
individuals from a given 
region is equivalent to looking at the individuals in that region with the $k$ lowest levels.
Moreover, as we pass to the infinite population limit,
the collection of (individual, level) pairs converges,
as we show in Theorem~\ref{thm:lookdown_convergence}.
See~\citet{etheridge/kurtz:2019} for an introduction to these ideas.
In particular, even in the infinite
population limit, we can sample an individual
from a region (it will be the individual in that region with the lowest level) 
and trace its line of descent. 
This will allow us to calculate, for each $x$ and $y\in\IR^d$, the proportion of
the population at location $x$ in the present day population that is descended from 
a parent who was at location $y$ at time $t$ in the past. To make
sense of this in our framework, in Section~\ref{sec:lineages_proof}, 
we justify a weak reformulation of this idea.

We are interested in two questions about the limiting process.
First, when is the motion of an ancestral 
lineage, given complete knowledge of the population process, a well-defined process?
In other words, is knowledge of the process $(\eta_t)_{t=0}^T$ that records 
numbers of individuals
but not their ancestry sufficient to define the distribution of $(L_t)_{t=0}^T$? \revpoint{1}{15}
Second, does the process have a tractable description?

We focus on the simplest situation, that in which the population process is deterministic.
However, the results here apply when the population process solves 
either a nonlocal or a classical PDE.
There will be no coalescence of ancestral lineages in the deterministic limit,
but understanding motion of single lineages is useful in practice,
and our results can
be seen as a first step towards understanding genealogies for high population densities. 
Since the time scale on which coalescence occurs \revpoint{1}{17}
goes to infinity in the deterministic limit,
an important question to answer will be whether this description of lineage motion
is a good approximation over such a long time scale.
Other information may be important -- for instance, in \citet{etheridge/pennington:2022}
the form of the coalescent that is obtained depends
on fluctuations happening on a longer time scale than the mixing time of a lineage.

Proofs of results in this section are found in Section~\ref{sec:lookdown_proofs}.

\begin{definition}[Ancestral lineage] 
\label{def:lineage_generator}
    Let $(\varphi_t(x))_{0 \le t \le T}$
    denote the density of the scaling limit of our population model, 
solving \eqref{eqn:nonlocal_pde},
    let $y$ be a point with $\varphi_T(y) > 0$,
    and suppose we have sampled an individual from location $y$ at time $T$.
    We define $(L_s)_{s=0}^T$,
    the ancestral lineage of that sampled individual by setting $L_0=y$
    and $L_s$ to be the position of the unique ancestor of that individual at time $T - s$.
    We define
    $(Q_s)_{s \geq 0}$
    to be the time inhomogeneous semigroup satisfying
    \begin{align*}
        Q_s f(y) := \IE_y[ f(L_s) ] .
    \end{align*}
\end{definition}

The precise sense in which we can look at ``the lineage of a sampled individual'' \revpoint{1}{2}
in the scaling limit is made clear by the introduction of the lookdown construction,
in Section~\ref{sec:lookdown}.
(For now, we can take the definition to refer to the distribution obtained
from a scaling limit of the finite-$N$ process.)
It turns out that in the scaling limit,
the process is Markovian,
and our next result identifies the ancestral lineage as a diffusion
by characterizing its generator.

\begin{theorem} \label{thm:lineages}
    For $\varphi: \IR^d \to \IR$, define
    \begin{align}
        \label{eqn:lineage_generator}
        \Lgen_\varphi f
        &=
        \frac{r}{\varphi}
        \left[
            \DG^*(\gamma \varphi f) 
            - f \DG^* (\gamma \varphi)
        \right] \\
        &= \label{eqn:lineage_generator2}
        r\gamma
        \left[
            \frac{1}{2} \sum_{ij} \covq_{ij} \partial_{x_ix_j} f
            + \sum_j \vec{m}_j \partial_{x_j} f
        \right] ,
    \end{align}
    where $\vec{m}$ is the vector
    $$
    \vec{m}_j
    =
    \sum_i C_{ij} \partial_{x_i} \log(\gamma \varphi)
    + \sum_i \partial_{x_i} C_{ij}
    - \meanq_j .
    $$
    Then the generator of the semigroup $Q_s$
    of Definition~\ref{def:lineage_generator} is given by
	$\partial_sQ_sf(y)= \Lgen_{\varphi_{T-s}}Q_sf(y).$ 
\end{theorem}

\begin{remark}
As usual,
to make the generator readable, we've written it in concise notation,
omitting the dependencies on location and population density,
which itself changes with time.
When interpreting this,
remember that everything depends on location and density at that location and time --
for instance, ``$r$'' is actually $r(x, \varphi(x))$ (in the classical case),
or $r(x, \smooth{r} \eta(x))$ (in the nonlocal case).

Moreover, we haven't proved any regularity of the population density process $\varphi$,
so, as written, the generator~(\ref{eqn:lineage_generator})
may not make sense. Instead, it should be interpreted in a weak sense which is made precise
in Section~\ref{sec:lineages_proof}.
\end{remark}

\begin{corollary} \label{cor:lineages_simple}
    In addition to the assumptions of Theorem~\ref{thm:lineages},
    if the covariance of the dispersal process is isotropic
    (i.e., $\covq = \sigma^2 I$),
    then
    \begin{equation}
        \Lgen_\varphi f
        =
        \frac{\sigma^2}{2}
        r \gamma
        \left(
            \Delta f
            +
            \left(
                2 \grad \log(\gamma \varphi)
                - \frac{2 \meanq}{\sigma^2}
            \right)
            \cdot \grad f
        \right) .
    \end{equation}
    (However, $\meanq$ can still depend on location.)
\end{corollary}

In other words, 
the lineage behaves as a diffusion driven by Brownian motion 
run at speed $\sigma^2$ multiplied by the local per-capita production of 
mature offspring ($r \gamma$)
with mean displacement in the direction of $\grad \log(\varphi_s \gamma) - \meanq/\sigma^2$.
In particular, lineages are drawn to regions of high fecundity (production of juveniles), but their
speed is determined by the rate of production of mature offspring.
This can be compared to Remark~\ref{ancestral lineages: first guess}.

\begin{corollary} \label{cor:stationary_dist}
    In addition to the assumptions of Corollary~\ref{cor:lineages_simple},
    if the population process is stationary (so $\varphi_t \equiv \varphi$),
    and $\meanq(x) = \grad h(x)$ for some function $h$,
    then $Y$ is reversible with respect to
    \begin{align}
        \pi(x)
        =
        \frac{\gamma}{r} \varphi(x)^2 e^{-2 h(x) / \sigma^2} .
    \end{align}
\end{corollary}

The long-term reproductive value of an individual
is proportional to the fraction of lineages from the distant future that pass through the individual,
and hence the total long-term reproductive value at a location is proportional to
the stationary distribution of $Y$ there, if it exists.
Therefore,
if $\pi$ is integrable then the per-capita long-term reproductive value of an individual at $x$
is proportional to $\pi(x) / \varphi(x)$.

\begin{corollary} \label{cor:wavefront}
    In addition to the assumptions of Corollary~\ref{cor:lineages_simple},
    suppose that the population process is described by a travelling wave with 
velocity $\wavespeed$,
    i.e., the population has density
    $\varphi(t, x) = w(x - t \wavespeed)$
    where $w$ solves
    \begin{align*}
        r \DG^* (\gamma w) + w F + \wavespeed \cdot \grad w = 
        \frac{1}{2} r\sigma^2\Delta (\gamma w) -\meanq \cdot\grad (\gamma w) +\wavespeed \cdot \grad w =0 .
    \end{align*}
    Then the semigroup $Q_s$
    of the motion of a lineage in the frame that is moving at speed $\wavespeed$
    is time-homogeneous with generator
    \begin{align}
\label{Lgen}
        \Lgen f
        &=
        \frac{1}{2} \sigma^2 r \gamma
        \left(
            \Delta f
            +
            2 \grad \log (\gamma w)
            \cdot \grad f
        \right)
        + (\wavespeed - r\gamma \meanq) \cdot \grad f .
    \end{align}
\end{corollary}

\section{Examples and applications}
\label{sec:applications}

We now discuss some consequences of these results.

\subsection{Beyond linear diffusion}
\label{beyond linear diffusion}

Equation~(\ref{eqn:nonlocal_pde}) is a nonlocal version of a reaction-diffusion equation;
the diffusion is nonlinear if $\gamma$ depends on population density:
in other words, if the diffusivity of the population depends on the population density.
Passing to the classical limit, we recover equations like~(\ref{PME}).
Such equations are widely used in a number of contexts in biology in which
motility within a population varies with population density.
For example, density dependent dispersal is a common feature in spatial
models in ecology, eukaryotic cell biology, and avascular tumour growth;
see~\cite{sherratt:2010} and references therein for further discussion. 
In particular, such equations have been suggested as a model for
the expansion of a certain type of bacteria 
on a thin layer of agar in a Petri dish 
\citep{cohen/golding/kozlovsky/benjacob/ron:1999}. 
We shall pay particular attention to the case in which the equation can be 
thought of as modelling the density of an expanding population. 
We focus on the monostable reaction of~(\ref{PME}).

Comparing with \eqref{eqn:nonlocal_pde},
we see that to set up a limit in which the population density $\varphi$ follows the 
porous medium equation with logistic growth of~(\ref{PME}),
we need $r=1$,
$\gamma = \varphi$, and $F = 1 - \varphi$.
Consulting equation~\eqref{eqn:mu_defn},
this implies that $\mu_\theta = (1 + 1/\theta) \varphi - 1/\theta$.
In other words,
establishment is certain
and birth rates increase linearly with population density,
but to compensate, death rates increase slightly faster (also linearly).
Alert readers will notice that
the condition from Assumptions~\ref{def:model_setup} 
that $\gamma(x,m)$ be uniformly bounded 
is violated.
This can be corrected by use of a cut-off, and in fact the downwards drift
provided by the logistic control of the population size prevents $m$ from 
getting too big. 
In practice the simulations shown in Figure~\ref{fig:pme_waves}
take discrete time steps of length $dt$ (with $dt$ suitably small),
and have each individual reproduce and die with probabilities, respectively,
\begin{align*}
    p_\text{birth}(m) = \left(1 - e^{- m dt}\right)
    \qquad 
    p_\text{death}(m) = \left(1 - e^{- (m (1+1/\theta) - 1/\theta) dt} \right) ,
\end{align*}
where $m$ is the local density at their location.
This makes $\gamma(x, m) = p_\text{birth}(m)/dt \approx m$
and
\begin{align*}
    F(x, m) = \theta(\gamma(x, m) - \mu(x, m)) / dt \approx 1 - m .
\end{align*}
Birth and death rates are equal at density $m = 1$,
corresponding to an unscaled density of $N$ individuals per unit area.

\begin{figure}
    \begin{center}
        \includegraphics[width=4.5in]{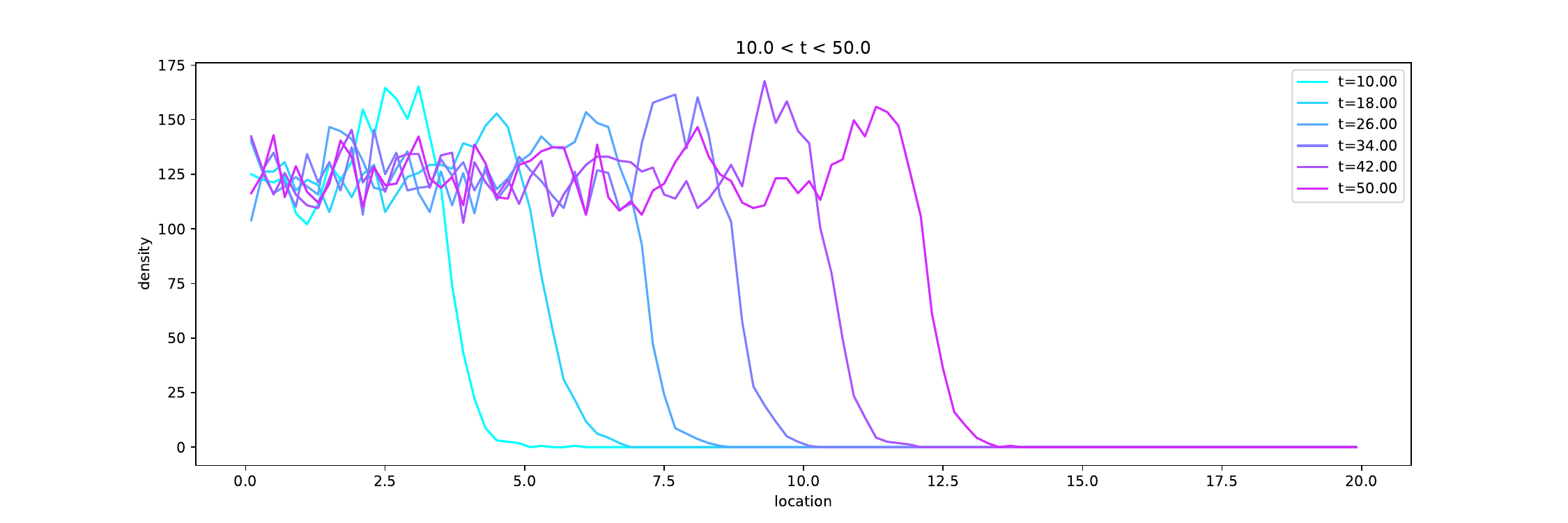}
        \includegraphics[width=4.5in]{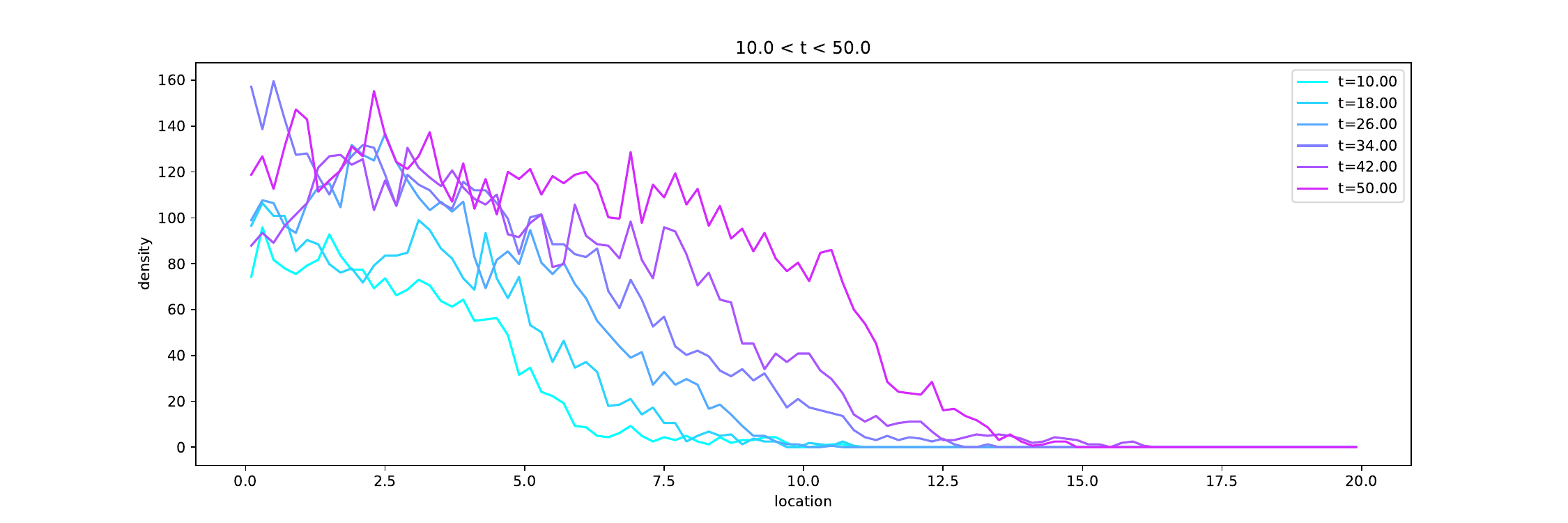}
    \end{center}
    \caption{
        Simulated populations under a porous medium equation with logistic growth~(\ref{PME}) in $d=1$,
        $\theta/N$ small on the top; large on the bottom.
        Values of $\theta$ in top and bottom figures
        are 1 and 100, respectively, and both have $N$ set
        so that the density is roughly 100 individuals per unit of habitat
        (as displayed on the vertical axis).
        See text for details of the simulations.
        \label{fig:pme_waves}
    }
\end{figure}

In one dimension, equation~(\ref{PME}) has an explicit travelling wave solution
\begin{align} \label{eqn:pme_wave}
    w^P(t, x)
    :=
    \left( 1 - e^{ \frac{1}{2} (x - x_0 - t) } \right)_+ .
\end{align}
Notice that the wave profile has a sharp boundary at $x = x_0 + t$.
There are also travelling wave solutions with $c>1$ \citep{gilding/kersner:2005},
which lack this property.
However, for initial conditions that decay sufficiently rapidly at infinity,
such as one might use in modelling a population invading new territory,
the solution converges to \eqref{eqn:pme_wave} \citep{kamin/rosenau:2004}.
In Figure~\ref{fig:pme_waves} we show simulations of the individual based model described above,
which display travelling wave solutions qualitatively similar to solutions of~(\ref{PME}),
with better agreement for smaller $\theta/N$
(but in both cases, $N$ is reasonably large).

\subsection{Ancestry in different types of travelling waves}
\label{ancestral lineages for travelling waves}


Although it remains challenging to establish the distribution of genealogical
trees relating individuals sampled from our population model, as described in
the introduction, we can gain some insight by investigating the motion of a
single ancestral lineage. Here we do that in the context of a one-dimensional
population expanding into new territory as a travelling wave. 
We focus on three cases in which we have 
explicit information about the shape of the travelling wave profile:
the Fisher-KPP equation,
a special case of the Allen-Cahn equation with a bistable nonlinearity,
and the porous media equation with logistic growth, equation~(\ref{PME}).
We work here in one dimension,
and take $\sigma^2 = 2$ and $\meanq = 0$.

Ancestry in travelling waves for populations described by reaction-diffusion equations
has been studied before by various authors, \revpoint{1}{16}
although most work assumes the diffusion term is linear
(in our notation, $r$ and $\gamma$ are constant,
but $F$ can depend on density).
For instance, following \citet{hallatschek/nelson:2008},
various authors (e.g., \citet{roques2012allee,birzu/hallatschek/korolev:2017})
describe genetic diversity
and ancestry in (possibly noisy) travelling waves for reaction-diffusion equations
in a situation that covers our first two examples below.
However, we are not aware of previous work covering the third case 
with nonlinear diffusive term (the porous medium equation).

\paragraph{Fisher-KPP equation:}
Consider the classical Fisher-KPP equation,
\begin{align} \label{eqn:fkpp}
    \partial_t\varphi = \partial_{xx}\varphi + \varphi (1-\varphi) .
\end{align}
Even though we do not have an explicit formula for the wave shape in this case,
our methods provide information about ancestral lineages.
The equation has non-negative travelling wave solutions of 
speed $c$ for all $c \geq 2$, 
but, started from any compact perturbation of a Heaviside function, the 
solution will converge to the profile $w^F$ with the minimal wavespeed, $c=2$
\citep{kolmogorov/petrovsky/piscounov:1937,fife1977approach,bramson:1983}.
No matter what initial condition,
for any $t>0$ the support of the 
solution will be the whole real line. 
In this case, we must have $r = \gamma = 1$,
and $F(x, m) = 1 - m$ so $\mu_\theta(x, m) = 1 + (m-1)/\theta$.
By Corollary~\ref{cor:wavefront},
the generator of the motion of an ancestral lineage is
\begin{equation} \label{eqn:fkpp_generator}
    \Lgen_F  f
    =
    \partial_{xx}f + 2 \frac{\partial_xw^F}{w^F} \partial_x f + 2 \partial_xf .
\end{equation}
Near the tip of the wave (for $x$ large), $w^F(x) \sim e^{-x}$,
so \eqref{eqn:fkpp_generator} implies that
the motion of a lineage is close to unbiased Brownian motion.
On the other hand, in the ``bulk'', 
a lineage behaves approximately as
Brownian motion with drift at rate two to the right.
This implies that
ancestral lineages are pushed into the tip of the wave,
and there is no stationary distribution,
so that long-term dynamics of genetic inheritance
depend on the part of the wave not well-approximated by a smooth profile,
in agreement with the previous results referred to in the Introduction.

\paragraph{Allen-Cahn equation:}
Now take the Allen-Cahn equation:
\begin{align} \label{eqn:allen_cahn}
    \partial_t\varphi = \partial_{xx}\varphi + \varphi(1-\varphi)(2\varphi-1+s),
\end{align}
for a given $s \in (0,2)$.
Once again we have taken $r=\gamma=1$, but now the reaction term
$F(x,m) = (1-m) (2m-1+s)$ is bistable.
This equation can be used to model the motion of so-called 
hybrid zones in population genetics; see, for example,
\citet{barton:1979}, \citet{gooding:2018}, and \citet{etheridge/gooding/letter:2022}.
This equation has an explicit travelling wave solution with speed $s$
and shape
\[ w^A(x) = (1+e^{x})^{-1} , \]
i.e., $\phi_t(x)=w^A(x-st)$ solves~(\ref{eqn:allen_cahn}).
Substituting $w^A$ in place of $w^F$ in~(\ref{eqn:fkpp_generator}),
we find that the generator of an ancestral lineage relative to the wavefront is now, 
\begin{align*}
\Lgen_A f
    &=
    \partial_{xx}f
    + 
    2 \frac{\partial_xw^A}{w^A} \partial_xf
    +
    s \partial_x f\\
    \qquad &=
    \partial_{xx}f
    -
    2 \frac{e^x}{1+e^x} \partial_xf 
    + 
    s \partial_xf,
\end{align*}
so lineages in the tip are pushed leftwards into the bulk of the 
wave at a rate $s-2e^x/(1+e^{x})$.
The density of the speed measure for this diffusion is
$$
    m_A(x) \propto e^{sx}(1+e^x)^{-2},
$$
which is integrable, and so determines the unique stationary distribution.
Thus the position of the ancestral lineage relative to the wavefront will converge to a stationary 
distribution which is maximised away from the extreme tip of the wave.
This is consistent with~\cite{etheridge/penington:2022}, who consider an analogous
stochastic population model, although the stronger result there (that the genealogy
of a sample from behind the wavefront is approximately a Kingman coalecsent) requires
the stronger condition $s<1$.


\paragraph{Porous Medium equation with logistic growth:}
Finally, consider equation~(\ref{PME}). Setting $x_0=0$ (for definiteness) and substituting
the form of $w^P$ from equation~\eqref{eqn:pme_wave}
into Corollary~\ref{cor:wavefront},
with $c=1$,
$\gamma(x, m) = m$,
$r(x,w) = 1$,
and $F(x, m) = (1 - m)$,
the generator of the diffusion governing the position of the ancestral lineage relative to the wavefront
is, for $x < 0$,
\begin{align*}
    \Lgen_P f
    &=
        w^P
        \left(
        \partial_{xx} f
         +
         2 \frac{\partial_x((w^P)^2)}{(w^P)^2} \partial_xf
        \right)
        + \partial_xf \\
    &=
        \left(1 - e^{\frac{1}{2} x} \right)
        \partial_{xx}f
        -
        2 e^{\frac{1}{2} x} \partial_xf
        +
        \partial_xf .
\end{align*}
The speed measure corresponding to this diffusion has density
\begin{align*}
    m_P(\xi)
    &\propto
        \frac{ 1 }{ 2 (1 - e^{\xi/2}) }
        \exp\left(
            \int_\eta^\xi \left\{
                1 - \frac{e^{x/2}}{1 - e^{x/2}}
            \right\} dx
        \right) \\
    &\propto
        e^\xi\left(1-e^{\xi/2}\right),
        \quad
        \text{for } \xi < 0 
\end{align*}
and $m_P(\xi) = 0$ for $\xi \ge 0$,
which is integrable and so when suitably normalised gives the unique stationary distribution.
Notice that even though we have the same reaction term as in the Fisher-KPP equation, with this 
form of nonlinear diffusion, at stationarity
the lineage will typically be significantly behind the front, 
suggesting a different genealogy.

\subsection{Clumping from nonlocal interactions}

Simulating these processes and exploring parameter space,
one sooner or later comes upon a strange observation:
with certain parameter combinations, the population spontaneously forms a regular grid
of stable, more or less discrete patches, separated by areas with nearly no individuals,
as shown in Figure~\ref{fig:clumping}.
The phenomenon is discussed in Section~16.15 of \citet{haller2022evolutionary},
and has been described in similar models, e.g.,
by \citet{britton1990spatial,sasaki1997clumped,hernandezgarcia2004clustering,young2001reproductive}, and \citet{berestycki2009nonlocal}.
For example, if the density-dependent effects of individuals extend farther (but not too much farther)
than the typical dispersal distance,
then depending on the interaction kernel
new offspring landing between two clumps
can effectively find themselves in competition with 
\emph{both} neighbouring clumps,
while individuals within a clump compete with only one.

More mathematically, consider the case in which
$\DG = \sigma^2 \Delta$
(so that dispersal variance is $2 \sigma^2$)
and all parameters are spatially homogeneous,
so that $r(x, \eta) = r(\smooth{r}\eta(x))$, and similarly for $\gamma$ and $F$.
If $\varphi_0$ is such that $F(\varphi_0)=0$ and $F'(\varphi_0)<0$, then the
constant solution $\varphi\equiv\varphi_0$ is a nontrivial equilibrium
of~\eqref{general deterministic limit}.
However, this constant solution may not be unique, it may be unstable,
and a stable solution may have oscillations on a scale determined by the interaction distance.

To understand the stability of the constant solution $\varphi\equiv\varphi_0$, 
we linearise~\eqref{general deterministic limit} around $\varphi_0$:
let $\varphi_t(x) = \varphi_0 + \psi_t(x)$,
and (informally) $r(x) \approx r(\varphi_0) + r'(\varphi_0) \smooth{r} \psi(x)$.
Recall that in this section we are in $d=1$.
Writing $r_0 = r(\varphi_0)$ and $r'_0 = r'(\varphi_0)$,
with analogous expressions for $\gamma$ and $F$,
\begin{align*}
    \partial_t \psi
    &\approx
    \sigma^2 \varphi_0 r_0 \gamma_0' \Delta \smooth{\gamma} \psi
    + \sigma^2 r_0 \gamma_0 \Delta \psi
    + \varphi_0 F'_0 \smooth{F} \psi .
\end{align*}
Letting $\hat f(u) = \int e^{i u x} f(x) dx / \sqrt{2 \pi}$ denote the Fourier transform,
\begin{equation}
	\label{fourier transform psi}
    \partial_t \hat \psi(u)
    \approx
	\left\{
	- u^2\sigma^2 \varphi_0 r_0 \gamma_0' \hat{\rho}_\gamma(u) 
        - u^2 \sigma^2 r_0 \gamma_0
	+ \varphi_0 F'_0 \hat{\rho}_F(u) 
	\right\} \hat \psi(u) .
\end{equation}
In the simplest case, in which $\gamma$ is constant, so $\gamma_0'=0$,
this reduces to
\begin{equation} 
	\label{simplest FT}
	\partial_t \hat \psi(u)
    \approx
    \left(
        - u^2 \sigma^2 r_0 \gamma_0
	+ \varphi_0 F'_0 \hat{\rho}_F(u) 
    \right) \hat \psi(u) .
\end{equation} 
If we take $\rho_F=p_{\epsilon^2}$, then 
$\hat{\rho}_F(u)=\exp(-\epsilon^2 u^2 / 2) / \sqrt{2 \pi}$ and (recalling 
that $F_0'<0$) the term in brackets is always negative, and we recover
the well-known fact that in this case the constant solution is stable.
If, on the other hand, $\hat{\rho}_F$ changes sign, there may be values of $u$ for which the corresponding quantity is positive. For example, if $d=1$ and
$\rho_F(x)={\mathbf 1}_{[-\epsilon, \epsilon]}(x)/2\epsilon$, then 
$\hat{\rho}_F(u)=\sin(\epsilon u)/(\sqrt{2\pi}\epsilon u)$, which is 
negative for $u\in (\pi/\epsilon, 2\pi/\epsilon)$ (and periodically 
repeating intervals). Setting $v=\epsilon u$, the bracketed term on 
the right hand side of~\eqref{simplest FT} becomes 
\[
    \varphi_0F_0'\frac{1}{\sqrt{2\pi} v} \sin(v)
	-\frac{\sigma^2}{\epsilon^2}v^2r_0\gamma_0,
\]
and we see that if $\sigma^2/\epsilon^2$ is sufficiently small, there are
values of $v$ for which this is positive. In other words, in keeping
with our heuristic above, if dispersal is sufficiently short range relative to
the range over which individuals interact, there are unstable frequencies that
scale with the interaction distance $\epsilon$. In two dimensions, replacing 
the indicator of an interval by that of a ball of radius $\epsilon$, a
similar analysis applies, except that the sine function is replaced by a 
Bessel function.

Now suppose that $\gamma$ is not constant. Then, 
from~\eqref{fourier transform psi},
if we take 
$\rho_\gamma=\rho_F=p_\epsilon^2$,
\begin{equation*}
	\partial_t \hat \psi(u) \approx
    \frac{1}{\sqrt{2\pi}}
    e^{-\epsilon^2 u^2/2} 
	\left\{
		-\sigma^2 \varphi_0 r_0 \gamma_0' u^2
        -\sigma^2 r_0 \gamma_0 u^2 \sqrt{2\pi} e^{\epsilon^2 u^2/2}
        +\varphi_0 F_0'
	\right\}
	\hat{\psi}(u).
\end{equation*}
If we make the (reasonable)
assumption that $\gamma_0'<0$, then we see that even when the Fourier
transform of $\rho$ does not change sign, there may be parameter values
for which the constant solution is unstable.
As before, we set $v=\epsilon u$. The term in brackets becomes
\[
	\frac{\sigma^2}{\epsilon^2}v^2 r_0
    \left(
        -\varphi_0 \gamma_0' - \gamma_0 \sqrt{2\pi} e^{v^2/2}
    \right) + \varphi_0F_0',
\] 
and, provided $-\varphi_0\gamma_0'/\gamma_0\sqrt{2\pi} >1$, for sufficiently small $v$ 
the term in round brackets is positive. We now see that if 
$\sigma^2/\epsilon^2$ is sufficiently {\em large},
the equilibrium state $\varphi\equiv\varphi_0$ 
is unstable. As before, the unstable frequencies will scale with $\epsilon$
and for given $F$, $r$ and $\gamma$, whether or not such unstable
frequencies exist will be determined by $\sigma^2/\epsilon^2$, but 
in this case of Gaussian kernels, 
it is interaction distance being sufficiently small relative
to dispersal that will lead to instability.

\begin{figure}
    \begin{center}
        \includegraphics[width=2.5in]{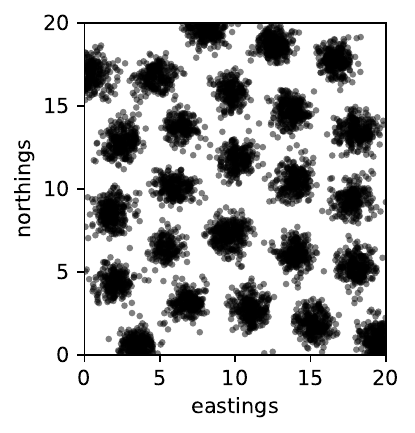}
        \includegraphics[width=2.5in]{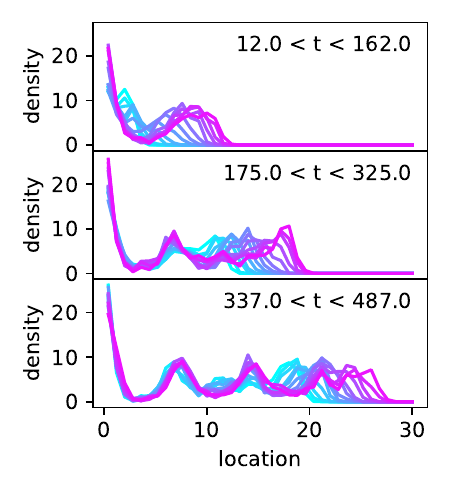}
    \end{center}
    \caption{
        \textbf{Left:} A snapshot of individual locations in a two-dimensional simulation
        in which the constant density is unstable
        and a stable, periodic pattern forms.
        \textbf{Right:} Population density in an expanding wave in a one-dimensional simulation
        forming a periodic pattern;
        each panel shows the wavefront in three periods of time;
        within each period of time the wavefront at earlier times is shown in blue
        and later times in pink.
        In both cases, $\gamma(m) = 3/(1 + m)$, $\mu \equiv 0.3$, and $r \equiv 1$;
        dispersal is Gaussian with $\sigma=0.2$ and density is measured with $\rho_\gamma(x) = p_9(x)$,
        i.e., using a Gaussian kernel with standard deviation 3.
        \label{fig:clumping}
    }
\end{figure}

\subsection{Lineage motion distinguishes different models with the same equilibrium density}
\label{sec:nonunique_lineage}

It is natural for applications to wonder about identifiability:
when can the observed quantities like population density
or certain summaries of lineage movement uniquely determine
the underlying demographic parameters?
Consider a deterministic,
continuous population generated by parameters $\gamma$, $r$, and $F$,
with $\meanq = 0$ and $\covq = 2I$.
Suppose it has a stationary profile $w(x)$, that must satisfy
$$
   r \Delta(\gamma w) + F w = 0 .
$$
It is easy to see that $w$ does not uniquely specify $\gamma$, $F$, and $r$:
let $\lambda(x)$ be a smooth, nonnegative function on $\IR^d$,
and let $\widetilde{r}(x, m) = \lambda(x) r(x, m)$ and $\widetilde{F}(x, m) = \lambda(x) F(x, m)$
(and, let $\widetilde{\gamma} = \gamma$).
Since $\mu = r \gamma - F/\theta$,
this corresponds to multiplying both establishment probabilities and death rates by $\lambda$.
Then the population with parameters $\widetilde{\gamma}$, $\widetilde{r}$, and $\widetilde{F}$
has the same stationary profile(s) as the original population.

Can these two situations be distinguished from summaries of lineage movement?
The first has lineage generator
\[
    f \mapsto \Lgen f = r \gamma \left( \Delta f + 2 \grad \log(\gamma w) \cdot \grad f \right),
\]
while the second has lineage generator $f \mapsto \lambda(x) \Lgen f(x)$.
In other words,
although the stationary profile of the population is unchanged when we scale
local establishment and death by $\lambda$,
the motion of lineages is sped up locally by $\lambda$.
This corresponds to making areas with $\lambda > 1$ more ``sink-like'' 
and $\lambda < 1$ more ``source-like'':
if $\lambda(x) > 1$, then at $x$ both the death rate and probability of 
establishment of new individuals are higher.
As a result, lineages in the second model spend more time in areas with $\lambda < 1$,
i.e., those areas have higher reproductive value,
something that is, in principle, discernible from genetic data
(because, for instance, making reproductive value less evenly distributed
reduces long-term genetic diversity \citep{ewens:1982}). \revpoint{1}{18}

\section{Heuristics}
    \label{sec:heuristics}

In this section we perform some preliminary calculations and use them 
to provide heuristic arguments for our main results,
to build intuition before the proofs.

\subsection{The population density}
    \label{sec:population_heuristics}

{\em We reiterate that in our prelimiting model,
the population is represented by a point measure $\eta^N$ in which each
individual is assigned a mass $1/N$. We use the term ``population density'' for
this process, as it is supposed to measure population size relative to a 
nominal occupancy of $N$ individuals per unit area, but it is not
absolutely continuous with respect to Lebesgue measure.}

We write $\Pgen^N$
for the generator of the scaled population process $\eta^N$
of Definition~\ref{defn:mgale_construction}
acting on test functions of the form $G( \langle f, \eta \rangle )$,
where $f \geq 0$ is smooth and bounded on $\IR^d$ and 
$G \in C^\infty ([0,\infty))$.
Recall that $\theta =\theta(N) \to \infty$ as $N\to\infty$ in such 
a way that $\theta(N)/N\to\alpha$.

A Taylor expansion allows us to write
\begin{multline}
	\label{generator prelimit}
    \Pgen^N
    G(\langle f,\eta \rangle)
    =
    G'(\langle f, \eta \rangle)
    \lim_{\delta t\downarrow 0} \frac{1}{\delta t}
    \IE\left[
        \left. \langle f, \eta_{\delta t} \rangle
        -
        \langle f, \eta \rangle
        \right| \eta_0=\eta
    \right]
    \\
    \qquad {}
    + \frac{1}{2}
        G''(\langle f,\eta\rangle)
    \lim_{\delta t\downarrow 0}\frac{1}{\delta t}
    \IE\left[
        \left.\big(\langle f,\eta_{\delta t}\rangle
        -
        \langle f, \eta\rangle\big)^2 \right|\eta_0=\eta
    \right]
    +
    \epsilon_N(f, G, \eta),
\end{multline}
where the terms that make up 
$\epsilon_N(f, G, \eta)$
will be negligible in our scaling limit (at least if $G'''<\infty$). 

\subsubsection*{Mean measure}

Recall that in our parameterization only death rates $\mu_\theta$
and the dispersal kernel $q_\theta$ depend on $\theta$.
For a suitable test function $f$, we find
\begin{equation} \label{mean measure}
    \begin{split}
    \Pgen^N \langle f, \eta \rangle
    &=
    \lim_{\delta t\downarrow 0} \frac{1}{\delta t}
    \IE\left[ \left.
        \langle f, \eta_{\delta t} \rangle
        -
        \langle f, \eta\rangle
        \right| \eta_0 = \eta
    \right]
    \\
    &=
    \theta \int
        \int f(z) r(z,\eta) q_\theta(x,dz)
    \gamma(x, \eta) \eta(dx)
    -
    \theta \int f(x)\mu_\theta(x, \eta)
    \eta(dx).
    \end{split}
\end{equation}
The first term is the increment in $\langle f,\eta\rangle$
resulting from a birth event (recalling that
we don't kill the parent) integrated against the rate of such events,
and the second reflects death events.
The factor of $\theta$ appears from the time rescaling.
In both terms,
the rate of events has a factor of $N$ (because events happen at a rate 
proportional to the number of individuals,
whereas $\eta$ has mass $1/N$ for each individual)
which is offset by the fact that  
the birth or loss of a single 
individual at the point $y$, say, changes $\langle f,\eta\rangle$
by $f(y)/N$.

We use the fact that $\int q_\theta(x,dz)=1$ to rewrite~(\ref{mean measure})
as 
\begin{equation}
\label{eqn:rewritten mean measure}
\begin{split}
    \int\left(
        \int \theta \left( f(z) r(z,\eta)- f(x) r(x,\eta) \right) q_\theta(x,dz)
    \right)
    \gamma(x,\eta)
    \eta(dx)
    \\
    + \int f(x) \theta \Big(
        r(x,\eta) \gamma(x,\eta)
        - \mu_\theta(x,\eta)
    \Big) \eta(dx).
\end{split}
\end{equation}
We have defined $\mu_\theta$ so that the second term is simple:
\begin{align*}
    \theta \Big( r(x,\eta) \gamma(x,\eta) - \mu_\theta(x,\eta) \Big)
    = F(x, \eta) .
\end{align*}
Furthermore, recall from Remark~\ref{rem:DG_limit} that
\begin{align} \label{eqn:near_critical}
    \int \theta \Big(
        r(z,\eta) f(z)
        -
        r(x,\eta) f(x)
    \Big) q_\theta(x,dz) 
    \qquad \stackrel{\theta\to\infty}{\longrightarrow} \qquad  
    \DG \big(r(\cdot,\eta)f(\cdot)\big)(x) .
\end{align}
In particular,
if dispersal is determined by a standard multivariate Gaussian
with mean zero and covariance $\sigma^2 I / \theta$,
then $\DG = \sigma^2 \Delta / 2$, where $\Delta$ denotes the Laplacian.

In summary, equation~\eqref{eqn:rewritten mean measure} converges to
\begin{equation} \label{limit of mean measure equation}
\int \gamma(x,\eta)
\DG \big(f(\cdot)r(\cdot,\eta)\big)(x)
\eta(dx)
+
\int f(x)
F(x,\eta)
\eta(dx) ,
\end{equation}
which explains the form of the martingale of Theorem~\ref{thm:nonlocal_convergence}.

\subsubsection*{Quadratic variation}


We now look at the second order term in~(\ref{generator prelimit}),
which will converge to the quadratic variation of the limiting process.
An individual at location $x$ gives birth 
to a surviving offspring at $y$ at rate
$$
\gamma(x,\eta) r(y,\eta) q_{\theta}(x, dy) ,
$$
and since this increments $\langle f, \eta \rangle$ by $f(y) / N$,
the contribution to the quadratic variation from birth events,
which occur at rate $\theta$ per individual 
(so, rate $N\theta |\eta|$ overall), is
$$
\int
    N \theta \gamma(x,\eta)
    \int \frac{1}{N^2} f^2(y) r(y,\eta)
    q_\theta(x,dy) 
\eta(dx) .
$$
Similarly, the increment in $\langle f, \eta\rangle$ resulting from 
the death of an individual at $x$ is $-f(x)/N$, and so combining with the 
above, the second order term in the generator takes the form
\begin{align*}
& G''(\langle f,\eta\rangle)
\frac{1}{2} N \theta
\left\{
    \int
        \gamma(x,\eta)
        \int \frac{1}{N^2}f^2(y)r(y,\eta)q_\theta(x,dy) 
    \eta(dx)
    +
    \int
        \mu_\theta(x,\eta)\frac{1}{N^2}f^2(x) 
    \eta(dx)
\right\} \\
&\qquad
= \frac{1}{2} G''(\langle f, \eta \rangle)
    \frac{\theta}{N}
	\int\left\{
		\gamma(x, \eta) \int f^2(y) r(y, \eta) q_\theta(x, dy) + f^2(x) \mu_\theta(x, \eta)
    \right\} \eta(dx) .
\end{align*}
Since $\int f^2(y) r(y, \eta) q_\theta(x, dy) \to f^2(x) r(x, \eta)$
and $r\gamma  + \mu_\theta = 2 r \gamma  - F / \theta \to 2 r \gamma$
as $\theta \to \infty$,
this converges to
\begin{align*}
\frac{\alpha}{2} G''(\langle f, \eta \rangle)
    \big\langle
    2 r(x, \eta) \gamma(x, \eta) f(x)^2,
        \eta(dx)
    \big\rangle .
\end{align*}

An entirely analogous argument shows that if $G'''$ is bounded, then
the term $\epsilon_{\theta,N}(f, G, \eta)$ 
in~(\ref{generator prelimit}) will be $\bigO(\theta/N^2)$. 

If we hold $\rho_\gamma$, $\rho_r$, $\rho_F$ fixed, then 
by taking $\theta/N \rightarrow 0$, the second order term in the generator 
will vanish and we expect a deterministic limit,
for which $\partial_t \langle f, \eta_t \rangle$ is equal to~\eqref{limit of mean measure equation}.
In other words, the limit is a weak solution to the deterministic equation
\begin{equation}
\label{deterministic limit}
\partial_t\varphi_t(x)
=
    r(x,\varphi_t)
    \DG\big(
        \gamma(\cdot,\varphi_t) \varphi_t(\cdot)
    \big)(x)
    + F(x, \varphi_t) \varphi_t(x) 
\end{equation}
in the sense of Definition~\ref{defn:weak_solutions},
where $\varphi_t$ is the density of $\eta_t$, if it has a density.
On the other hand, if $N = \alpha \theta$ for some $\alpha > 0$,
the second order term remains, and we expect a ``generalised superprocess'' limit.
The limiting quadratic variation
is exactly as seen in Theorem~\ref{thm:nonlocal_convergence}.

\paragraph{One-step convergence:} \label{sec:one_step_heuristics} \llabel{one_step_heuristics}
In order to pass directly to a classical PDE limit
in Theorem~\ref{thm:local_convergence}
we impose the stronger condition
that $\theta/(N\epsilon^d) \to 0$ and also require that
$\theta\epsilon^2\to\infty$. 
Recall that in this case, we take $\rho_F^\epsilon$ 
to be a symmetric Gaussian density with variance 
$\epsilon^2$. The condition $\theta\epsilon^2\to\infty$ 
ensures that $\epsilon^2$ is large enough relative to $1/\theta$
that the regularity gained by smoothing our population density by convolution with
$\rho_\epsilon$ is preserved under the dynamics dictated by $q_{\theta}$.
To understand the first condition, note that we are aiming to obtain a 
deterministic expression for the limiting population density. 
It is helpful to think
about a classical Wright-Fisher model (with no spatial structure and just two types, say). 
We know then that if the
timescale $\theta$ is on the same order as population size $N$, we see stochastic
fluctuations in the frequencies of the two types in the limit as $N\to\infty$; to 
obtain a deterministic limit, we look over timescales that are short relative to population
size. In our setting, the total population size is replaced by the local population
size, as measured by convolution with $\rho_{\epsilon}$, which we expect to be 
of order $N\epsilon^d$, and so in order to ensure a deterministic limit we 
take $\theta/(N\epsilon^d)\to 0$.

\subsection{Motion of ancestral lineages}
\label{sec:lineage_motion}

Although our proof of Theorem~\ref{thm:lineages}
uses an explicit representation in terms of the lookdown process,
the result can be understood through informal calculations.
Suppose that we have traced a lineage back to an individual 
at location $y$ at time $t$.
Looking further back through time, at the time of the birth of that 
individual, the lineage will jump to the location of the parent of the
individual.
Now, the rate at which new individuals are born to parents at $x$ 
and establish at $y$ is
$$
\theta N\eta^N_t(dx)    \gamma(x, \eta^N_t) q_\theta(x, dy) r(y, \eta^N_t) .
$$
Suppose that $\eta^N$ did have a density (in the prelimit it does not),
say $\eta^N_t(dx) = \varphi^N_t(x) dx$.
Informally,
since the number of individuals near $y$ is $N\varphi^N_t(y) dy$,
the probability that a randomly chosen individual near $y$
is a new offspring from a parent at $x$ in $[t, t+dt)$ is
\begin{equation} \label{eqn:informal_rates}
\frac{\theta
    \varphi^N_t(x) \gamma(x, \eta^N_t) r(y, \eta^N_t)
}{
    \varphi^N_t(y)
} \frac{ q_\theta(x, dy) }{ dy } dx dt .
\end{equation}
Leaving aside questions of whether a lineage can be treated 
as a randomly chosen individual, we define
a continuous-time jump process 
whose transition rates, conditional on $(\varphi^N_t)_{t=0}^T$, 
are given by \eqref{eqn:informal_rates}. Because we are tracing the
lineage backwards 
in time we make the substitution $s=T-t$ and write
$(L^N_s)_{s=0}^T$ for the location of a lineage that moves 
according to these jump rates.
Then, abusing notation to write 
$q_\theta(x, y)$ for the density of $q_\theta(x,dy)$,
\begin{align} \label{eqn:lineage_generator_Edt}
    \begin{split}
    &\IE[f(L^N_{s+ds}) - f(y) \;|\; L^N_s = y]
    \\&\qquad 
    =
    ds\, \theta \int \left(f(x) - f(y)\right)
    \frac{
        \varphi^N_{T-s}(x) \gamma(x, \eta^N_{T-s}) r(y, \eta^N_{T-s})
    }{
        \varphi^N_{T-s}(y)
    }
    q_\theta(x, y) dx .
    \end{split}
\end{align}
(Note that this integral is with respect to $x$.)
Referring back to Remark~\ref{rem:DG_limit},
a quick calculation shows that as $N \to \infty$,
\begin{align*}
    &
    \theta \int \big(f(x) - f(y)\big) g(x) q_\theta(x, y) dx \\
    &\qquad =
    \theta \int \big\{
        (f(x) g(x) - f(y) g(y)) - f(y) (g(x) - g(y))
    \big\} q_\theta(x, y) dx \\
    &\qquad \to
        \DG^*(fg)(y) - f(y) \DG^* g(y) . 
\end{align*}
Applying this to~\eqref{eqn:lineage_generator_Edt} with $g = \varphi_{T-s} \gamma$,
this suggests that the generator of the limiting process is
\begin{align} \label{eqn:heuristic_lineage_generator}
    \Lgen_s f
    &=
    \frac{r}{\varphi_{T-s}}
    \left\{
        \DG^*(\gamma \varphi_{T-s} f) - f \DG^*(\gamma \varphi_{T-s})
    \right\} .
\end{align}
This agrees with Theorem~\ref{thm:lineages}.

\section{The lookdown process}
    \label{sec:lookdown}

Our characterisation of the motion of lines of descent (from which we establish
that of ancestral lineages) when we pass
to the scaling limit in our model will be justified via a lookdown construction.
In this section 
we present such a construction for the general population model 
of Definition~\ref{defn:mgale_construction}. It will be
in the spirit of~\cite{kurtz/rodrigues:2011}. 
The general set-up is as follows.
Each individual will be labelled with a ``level'', a number in $[0, N]$.
We will still encode the process embellished by these levels
as a point measure:
if the $i^\mathrm{th}$ individual's spatial location is $x_i$
and level is $u_i$, then we will write
\[
    \lp^N = \sum_i \delta_{x_i, u_i} ,
\]
which is a measure on $\IR^d \times [0, N]$.
Note that each individual contributes mass 1 to the measure,
not $1/N$ as above. If we assign mass $1/N$ to each individual and 
ignore the levels we will recover our population model. 
Moreover, at any time, the levels of 
individuals in a given spatial region 
will be exchangeable and conditionally uniform on $[0, N]$:
in particular, choosing the $k$ individuals with the lowest levels 
in that region
is equivalent to taking a uniform random sample of size $k$ 
from the population in the region.
However, this exchangeability is only as regards the \emph{past}:
an individual's level encodes information about their future reproductive output,
since individuals with lower levels tend to live longer, and have more offspring.
For more explanation of the set-up and how this is possible,
see \citet{kurtz/rodrigues:2011} and \citet{etheridge/kurtz:2019} 
(and note that our $N$ corresponds to the $\lambda$ of those papers).
The power of this approach is that we can pass to a limit under the
same scalings as described in Theorem~\ref{thm:nonlocal_convergence}, 
and 
the limiting ``spatial-level'' process will still be a point measure,
and so we explicitly retain the notion of individuals and lineages
in the infinite-population limit.

\subsection{Lookdown representation of the model of 
Definition~\ref{defn:mgale_construction}}
\label{sec:lookdown_defn}

{\em For the remainder of this subsection, when there is no risk of ambiguity we shall
suppress the superscript $N$ on the processes $\eta$ and $\xi$.}

In this subsection,
we'll define the process $(\lp_t)_{t \ge 0}$ in terms of the 
dynamics of labelled particles,
and write down its generator.
The dynamics depend on the spatial locations of particles,
and in this section $\eta_t$ is the corresponding spatial measure,
i.e.,
$$
    \eta_t(\cdot) = \frac{1}{N} \lp_t(\cdot \times [0, N])  .
$$
A nontrivial consequence of the way we define $\lp_t$ will be that
the process $(\eta_t)_{t \ge 0}$ defined in this way \revpoint{1}{19}
has the same distribution as the process $(\eta_t)_{t \ge 0}$ 
of Definition~\ref{defn:mgale_construction},
which provides our justification for using the same notation for both.

Following~\cite{etheridge/kurtz:2019}, we build the generator step by step
from its 
component parts. 
Suppose that the initial population is composed of $O(N)$ particles
with levels uniformly distributed on $[0, N]$,
and that the current state of the population is $\lp$,
with spatial projection $\eta$.

An individual at spatial location $x$ with level $u$
produces one juvenile offspring at rate 
$$
2 \theta \left(1 - \frac{u}{N}\right) \gamma(x, \eta) ,
$$
which disperses to a location relative to $x$ drawn 
from the kernel $q_{\theta}(x, \cdot)$.
Averaging over the uniform distribution of the level $u$,
we recover the birth rate $\theta\gamma(x, \eta)$.
This juvenile -- suppose its location is $y$ --
either survives, with probability $r(y, \eta)$, or immediately dies.
(As before, ``maturity'' is instantaneous.)
If it survives, 
a new level $u_1$ is sampled independently and uniformly from $[u,N]$,
and the parent and the offspring are assigned in random order to the 
levels $\{u, u_1\}$.
This random assignment of levels to parent
and offspring will ensure that assignment of individuals to levels remains exchangeable.

Evidently this mechanism increases the proportion
of individuals with higher levels.
To restore the property that
the distribution of levels is conditionally uniform
given $\eta$,
we impose that 
the level $v$ of an individual at location $x$
evolves according to the differential equation
$$
    \dot{v}
    =
    -\theta \frac{v}{N} \left(N - v\right)
    \gamma(x, \eta) \int_{\IR^d} r(y, \eta) q_\theta(x, dy) .
$$
Since $v \in [0, N]$, this moves levels down;
see~\cite{etheridge/kurtz:2019}, Section~3.4 for a detailed explanation.

This drift does not allow levels to cross below 0, \revpoint{1}{20}
while we will declare that particles whose levels move above $N$ are regarded as dead
(and are removed from the population).
Therefore, in order to incorporate death,
the level of the individual at location $x$ with level $u$
moves upwards at an additional rate $\theta\mu_\theta(x,\eta) u$.
Since levels are uniform,
it is easy to check that if $\mu_\theta$ were constant,
this would imply an exponential lifetime for each individual;
see~\cite{etheridge/kurtz:2019}, Section~3.1
for more general justification.

Putting these together,
the level $u$ of an individual at $x$ evolves according to:
\begin{equation} \label{eqn:dot_u}
    \dot u
    =
    - \theta\frac{u}{N} \left(N - u\right)
    \gamma(x, \eta) \int_{\IR^d} r(y, \eta) q_\theta(x, dy) 
    +
    \theta\mu_\theta(x,\eta) u .
\end{equation}
We shall write 
\begin{align*}
    b_\theta(x, \eta)
    :=
    \theta\left(
    \gamma(x,\eta) \int_{\IR^d} r(y, \eta) q_\theta(x, dy)
    -
    \mu_\theta(x,\eta)
    \right) ,
\end{align*}
which captures the local net difference between reproduction and death,
and
\begin{align} \label{c_defn}
    c_\theta(x, \eta)
    :=
    \frac{\theta}{N} \gamma(x, \eta) \int_{\IR^d} r(y, \eta) q_\theta(x, dy) ,
\end{align}
which captures the local rate of production of successful offspring.
Recall from equation~\eqref{eqn:mu_defn} that \revpoint{1}{21}
$F(x,\eta) = \theta(r(x,\eta)\gamma(x,\eta) - \mu_\theta(x,\eta))$,
and so 
\begin{equation}  \label{b_defn}
b_\theta(x, \eta)
=    \theta \gamma(x, \eta) \int_{\IR^d} \left( r(y, \eta) - r(x, \eta) \right) 
	q_\theta(x, dy)
    + F(x, \eta). 
\end{equation}
Under Assumptions~\ref{def:model_setup}, as $\theta \to \infty$,
$c_\theta(x, \eta)$ will tend to $\alpha \gamma(x, \eta) r(x, \eta)$, and
\begin{equation}
\label{eqn:b_limit}
    b_\theta(x, \eta) \to \gamma(x, \eta) \DG r(x, \eta) + F(x, \eta) .
\end{equation}

We can then rewrite the differential equation 
governing the dynamics of the level of each individual as
\begin{align}
\dot{u}
    &=
    \theta\gamma(x,\eta) \int_{\IR^d} r(y, \eta) q_\theta(x, dy)
    \left\{
        -\frac{u}{N}\left(N - u\right)
        + u
    \right\}
    -
	b_\theta(x,\eta) u
    \nonumber \\
    &=
    c_\theta(x, \eta) u^2
    -
	b_\theta(x, \eta) u
    . \label{differential equation for level}
\end{align}

Now, we can write down the generator for $(\lp_t)_{t \ge 0}$,
the lookdown process.
In what follows, we will write sums (and, products) over ``$(x, u) \in \xi$''
to mean a sum over the (location, level) pairs of each individual in the population.
Test functions for $\lp$ will take the form
\begin{equation} \label{eqn:test_functions}
    f(\lp) = \prod_{(x,u)\in \lp}g(x,u)=\exp\left(\int \log g(x,u)\lp(dx, du)\right),
\end{equation}
where
$g(x,u)$ is differentiable in $u$ and 
smooth in $x$.
We will also assume that $0\leq g(x,u) \leq 1$ for all $u\in [0,N]$,
and $g(x,u)\equiv 1$ for $u\geq N$.
In the expressions that follow,
we shall often see one or more factor of $1/g(x,u)$;
it should be understood that if $g(x,u)=0$,
then it simply cancels 
the corresponding factor in $f(\lp)$.

First consider the terms in the generator that come from birth events.
When a birth successfully establishes,
a new level is generated above the parent's level,
and this new level is assigned to either the offspring or the parent.
Since the probability of each is 1/2,
the contribution of birth to the generator maps $f(\lp)$ to
\begin{align}
\label{partition over level assignment}
    &
    f(\lp)
    \sum_{(x, u) \in \lp}
    2 \frac{\theta}{N} \gamma(x, \eta)
    \int_u^N
    \int_{\IR^d}
    \left(
    \frac{1}{2}
    \bigg\{
            g(y, u_1)
        + \frac{ g(y, u) g(x, u_1) }{ g(x, u) }
    \bigg\}
        - 1
    \right)
    r(y, \eta) q_\theta(x, dy)
    du_1
    \\
    \begin{split} \label{eqn:birth_generator}
&=
    f(\lp)
    \sum_{(x, u) \in \lp}
    2 \gamma(x, \eta)
    \bigg\{
        \frac{1}{2 N}
        \int_u^N
        g(x, u_1) du_1
        \frac{
            \theta \int_{\IR^d} (g(y, u) - g(x, u)) r(y, \eta) q_\theta(x, dy)
        }{
            g(x, u)
        }
    \\ & \qquad \qquad \qquad {}
        + \frac{\theta}{N}
        \int_u^N \int_{\IR^d}
        \left( \frac{g(y, u_1) + g(x, u_1)}{2} - 1 \right)
        r(y, \eta) q_\theta(x, dy)
        du_1
    \bigg\}
    .
    \end{split}
\end{align}
In (\ref{partition over level assignment}),
$u_1$ is the new level and $y$ is the offspring's location,
and so the two terms in the integral correspond to the two situations:
in the first, we have added an individual at $(y, u_1)$,
while in the second, we replace an individual at $(x, u)$
by one at $(x, u_1)$ and another at $(y, u)$.
We've rewritten it in the form~(\ref{eqn:birth_generator})
because each of the two pieces
naturally converges to a separate term in the limit.

The remaining term in the generator is due to the motion of particles' levels.
Reading off from~(\ref{differential equation for level}),
it takes the form
\begin{align} \label{eqn:level_generator}
    f(\lp)
    \sum_{(x, u) \in \lp}
    \left(
        c_\theta(x, \eta) u^2
        -
        b_\theta(x, \eta) u
    \right)
    \frac{\partial_u g(x,u)}{g(x,u)} .
\end{align}

We can now define the spatial-level process
explicitly as a solution to a martingale problem,
whose generator is just the sum of
\eqref{eqn:birth_generator} and \eqref{eqn:level_generator}.
We need some notation. 
Write ${\mathcal C}={\mathcal C}(\IR^d\times [0,\infty))$ for the counting measures 
on $\IR^d\times [0,\infty)$ and ${\mathcal C}_N$ for the subset consisting of 
counting measures on $\IR^d\times [0,N]$.

\begin{definition}[Martingale Problem Characterisation]
    \label{defn:lookdown_mgale}
For given positive values of $N$ and $\theta$, define the generator $A^N$
by
\begin{equation}
	\label{eqn:lookdown_generator}
\begin{split}
& A^{N}f(\lp ) \\
&\quad =
    f(\lp)
    \,\sum_{(x,u)\in \lp}\,
    2 \gamma(x, \eta)
    \Bigg\{ \frac{1}{2N} \int_u^N g(x,u_1) du_1
            \frac{
                \theta \int_{\IR^d}
                (g(y,u) - g(x,u))
                r(y, \eta) q_{\theta}(x,dy)
            }{ g(x,u) }
        \\
    &\qquad\qquad\qquad\qquad\qquad\qquad\qquad {} +
        \frac{\theta}{N} \int_u^N
        \int_{\IR^d}\left(
            \frac{ g(y,u_1) + g(x,u_1) }{ 2 } - 1
        \right)
        r(y, \eta) q_{\theta}(x,dy)
        du_1
    \Bigg\}\\
    &\qquad\qquad {} +
    f(\lp) \sum_{(x,u)\in\lp}\,
    \left(
        c_\theta(x,\eta) u^2 - b_{\theta}(x,\eta)u
    \right)
    \frac{\partial_u g(x,u)}{g(x,u)}
    ,
\end{split}
\end{equation}
where $f(\lp) = \prod_{(x, u) \in \lp} g(x, u)$ is as defined in \eqref{eqn:test_functions},
and $\eta(\cdot) = \lp(\cdot \times [0, N]) / N$
as before. 
Given $\lp_0 \in {\mathcal C}_N$, 
we say that a ${\mathcal D}_{[0,\infty)}({\mathcal C}_N)$-valued process 
$(\xi_t)_{t \ge 0}$ is a solution to the $(A^N, \lp_0)$ martingale problem
if $f(\lp_t) - f(\lp_0)-\int_0^t A^N f(\lp_s) ds$ is a martingale (with 
respect to the natural filtration) for all 
test functions $f$ as defined above.
\end{definition}

The martingale problem for finite $N$ has a unique solution, \revpoint{1}{22}
since it is a finite-rate jump process.
Next we state the limiting martingale problem,
for which we do not necessarily have uniqueness.
As before, the parameter $\alpha$ will correspond to 
$\lim_{N\to\infty}\theta(N)/N$. Whereas for finite $N$, conditional
on the population process $\eta_t^N$, the levels of particles are
independent and 
uniformly distributed on $[0,N]$, in the infinite population limit, 
conditional on $\eta_t$, the process $\xi_t$ is Poisson distributed on 
$\IR^d\times [0,\infty)$ with mean measure $\eta_t\times\lambda$, where 
$\lambda$ 
is Lebesgue measure.

\begin{definition}[Martingale Problem Characterisation, scaling limit]
    \label{defn:limiting_lookdown_mgale}
    Fix $\alpha \in [0, \infty)$, and define  
    test functions $f$ by 
    $f(\xi) = \prod_{(x, u) \in \xi} g(x, u)$ 
	with $g$ differentiable in $u$, smooth in $x$, satisfying 
$0\leq g(x,u)\leq 1$ and
    such that there exists a $u_0$ with $g(x, u) = 1$ for all $u > u_0$.
    Then, define the operator $A$ on such test functions by 
    \begin{align} \begin{split} \label{eqn:limiting_lookdown_generator}
    A f(\lp)
    &=
        f(\lp)  \sum_{(x, u) \in \xi}
        \gamma(x, \eta)
            \frac{
                \DG(g(\cdot, u) r(\cdot, \eta))(x) - g(x,u) \DG r(x,\eta)
            }{
                g(x, u)
            }
    \\ &\qquad {} +
        f(\lp) \sum_{(x, u) \in \xi}
        2 \alpha \gamma(x, \eta) r(x, \eta) \int_u^\infty (g(x, u_1) - 1) du_1
    \\ &\qquad {} +
        f(\lp) \sum_{(x, u) \in \xi}
        \left(
            \alpha \gamma(x, \eta) r(x, \eta) u^2
            -
            \left\{
                \gamma(x, \eta) \DG r(x, \eta) + F(x, \eta)
            \right\} u
        \right)
        \frac{\partial_u g(x, u)}{ g(x,u) }  ,
    \end{split} \end{align}
    where $\eta(\cdot) = \lim_{u_0 \to \infty} \frac{1}{u_0} \xi(\cdot \times [0, u_0])$.
    We say that a ${\mathcal D}_{[0,\infty)}(\mathcal{C})$-valued process $(\xi_t)_{t \ge 0}$
    is a solution to the $(A, \lp_0)$ martingale problem
    if it has initial distribution $\lp_0$
    and $f(\lp_t) - f(\lp_0)-\int_0^t Af(\lp_s) ds$ is a martingale (with
respect to the natural filtration) for all 
test functions $f$
    as defined above.
\end{definition}

The lookdown processes have been carefully constructed so that
observations about the past spatial positions of individuals in 
the population do not give us any information about the assignment of 
individuals to levels.
In other words, the dynamics of the lookdown
process preserve the conditionally uniform (or in the limit, conditionally Poisson) structure
-- if started with uniform levels,
levels are uniform at all future times.
Moreover, if we average over levels in the expression for the generator
(equation~\eqref{eqn:lookdown_generator} or~\eqref{eqn:limiting_lookdown_generator})
we recover the generator for the population process.
Once this is verified (along with some boundedness conditions)
the Markov Mapping Theorem (Theorem~\ref{thm:mmt}; also see \citet{etheridge/kurtz:2019})
tells us that by ``removing labels'' from the lookdown process $\lp$
we recover the population process $\eta$.

To make this precise,
define the spatial projection maps
$\kappa^N: \mathcal{M}(\IR^d \times [0, N]) \to \mathcal{M}(\IR^d)$
by $\kappa^N(\xi^N)(\cdot) = \xi^N(\cdot \times [0, N]) / N$,
and $\kappa: \mathcal{M}(\IR^d \times [0, \infty)) \to \mathcal{M}(\IR^d)$
by $\kappa(\xi)(\cdot) = \lim_{u_0 \to \infty} \xi(\cdot \times [0, u_0]) / u_0$.
We will also need an inverse notion:
for a measure $\lp^N$ on $\IR^d \times [0, N]$ and a $\sigma$-field $\mathcal{F}$,
we say that \emph{$\lp^N$ is conditionally uniform given $\mathcal{F}$}
if $\kappa^N(\lp)$ is $\mathcal{F}$-measurable and 
for all compactly supported $f$,
\begin{align} \label{eqn:conditionally_uniform}
    \IE[ e^{-\langle f, \lp \rangle} \;|\; \mathcal{F} ]
    &=
    e^{-\langle H_f^N, \kappa^N(\lp) \rangle} ,
\end{align}
where
$$ H_f^N(x) = - N \log \frac{1}{N} \int_0^N e^{-f(x, u)} du . $$
In other words, the $[0, N]$ components of $\lp$
are independent, uniformly distributed on $[0, N]$, and independent of 
$\kappa^N(\lp)$.
Similarly, for a measure $\lp$ on $\IR^d \times [0, \infty)$
we say that \emph{$\lp$ is a conditionally Poisson random measure given $\mathcal{F}$}
if $\kappa(\lp)$ is $\mathcal{F}$-measurable and 
for all compactly supported $f$,
\begin{align} \label{eqn:conditionally_poisson}
    \IE[ e^{-\langle f, \lp \rangle} \;|\; \mathcal{F} ]
    &=
    e^{-\langle \int_0^\infty (1 - e^{-f(x, u)}) du, \kappa(\lp)(dx) \rangle} .
\end{align}
In other words, $\lp$ is conditionally Poisson with Cox measure $\kappa(\lp) \times \lambda$,
where $\lambda$ is Lesbegue measure.

\begin{proposition}
    \label{thm:mmt_application}
    If $\tilde \eta^N$ is a solution
    of the martingale problem of Definition~\ref{defn:mgale_construction}
    with initial distribution $\eta^N_0$
    then there exists a solution $\lp^N$
    of the $(A^N, \lp^N_0)$-martingale problem of Definition~\ref{defn:lookdown_mgale}
    such that $\eta^N = \kappa^N \circ \lp^N$
    has the same distribution on $D_{\measures}[0, \infty)$
    as $\tilde \eta^N$.
    Furthermore, for each $t$, $\lp^N_t$ is conditionally uniform 
	given $\mathcal{F}^{\eta^N}_t$
    in the sense of \eqref{eqn:conditionally_uniform}.
    If uniqueness holds for the $(A^N, \lp^N_0)$-martingale problem,
    then uniqueness also holds
    for the martingale problem of Definition~\ref{defn:mgale_construction}.

    Similarly,
    if $\tilde \eta$ is a solution
    of the limiting martingale problem of Theorem~\ref{thm:nonlocal_convergence}
    with initial distribution $\eta_0$
    then there exists a solution $\lp$
    of the martingale problem of of Definition~\ref{defn:limiting_lookdown_mgale}
    such that $\eta = \kappa \circ \lp$
    has the same distribution on $D_{\measures}[0, \infty)$
    as $\tilde \eta$.
    Furthermore, $\lp_t$ is conditionally Poisson given $\mathcal{F}^{\eta}_t$
    in the sense of \eqref{eqn:conditionally_poisson}.
    If uniqueness holds
    for the martingale problem of Definition~\ref{defn:limiting_lookdown_mgale}
    then uniqueness also holds for 
    the martingale problem of Theorem~\ref{thm:nonlocal_convergence}.
\end{proposition}

Now we can present the main convergence theorem that is analogous
to Theorem~\ref{thm:nonlocal_convergence} for the population process.

\begin{theorem}
    \label{thm:lookdown_convergence}
    Let $(\lp_t^N)$ satisfy Definition~\ref{defn:lookdown_mgale}
    and assume that as $N \to \infty$,
    $\theta \to \infty$ in such a way that $\theta/N \to \alpha$.
    Let $\eta_0^N = \kappa(\lp_0^N)$ and
    suppose also that $\eta^N_0 \to \eta_0$ in $\measures$,
    and that for each $N$, $\lp^N_0$ is conditionally uniform given $\eta_0^N$
    in the sense of \eqref{eqn:conditionally_uniform}.
    Then, $(\lp_t^N)_{t \ge 0}$ has a subsequence which converges in distribution as $N \to \infty$
    to a measure-valued process $(\lp_t)_{t \ge 0}$
    with $\lp_t$ conditionally Poisson given $\eta_t = \kappa(\lp_t)$ for each $t$
    in the sense of \eqref{eqn:conditionally_poisson},
    that is a solution to the martingale problem of Definition~\ref{defn:limiting_lookdown_mgale}.
\end{theorem}

Both results are proved in Section~\ref{sec:lookdown_proofs}.

\subsection{Explicit construction of lines of descent}
    \label{sec: individual lines of descent}

The main interest in using a lookdown construction for our population
processes is that it allows us to retain information about the 
relatedness of individuals as we pass to the infinite population limit.
In order to exploit this, in this section we write down stochastic 
equations for the locations and levels of individuals in the prelimiting lookdown model.
We will then be able to pass to the scaling limit.
This provides an explicit description of the solution to
the limiting martingale problem of Definition~\ref{defn:limiting_lookdown_mgale}
which will enable us to identify all 
individuals in the current population that are descendants of a
given ancestor at time zero. In theory at least, this allows us to
recover all the information about genealogies relating individuals 
sampled from the present day population. This idea draws on the notion
of ``tracers'', popular in statistical physics and used in population
genetics by a number of authors including 
\citet{hallatschek/nelson:2008}, \citet{durrett/fan:2016}, and \citet{biswas/etheridge/klimek:2018}.

We will construct the process using a Ulam-Harris indexing scheme.
First, we assign each individual alive at time 0 a unique label from $\IN$.
Suppose an individual with label $a$ and level $u$ reproduces,
and as a result there are two individuals, one with level $u$ and one with a new level $u_1 > u$.
The parent individual, previously labeled $a$, might be assigned either level.
We will track chains of descendant individuals forwards through time by
following levels, rather than individuals, and will call this a \emph{line of descent}.
So, after reproduction, we give a new label to \emph{only} the individual
that is given the new level $u_1$,
retaining the label $a$ for the individual with the old level $u$.
In this way, at each birth event, a unique label is 
assigned to the resulting individual with the higher level,
and the label of an individual may change throughout its lifetime.

Concretely, then: for each label $a$ in
$\labelspace = \bigcup_{k \ge 1} \IN^k$,
let $\Pi_a$ be an independent Poisson process on 
$[0, \infty)^2 \times \IR^d \times \{0,1\}$.
The mean measure of each $\Pi_a$ is a product of Lebesgue measure on $[0, \infty)^2$,
the density of the standard Gaussian on $\IR^d$, and 
$(\delta_0 + \delta_1)/2$ on $\{0, 1\}$.
It will also be convenient
to suppose that for each label $a$ we have an enumeration of the points in $\Pi_a$,
so we may refer to ``the $j^\text{th}$ point in $\Pi_a$'',
although the precise order of this enumeration is irrelevant.
If $(\tau, v, z, \kappa)$ is the $j^\text{th}$ point in $\Pi_a$,
then $\tau$ will determine a possible birth time,
$v$ will determine the level of the offspring,
$z$ will determine the spatial displacement of the offspring relative to the parent,
$\kappa$ will be used to determine whether parent or offspring is assigned the new level,
and the new label produced will be $a \concat j$,
i.e., the label $a$ with $j$ appended
(so, if $a = (a_1, \ldots, a_k)$ then $a \concat j = (a_1, \ldots, a_k, j)$).
Each label $a$ has a birth time $\tau_a$,
when it is first assigned,
and a (possibly infinite) death time $\sigma_a$, when its level first hits $N$.
For any $\tau_a \le t \le \sigma_a$ we denote by $X_a(t)$ and $U_a(t)$ the spatial location and level
of the individual carrying label $a$ at time $t$, respectively.
Furthermore, define
\begin{align*}
    \eta^N_t = \frac{1}{N} \sum_{a : \tau_a \le t < \sigma_a} \delta_{X_a(t)}
    \qquad \text{and} \qquad
    \lp^N_t = \sum_{a : \tau_a \le t < \sigma_a} \delta_{(X_a(t), U_a(t))} .
\end{align*}

Now, since we have defined labels so that the level does not jump,
$U_a$ satisfies~\eqref{differential equation for level}
	for $\tau_a \le t \le \sigma_a$, i.e.,
\begin{equation} \label{eqn:U_line_of_descent}
    \begin{split}
& U_a(t)
    =
    U_a(\tau_a) \\
&\qquad {}   
    + \int_{\tau_a}^{t}
    \left(
        c_\theta(X_a(s),\eta_s) U_a(s)^2
        -
        b_\theta(X_a(s),\eta_s) U_a(s)
    \right)
    ds ,
\end{split}
\end{equation}
and, of course, $\sigma_a = \inf\{t \ge \tau_a : U_a(t) > N\}$.

Potential reproduction events occur at times $\tau$
for each point $(\tau, v, z, \kappa) \in \Pi_a$ with $\tau_a \le \tau < \sigma_a$.
(We say ``potential'' since if the level of the resulting offspring is greater than $N$,
the event does not happen.)
If this is the $j^\text{th}$ point in $\Pi_a$,
the potential new label is $a \concat j$, the birth time is $\tau_{a \concat j} = \tau$,
and the spatial displacement of the potential offspring is $y(X(\tau-), z)$, where
$$
    y(x, z)
    :=
    \frac{1}{\theta}\meanq(x)
    +
    \frac{1}{\sqrt{\theta}}K(x) z,
$$
and $K(x)K^{T}(x) = \covq(x)$.

Next we must choose the new level created at the birth event.
We would like an individual with level $u$ and at spatial position $x$
to produce offspring at $y$ at instantaneous rate
\begin{equation}
	\label{rate of birth}
	2 \Big(1 - \frac{u}{N}\Big) \theta \gamma(x, \eta) r(x + y, \eta).
\end{equation}
To do this
we will associate the point $(\tau, v, z, \kappa) \in \Pi_a$ with 
level $u + v \ell$,
where $\ell$ is chosen so that the rate of appearance of points in 
$\Pi_a$ with level below $N$, that is points with
$v \ell <N-u$, is given by~(\ref{rate of birth}).
Since the mean measure of $\Pi_a$ is Lebesgue measure
in the $t$ and $v$ directions, we must take
\begin{equation}
	\label{expression for ell}
    \ell(x, y, \eta)
    =
    \frac{N-u}{2 (1-u/N) \theta \gamma(x, \eta) r(x + y, \eta) } 
    =\frac{1}{2N^{-1} \theta \gamma(x, \eta) r(x + y, \eta) } 
    ,
\end{equation}
and, using this, the (potential) new level is
\begin{equation*}
    U_{a \concat j}(\tau)
    =
    U_a(\tau)
    +
    v \ell\big(X_a(\tau-), y(X_a(\tau-), z), \eta_{\tau-} \big) .
\end{equation*}
If $U_{a \concat j}(\tau) < N$,
the new individual labeled $a \concat j$ is produced,
and $\kappa$ determines which label, $a$ or $a\concat j$,
is associated with the new location,
so
\begin{align*}
    X_{a \concat j}(\tau)
    &=
    X_a(\tau-) + (1 - \kappa) y\big(X_a(\tau-), z\big) .
\end{align*}
On the other hand
if $U_{a \concat j}(\tau) \ge N$, then $X_a$ is unchanged and 
$X_{a \concat j}$ is undefined,
so
\begin{align} \label{eqn:X_line_of_descent}
    X_a(\tau)
    &=
    X_a(\tau-) + \kappa y(X_a(\tau-), z) \ind_{U_{a \concat j}(\tau) < N} .
\end{align}
Recall that the parental \emph{individual} always retains their spatial location,
so that $\kappa = 1$ corresponds to the parent being assigned a new level,
and our line of descent switching to the offspring.
Combining these observations, $X_a$, for $\tau_a \le t < \sigma_a$, solves the equation 
\begin{align*}
    X_a(t)
    &=
    X_a(\tau_a) + 
    \int_{[\tau_a, t) \times [0, \infty) \times \IR \times [0, 1]}
    y(X_a(\tau-), z)
    \kappa
    \ind_{
        U_a(\tau) + v \ell(X_a(\tau-), y(X_a(\tau-), z), \eta_{\tau-})<N
    }
    d\Pi_a(\tau, v, z, \kappa) .
\end{align*}

Although we have described the evolution of a line of descent only for a given label
(i.e., for $\tau_a \le t < \sigma_a$),
we can extend the definition to times
$0 \le t < \sigma_a$ by setting
$X_a(t)$ equal to $X_{[a]_t}(t)$,
where $[a]_t$ is the label of the ancestor of label $a$ alive at time $t$,
and similarly for $U_a(t)$.
It is then straightforward, albeit tedious,
to write down the time evolution of $(X_a(t), U_a(t))$ for all time back 
to $t=0$
in terms of the driving Poisson processes.

\begin{remark}
Although we have a single construction that couples the processes across all $N$,
unlike in \citet{kurtz/rodrigues:2011} the actual trajectories, $X_a(\cdot)$,
do not necessarily coincide for different values of $N$,
since they are affected by the whole population process.
However, this does suggest approximating the genealogies in the infinite density limit
by simulating up until a sufficiently high level
that we have a good approximation to the population process.
\end{remark}


\subsection{Limiting processes for lines of descent}
\label{sec:limiting_lines_of_descent}

The previous section constructed the lookdown process
using the same underlying Poisson processes $\{\Pi_a\}_{a \in \labelspace}$ for
different values of $N$.
As a result, if the spatial projections $\eta$ converge, then
individual lines of descent converge pointwise (i.e., for each realization of $\{\Pi_a\}_{a \in \labelspace}$)
as $N \to \infty$. \revpoint{1}{23}
To see this, first note that if the Poisson processes are fixed
then the set of events with which a given label $a \in \labelspace$ is associated
is also fixed -- this is the sequence $(\tau_k, v_k, z_k, \kappa_k)$ associated 
with the label $a$.
To conclude that the lines of descent converge, 
first, we clearly need that the spatial projections $\eta$ converge.
Supposing that they do,
consider how a line of descent $(X_a(t), U_a(t))$ evolves.
It throws off a new line of descent at a higher level
when there is a point $(\tau, v, z, \kappa)$ in $\Pi_a$ with $\tau > \tau_a$ and
\begin{equation} \label{eqn:line_points}
    v < 2 \frac{\big(N - U_a(\tau)\big)}{N} \theta \gamma(X_a(\tau-), \eta_{\tau-}) 
	r\Big(X_a(\tau-) + y\big(X_a(\tau-),z\big), \eta_{\tau-}\Big) .
\end{equation}
Since the mean measure of the $v$ coordinate is Lebesgue measure,
$\theta/N \to \alpha$,
and $q_\theta(x, dy) \to \delta_x(dy)$,
this corresponds in the limit to new lines of descent being thrown off 
according to a Poisson process with intensity
$$
2 \alpha \gamma(X_a(t), \eta_t) r(X_a(t), \eta_t)dt \times du.
$$
Now consider the location of the line of descent:
at each birth event, with probability one half
the line of descent jumps to $X_a(t) + y$.
Taking $g$ to be a suitable test function on $\IR^d$, 
and rewriting~\eqref{eqn:line_points}, when the level is $u$
and the state of the population is $\eta$, the generator of the
spatial motion of the line of descent
applied to $g(x)$ is
\begin{align*}
    &
    \left(1 - \frac{u}{N}\right) \gamma(x, \eta)
    \theta \int_{\IR^d} r(x+y, \eta) (g(x+y) - g(x)) q_\theta(x, dy) \\
    &\qquad {}
    =
    \left(1 - \frac{u}{N}\right) \gamma(x, \eta)
    \bigg\{
        \theta \int_{\IR^d} (r(x+y, \eta) g(x+y) - r(x, \eta) g(x)) q_\theta(x, dy) \\
        &\qquad \qquad {}
        -
        \theta \int_{\IR^d} (r(x+y, \eta) - r(x, \eta) ) g(x) q_\theta(x, dy) 
    \bigg\} \\
    &\qquad {}
    \to
    \gamma(x, \eta)
    \left(
        \DG(rg)(x) - g(x) \DG(r)(x)
    \right) ,
    \qquad \text{as } N, \theta \to \infty .
\end{align*}
Notice that the factors of 2 have cancelled,
and that the result is independent of $u$.
Also recall that $r(x, \eta)$ depends on $\eta$ only
through $\smooth{r}\eta(x)$,
which is guaranteed to be smooth, so that $\DG(r)$ 
and $\DG(gr)$ are well-defined.

We write out the differential operator above in more detail.
Recall that $\DG g(x) = \sum_i \meanq_i \partial_i g(x) + \frac{1}{2}\sum_{ij} \covq_{ij} \partial_{ij} g(x)$,
and for the moment write $r(x)$ for $r(x, \eta)$, $\meanq(x) = \meanq$, and $\covq(x) = \covq$
so that
\begin{align}
\DG(rg)(x) - g(x) \DG(r)(x)
    &= \nonumber
    r(x) \sum_i \meanq_i \partial_i g(x)
    + \sum_{ij} \partial_i r(x) \covq_{ij} \partial_j g(x)
    + \frac{1}{2} r(x) \sum_{ij} \covq_{ij} \partial_{ij} g(x) \\
    &= \label{eqn:limiting_generator}
    r(x) \left\{
        \left(
        \meanq
        + \covq \grad \log r(x)
        \right)
        \cdot
        \grad g(x)
        +
        \frac{1}{2} \sum_{ij} \covq_{ij} \partial_{ij} g(x)
    \right\} .
\end{align}

The only thing that remains is to describe how the levels change,
but this is immediate from applying limit~\eqref{eqn:b_limit}
to equation~\eqref{differential equation for level}.

We summarize the results in a proposition.

\begin{proposition}[Line of descent construction]
    \label{prop:limiting_construction}
Define $J(x,\eta)$ and $\beta(x,\eta)$ by
\begin{gather*}
    r(x,\eta)\gamma(x,\eta)\covq(x) = J(x,\eta) J(x,\eta)^T \\
    \beta(x, \eta) = r(x,\eta)\gamma(x,\eta)\big(\meanq(x) + \covq(x) \grad \log r(x, \eta)\big) .
\end{gather*}
Associate with each label $a \in \labelspace = \cup_{k\geq 1}\IN^k$
an independent $d$-dimensional Brownian motion $W_a$
and an independent Poisson process $R_a$ on $[0, \infty)^2$
with Lebesgue mean measure,
and with points ordered in some way.
Given $\eta_0 \in \measures$,
let $(x_i, u_i)$ be the points of a Poisson process
on $\IR^d \times [0, \infty)$ 
with mean measure $\eta_0\times \lambda$
(the product of $\eta_0$ and Lebesgue measure).
For each $i$, begin a line of descent
with label $i$, location $X_i(0) = x_i$, 
level $U_i(0) = u_i$, and birth time $\tau_i = 0$.

Write $\tau_a$ for the birth time of the label $a$ and 
$\sigma_a = \lim_{u_0 \to \infty} \inf\{t \ge 0: U_a(t) > u_0\}$
the time the level hits $\infty$.
Suppose that the spatial locations and level of each line of descent $a$
solve, for $\tau_a \le t < \sigma_a$,
\begin{align}
    \label{eqn:limiting_construction}
    \begin{split}
X_a(t)
    &=
    X_a(\tau_a)
    + \int_{\tau_a}^{t}
        \beta(X_a(s), \eta_s) ds
    + \int_{\tau_a}^{t}
        J(X_a(s),\eta_s) dW_a(s)
    \\
U_a(t)
    &=
    U_a(\tau_a)
    + \int_{\tau_a}^{t}
    \bigg(
        \alpha \gamma(X_a(s),\eta_s)
        r(X_a(s), \eta_s) U_a(s)^2
\\ &\qquad \qquad \qquad {}   
        -
        \big\{
            \gamma(X_a(s),\eta_s) \DG r(X_a(s),\eta_s)
            + F(X_a(s), \eta_s)
        \big\}
        U_a(s)
    \bigg)
    ds ,
    \end{split}
\end{align}
where $\eta_t = \lim_{u_0 \to \infty} \eta_t^{[u_0]}$ and
\begin{align*}
    \eta_t^{[u_0]}
    :=
    \frac{1}{u_0}
        \sum_{a : \tau_a \le t < \sigma_a \; ; \; U_a(t) < u_0} \delta_{X_a(t)} .
\end{align*}
Each point in each $R_a$ denotes a potential birth time for $a$:
if the $j^\text{th}$ point in $R_a$ is $(\tau, v)$, with 
$\tau_a \le \tau < \sigma_a$,
then a new line of descent with label $a \concat j$ is produced,
with birth time $\tau_{a \concat j} = \tau$,
    location $X_{a \concat j}(\tau) = X_a(\tau)$, and level
\begin{align} \label{eqn:limiting_constribution_U}
    U_{a \concat j}(\tau) = U_a(\tau)
    + \frac{v}{ 2 \alpha \gamma(X_a(\tau), \eta_\tau) r(X_a(\tau), \eta_\tau) } ,
\end{align}
if this is finite.
For any solution $\{(X_a(t), U_a(t))_{t \ge 0}: a \in \labelspace \}$  \revpoint{1}{24}
to~\eqref{eqn:limiting_construction} and~\eqref{eqn:limiting_constribution_U},
the process $\eta_t$
is a solution to the martingale problem of Theorem~\ref{thm:nonlocal_convergence},
and the process
\begin{align} \label{eqn:lp_defn_XU}
    \lp_t = \sum_{a : \tau_a \le t < \sigma_a} \delta_{(X_a(t), U_a(t))} 
\end{align}
is a solution to the martingale problem of
Theorem~\ref{thm:lookdown_convergence}.
\end{proposition}

In particular, note that if $\alpha=0$,
no new lines of descent are produced.
More precisely, comparing with~(\ref{expression for ell}),
they are produced, but ``at infinity'',
and their trace is seen in the 
spatial motion of the line of descent which results from
the production of these lineages.

\begin{proof}[Proof of Proposition~\ref{prop:limiting_construction}:]
    Let $(X, U) = \{(X_a(t), U_a(t))_{t \ge 0}: a \in \labelspace \}$ \llabel{XU_clarification}
    be a solution to the system of equations~\eqref{eqn:limiting_construction}
    and~\eqref{eqn:limiting_constribution_U}.
    The fact that $\lp$ defined with these using~\eqref{eqn:lp_defn_XU}
    is a solution to the martingale problem of Theorem~\ref{thm:lookdown_convergence}
    is an application of It\^o's theorem.
    Furthermore, in Proposition~\ref{thm:mmt_application} we showed that
    the conditional Poisson property of $\lp_0$ is preserved
    (i.e., holds for $\lp_t$ for all $t$), and so
    $(\eta_t)_{t \ge 0}$ is well-defined,
    and furthermore that $\eta_t$ is a solution
    to the martingale problem of Theorem~\ref{thm:nonlocal_convergence}.

    For completeness, we should also show that $\eta_t$ defined in this way is c\`adl\`ag.
    However, this can be verified by 
    considering $\eta_t$ as a limit of the c\`adl\`ag processes $\eta^{[u_0]}_t$.
\end{proof}

\begin{remark} \label{remark_on_branching}
The process $\lp$ we consider is similar to the state-dependent branching processes \revpoint{1}{27}
of \citet{kurtz/rodrigues:2011}, so one might expect that the proofs there would carry over with little change.
However, there is an important difference:
Recall that the level $U_a(t)$ of a line of descent evolves as
\begin{equation} \label{eqn:u_evolution}
\dot u = c_\theta(x, \eta) u^2 - b_\theta(x, \eta) u,
\end{equation}
where $b_\theta(x, \eta)$ and $c_\theta(x, \eta)$ are defined in \eqref{b_defn}
and \eqref{c_defn} respectively.
Note that $c_\theta(x, \eta) \ge 0$, while $b_\theta(x, \eta)$ may take either sign.
Assumptions~\ref{def:model_setup} imply that $c_\theta(x, \eta)$ is bounded,
while $b_\theta(x, \eta)$, because of $F(x, \eta)$, is bounded above but not necessarily below.
In \citet{kurtz/rodrigues:2011}, $b_\theta$ was bounded above and $c_\theta$ was bounded away from zero,
so they noted that if $U_a(t) \ge b_\theta/c_\theta$ for some label $a$,
that line of descent would only move upwards from that time onwards.
Furthermore, coefficients did not depend on the state of the process (i.e., on $\eta$),
thus allowing the processes to be jointly and simultaneously constructed
for all values of $N$, with a pointwise embedding
of $(\lp^N_t)_{t \ge 0}$ within $(\lp^{M})_{t \ge 0}$ for $b_\theta/c_\theta < N < M$.
In other words, individuals with levels above $N > b_\theta/c_\theta$ at time $t_0$
do not affect $(\lp^N_t)_{t \ge t_0}$,
thus allowing a comparison of the number of lines of descent below level $u_0$
to a branching process.
Although we have provided a
joint construction of $\lp^N$ for all $N$ in Section~\ref{sec: individual lines of descent},
it does not have this monotonicity:
for one thing, $b_\theta$ and $c_\theta$ depend on the population process $\eta$
and so all individuals can affect all other ones (even those with lower levels).
Furthermore, in the deterministic case $\theta/N$, and hence $c$, converges to zero,
and so lines of descent with arbitrarily high level may drift back downwards.
Indeed, this must be the case if the population persists,
since in the deterministic case there is no branching.
\end{remark}

\section{Proofs of convergence for nonlocal models}
\label{sec:proofs}

In this section we present proofs of the first two of our
three scaling limits.
In Subsection~\ref{sec:population_density_proof} 
we prove Theorem~\ref{thm:nonlocal_convergence}, to obtain (both
stochastic and deterministic) limits in
which interactions between individuals in the population are nonlocal.
In Subsection~\ref{subsec:nonlocal to local}
we show how, in two important examples in which the nonlocal
limit is respectively a deterministic solution to a non-local equation of 
reaction-diffusion type and a deterministic solution to a nonlocal 
porous medium equation with an additional logistic growth term, 
one can pass to a further 
limit to obtain a classical PDE.

\subsection{Preliminaries}
\label{sec:preliminary_proofs}

Below we will have frequent use for the quantity
\begin{equation}
\label{generator Bf}
    B^\theta_f(x, \eta) = \theta \int_{\IR^d} (f(y) r(y, \eta) - f(x) r(x, \eta)) q_\theta(x, dy) .
\end{equation}

First, we prove Lemma~\ref{lem:conditions on r}.

\begin{proof}[Proof of Lemma~\ref{lem:conditions on r}:]
Here, we need to prove that $|\gamma(x, \eta) B^\theta_f(x, \eta)|$ is bounded,
uniformly over $x$ and $\eta$.
Note that Conditions~\ref{def:model_setup} assume nothing about $\eta$,
and so, for instance, although $r(x,m)$ has uniformly bounded derivatives,
it might still be the case that $r(x,\eta) = r(x, \smooth{r}\eta(x))$ changes arbitrarily rapidly;
the additional conditions of the Lemma prevent this from happening.

First suppose that
assumption~\ref{control through r} of  
Lemma~\ref{lem:conditions on r} is satisfied. \revpoint{1}{28}
We write
\begin{align*}
    r(y,\eta) f(y) - r(x,\eta) f(x)
&=
    r(y,\eta) (f(y) - f(x))
    + (r(y,\eta) - r(x,\eta)) f(x) 
\\ &=
    r(y,\eta) \left(
        \sum_i (y-x)_i \partial_{x_i} f(x)
        + \sum_{ij} (y-x)_i (y-x)_j \partial_{x_i x_j} f(z_1)
    \right)
\\ &\qquad {}
    + f(x) \left(
        \sum_i (y-x)_i \partial_{x_i} r(x,\eta)
        + \sum_{ij} (y-x)_i (y-x)_j \partial_{x_i x_j} r(z_2,\eta)
    \right)
 \\ &=
    \left( r(x,\eta) + \sum_j (y-x)_j \partial_{x_j} r(z_3,\eta) \right)
    \left(
        \sum_i (y-x)_i \partial_{x_i} f(x)
    \right)
\\ &\qquad {}
    + r(y,\eta) 
        \sum_{ij} (y-x)_i (y-x)_j \partial_{x_i x_j} f(z_1)
\\ &\qquad {}
    + f(x) \left(
        \sum_i (y-x)_i \partial_{x_i} r(x,\eta)
        + \sum_{ij} (y-x)_i (y-x)_j \partial_{x_i x_j} r(z_2,\eta)
    \right) ,
\end{align*}
for some $z_i = \kappa_i x + (1-\kappa_i) y$.
Integrating this against $q(x, dy)$, we get that
\begin{align*}
& \bigg| \theta \int \left(
    r(y,\eta) f(y) - r(x,\eta) f(x)
\right) q_\theta(x, dy) \bigg|
\\ &\qquad \le
    \bigg| \sum_i \left( r(x,\eta) \partial_{x_i} f(x) + f(x) \partial_{x_i} r(x,\eta) \right)
    \theta \int (y-x)_i q_\theta(x, dy) \bigg|
\\ &\qquad \qquad {} +
    \bigg| f(x) 
    \theta \int \sum_{ij} \partial_{x_i x_j} r(z_2,\eta) (y-x)_i (y-x)_j q_\theta(x, dy) \bigg|
\\ &\qquad \qquad {} +
    \bigg| \theta \int \sum_{ij} (y-x)_i (y-x)_j 
        \left(
            \partial_{x_i} f(x) \partial_{x_j} r(z_3,\eta)
            + r(y,\eta) \partial_{x_ix_j} f(z_1)
        \right) q_\theta(x, dy) \bigg| .
\end{align*}
Since $q_\theta(x,dy)$ is the density of a Gaussian with mean $\meanq(x)/\theta$
and covariance $\covq(x)/\theta$, and both $\meanq(x)$ and $\covq(x)$ are uniformly bounded,
$\theta \int (y-x)_i q_\theta(x,dy)$ is bounded as well.
Furthermore, a change of variables that diagonalizes $\covq(x)$
shows for any $g : \IR^d \to \IR^{d + d}$,
that if $C_g = \sup_y \sup_{\|z\| = 1} \sum_{ij} g(y)_{ij} z_i z_j$
and $\lambda_* = \sup_y \sup_{\|z\|=1} \sum_{ij} \covq(y)_{ij} z_i z_j$
then
$$ \theta \int \sum_{ij} g(y)_{ij} (y-x)_i (y-x)_j q_\theta(x, dy) \le C_g \lambda_* .  $$
Condition~\ref{control through r} 
gives uniform bounds on the derivatives of $r(x,\eta) = r(x,\rho_r*\eta(x))$
in this expression and so, provided $f$ also has uniformly bounded first and second derivatives,
we have a bound of the form
\begin{align*}
    |B^\theta_f| \leq K_1 + K_2 |f(x)| ,
\end{align*}
for suitable constants $K_1$, $K_2$ that depend only on the derivatives of $f$.

Now suppose instead that
assumption~\ref{control through gamma} of Lemma~\ref{lem:conditions on r} is satisfied.
First note that
\begin{align}
	    \nonumber
    & |B_f^\theta| =
        \left| \theta 
            \int_{\IR^n} 
                \big\{
                    f(y) r\big(y,\smooth{r}\eta(y)\big)
                    -
                    f(x) r\big(x,\smooth{r}\eta(x)\big)
                \big\}
            q_{\theta}(x,dy)
        \right|
    \\ 
 \label{eqn:q_bound_split}
	    &\qquad {} \leq
        \left| \theta 
            \int_{\IR^n} 
                \big\{
                    f(y) r\big(y,\smooth{r}\eta(y)\big)
                    -
                    f(x) r\big(x,\smooth{r}\eta(y)\big)
                \big\}
            q_{\theta}(x,dy)
        \right|
    \\ 
\nonumber
	    &\qquad \qquad {}  +
        \left| \theta 
            \int_{\IR^n} 
                \big\{
                    f(x) r\big(x,\smooth{r}\eta(y)\big)
                    -
                    f(x) r\big(x,\smooth{r}\eta(x)\big)
                \big\}
            q_{\theta}(x,dy) 
        \right| .
\end{align}
(Note the extra term introduced here,
$r(x,\smooth{r}\eta(y))$, has the two arguments to $r$
``at different locations'', contrary to the usual pattern.)

Writing
$K_3 = \sup_{x,m} \max_i |\partial_{x_i} f(x)r(x, m)|$
and 
$K_4 = \sup_{x,m} \max_{i,j} |\partial_{x_i x_j} f(x)r(x, m)|$,
the first term is bounded exactly as above. For the second, 
    \begin{multline*}
        r(x, \smooth{r}\eta(y)) - r(x, \smooth{r}\eta(x))
        \\
        =
        (\smooth{r}\eta(y) - \smooth{r}\eta(x)) r'(x,\smooth{r}\eta(x))
        + \frac{1}{2} (\smooth{r}\eta(y) - \smooth{r}\eta(x))^2 r''(x, \overline{m}),
    \end{multline*}
    where $\overline{m}= \kappa'\smooth{r}\eta(x) +(1-\kappa') 
\smooth{r}\eta(y)$ for some $0\leq \kappa'\leq 1$, and
    we have used $r'$ and $r''$ to denote the first and second derivatives of $r(x,m)$
with respect to the second argument.
    So, writing $K_5=\|r'\|_\infty$ and $K_6=\|r''\|_\infty$,
    the second term in~\eqref{eqn:q_bound_split} is bounded by $|f(x)|$ multiplied by
    \begin{align*}
       K_5
        \left|
            \theta \int_{\IR^d}
                (\smooth{r}\eta(y) - \smooth{r}\eta(x))
            q_\theta(x, dy)
        \right|
        +
       K_6
        \left|
            \theta \int_{\IR^d}
                (\smooth{r}\eta(y) - \smooth{r}\eta(x))^2
            q_\theta(x, dy)
        \right| .
    \end{align*}
Under Condition~\ref{control through gamma} of 
Lemma~\ref{lem:conditions on r}, this is bounded by 
a constant times $\smooth{\gamma}\eta(x)+(\smooth{\gamma}\eta(x))^2$ and 
$\sup_x m^2\gamma(x, m)$ is bounded.
Therefore $|\gamma(x,\eta) B^\theta_f(x,\eta)| \le K_7 + K_8 |f(x)|$,
where $K_7$ comes from $K_3$, $K_4$, and the supremum of $\gamma$,
while $K_8$ comes from $K_5$, $K_6$, and the supremum of $m^2 \gamma(x,m)$.
\end{proof}

\subsection{Proof of Theorem~\ref{thm:nonlocal_convergence}: convergence for the nonlocal process}
    \label{sec:population_density_proof}

In this section we prove Theorem~\ref{thm:nonlocal_convergence}.
This would be implied by convergence of the lookdown process
(see \citet{kurtz/rodrigues:2011} and \citet{etheridge/kurtz:2019});
however in our setting,
because the parameters in the lookdown process
depend on the empirical distribution,
we actually use tightness of the sequence of population processes
in the proofs of tightness for the corresponding lookdown processes.

\begin{proof}[Proof of Theorem~\ref{thm:nonlocal_convergence}.]
The proof follows a familiar pattern \revpoint{1}{29}
(see, for instance, Section~1.4 of~\citet{etheridge2000introduction}).
First we extend $\IR^d$ 
to its one-point compactification $\overline{\IR}^d$ and 
establish, in Lemma~\ref{lem:eta_compact_containment},
compact containment of the 
sequence of scaled population
processes in $\cmeasures$ 
(for which, since we have compactified $\IR^d$,
it suffices to consider the sequence of total masses); armed with this, 
tightness of the population processes in $\mathcal{D}_{[0,\infty)}(\cmeasures)$
follows from tightness of the real-valued processes $(H(\eta_t))_{t\geq 0}$ for a sufficiently large
class of test functions $H$, which we establish through an application of the
Aldous-Rebolledo criterion in Lemma~\ref{lem:eta_projections_tightness}.
These ingredients are gathered together in 
Proposition~\ref{tightness in one point compactification}
to deduce tightness of the scaled population processes in 
the larger space $\mathcal{D}_{[0,\infty)}(\cmeasures)$.

We then characterise
limit points as solutions to a martingale problem in
Lemma~\ref{lem:limit_mgale}; finally in Lemma~\ref{no mass at infty}
we check that 
in the process of passing to the limit, no mass `escaped to infinity', so that
in fact the limit points take values in $\mathcal{D}_{[0,\infty)}(\measures)$.
\end{proof}

As advertised, we work with the one-point compactification of $\IR^d$ and
consider $(\eta^N_t)_{t\geq 0}$ as a sequence of 
$\cmeasures$-valued processes. Since, for each $K>0$,
$\{\eta: \langle 1,\eta\rangle\leq K\}$
is a compact set in $\cmeasures$, we shall focus on 
controlling $(\langle 1,\eta^N_t\rangle)_{t\geq 0}$.
The key is that Assumptions~\ref{def:model_setup}
are precisely chosen to guarantee boundedness
of the net per-capita reproduction rate.

\begin{lemma}
    \label{lem:eta_f_bound}
    Under Assumptions~\ref{def:model_setup},
    for all $f \in C^2_b(\IR^d)$ with uniformly bounded first and 
	second derivatives, and all $T>0$, there exists a $C=C(f,T) < \infty$, 
	independent of $N$,
    such that
    \begin{equation}
\label{eta_f_bound}
        \IE[\langle f, \eta^N_t \rangle]
        \le
	    C\IE[\langle 1, \eta^N_0 \rangle]
    \end{equation}
    for all $N \geq 1$.
\end{lemma}

\begin{proof}
Consider the semimartingale decomposition from equation~\eqref{eqn:eta_martingale}:
\begin{align} \label{eqn:eta_f_mgale_decomp}
    \langle f, \eta^N_t \rangle
    &=
    \langle f, \eta^N_0 \rangle
    + \int_0^t \int_{\IR^d} \big\{
        \gamma(x, \eta^N_s) B^\theta_f(x, \eta^N_s)
        + f(x) F(x, \eta^N_s)
        \big\} \eta^N_s(dx) ds
    + M^N_t(f) ,
\end{align}
where $M^N_t(f)$ is a martingale and $B^\theta_f$ is defined in \eqref{generator Bf}.
First note that Condition~\ref{gamma_B_condition} of Assumptions~\ref{def:model_setup}
stipulates that $|\gamma B^\theta_f|$ is uniformly bounded by a contant times $1 + f$,
and so recalling that $F$ is bounded above, we conclude that
under Assumptions~\ref{def:model_setup}
$\gamma(x, \eta) B^\theta_f(x, \eta) + f(x) F(x, \eta) \le C_f (1 + |f(x)|)$ for some $C_f$.

Now, taking expectations in~(\ref{eqn:eta_f_mgale_decomp}),
    \begin{align}
\label{bound on intfdeta}
        \IE\left[ \langle f, \eta^N_t \rangle \right]
        &\le
        \IE\left[ \langle f, \eta^N_0 \rangle \right]
        + C_f \int_0^t \IE\left[ \langle 1 + |f|, \eta^N_s \rangle \right] dt .
    \end{align}
The bound~(\ref{eta_f_bound}) then follows by first applying Gronwall's
inequality in the case $f=1$, which yields
\[
	\IE\big[\langle 1,\eta_t^N\rangle\big]\leq e^{Ct}\IE\big[\langle 1,\eta_0^N\rangle\big],
\]
with $C$ independent of $N$, and
substituting the resulting bound on $\IE\big[\langle 1,\eta_s^N\rangle\big]$ into the
expression above.
\end{proof}

With a bound on per-capita net growth rate in hand,
bounds on the expectation of the supremum of 
the total population size over a finite time interval also follow easily.

\begin{lemma}[Compact containment for the population process]
    \label{lem:eta_compact_containment}
    Under the assumptions of Theorem~\ref{thm:nonlocal_convergence},
    for each $T>0$, there exists some constant $C_T$,
independent of $N$,
    such that
    \begin{align}
        \label{eqn:eta_mass_bound}
        \IE\left[
            \sup_{0 \le t \le T}
            \langle 1, \eta^{N}_t \rangle
        \right]
        \le C_T \IE[\langle 1,\eta_0\rangle].
    \end{align}
    In particular, for any $\delta > 0$, there exists $K_{\delta}>0$ such that
    \begin{equation}
    \limsup_{N \to \infty}
        \IP \left\{ \sup_{s \in [0,T]}
            \langle 1 ,\eta^{N}_{s}\rangle
            > K_\delta \right\}
        \leq \frac{C_T}{K_{\delta}}
        < \delta.
    \end{equation}
\end{lemma}

\begin{proof}
    First note that by the proof of Lemma~\ref{lem:eta_f_bound},
    $\IE[\langle 1, \eta^N_t \rangle] \le \IE[\langle 1, \eta^N_0 \rangle] e^{Ct}$
    for some $C$ (independent of $N$).
    Now, let $M^{N*}_t(f) = \sup_{0 \le s \le t} M^N_t(f)$,
    and as before let $\langle M^N(f) \rangle_t$ be the angle bracket process of $M^N_t(f)$.
    The Burkholder-Davis-Gundy inequality says that there is a $K$ for which
    $\IE\left[ M^{N*}_t(1) \right] \le K \IE[\sqrt{[M^N(1)]_t}]$,
    where $[M^N(1)]_t$ is the quadratic variation of $M^N(1)$.
    Furthermore, as discussed by \citet{hernandez/jacka:2022},
    the expectation of the quadratic variation of a local martingale
    is bounded by a (universal) constant multiple of the expectation of its angle bracket process
    (\cite{barlow/jacka/yor:1986}, Item (4.b'), Table 4.1, p. 162).
    Now, since $\sqrt{x} \le 1 + x$,
    in the notation of Lemma~\ref{lem:eta_f_bound}, there is a $C'$ such that
    \begin{align*}
        \IE\left[ M^{N*}_t(1) \right]
        &\le
        C'\left( 1 + \IE\left[ \left\langle M^N(1) \right\rangle_t \right] \right)
        \\ &=
        C'\left( 1 + \frac{\theta}{N} \IE\left[
            \int_0^t
	    \Big\langle\left\{
                \gamma(x, \eta^N_s)
                \int_{\IR^d} r(y, \eta^N_s) q_\theta(x, dy)
                + \mu_\theta(x, \eta^N_s)
            \right\}, \eta_s^N(dx)\Big\rangle
            ds
            \right] \right)
        \\ &=
        C'\Big( 1 + \IE\Big[
            \int_0^t
	    \Big\langle\Big\{
                \frac{2\theta}{N} \gamma(x, \eta^N_s) r(x, \eta^N_s)
        \\ &\qquad\qquad\qquad\qquad \qquad {}
		+ \frac{\gamma(x,\eta_s^N)}{N} 
		B^\theta_1(x, \eta^N_s)-\frac{1}{N}F(x,\eta^N_s)
            \Big\}, \eta_s^N(dx) \Big\rangle
            ds
            \Big] \Big) .
    \end{align*}

	We have not assumed that $F$ is bounded below, but to see that the 
	term involving $-F$ does not cause us problems, we rearrange
	equation~(\ref{eqn:eta_f_mgale_decomp}) 
	with $f=1$ to see that
	\begin{align}
		\label{integral of -F}
        \begin{split}
\IE\left[\int_0^t \Big\langle -F(x, \eta^N_s), \eta^N_s(dx) \Big\rangle ds \right]
        &= \IE[\langle 1, \eta^N_0\rangle] -\IE[\langle 1,\eta^N_t\rangle]
            \\&\qquad {}
	+ \IE\left[\int_0^t\Big\langle \gamma(x,\eta_s^N) B^\theta_1(x, \eta_s^N),\eta_s^N(dx)\Big\rangle ds\right],
        \end{split}
	\end{align}
which is bounded above since $\gamma (x,\eta)$ and $B^\theta_1(x,\eta)$ are both bounded and 
$\langle 1,\eta_t^N \rangle\geq 0$. 
    Since $\theta/N \to \alpha < \infty$,
    combining constants, we obtain that for some $C''$,
    \begin{align*}
        \IE\left[ M^{N*}_t(1) \right]
        &\le
        C' + C'' \IE[ \langle 1, \eta_0^N \rangle ] e^{tC} .
    \end{align*}
    Taking suprema and expectations on both sides of equation~\eqref{eqn:eta_f_mgale_decomp},
    then again using the fact that $\gamma(x, \eta) B^\theta_1(x, \eta) + F(x, \eta) \le C$,
    \begin{align*}
        \IE\left[\sup_{0 \le s \le T} \langle 1, \eta^N_s \rangle \right]
        &\le
        \IE[\langle 1, \eta^N_0 \rangle]
        + \IE\left[
            \sup_{0 \le t \le T}
            \int_0^t 
	    \Big\langle\left\{
                \gamma(x, \eta^N_s)
                B^\theta_1(x, \eta^N_s)
                + F(x, \eta^N_s)
            \right\}, \eta^N_s(dx) \Big\rangle ds
        \right]
        \\ & \qquad\qquad {}
        + \IE[M^{N*}_t(1)] 
        \\
        &\le
        \IE[\langle 1, \eta^N_0 \rangle]
        + C \IE\left[
            \int_0^T \sup_{0 \le s \le t} \langle 1, \eta^N_s \rangle dt
        \right]
        + C' + C'' \IE[ \langle 1, \eta_0^N \rangle ] e^{tC} .
    \end{align*}
    Once again applying Gronwall's inequality,
    \begin{align*}
        \IE\left[\sup_{0 \le s \le T} \langle 1, \eta^N_s \rangle \right]
        &\le
        C''' 
        \left(1+ \IE[\langle 1, \eta^N_0 \rangle]\right) e^{2TC} .
    \end{align*}
    For any $T$,
    the quantity on the right is bounded above by a constant $C(T)$ independent of $N$.
    As a result, for any $K > 0$,
    \begin{equation*}
    \limsup_{N \to \infty}
        \IP\left[ \sup_{0 \le s \le T} \langle 1, \eta^{N}_{s} \rangle \geq K \right]
        \leq
        \frac{C(T)}{K}.
    \end{equation*}
\end{proof}

Our next task is to show tightness of 
$(\langle f,\eta_t^N\rangle)_{t\geq 0}$ for 
$f\in C_b^\infty(\overline{\IR}^d)$. 

\begin{lemma}[Tightness of $(\langle f, \eta^{N}_t \rangle )_{t>0}$]
    \label{lem:eta_projections_tightness}
	For each $f \in C^{\infty}_{b}(\overline{\IR}^d)$, 
the collection of processes
$(\langle f, \eta^{N}_t \rangle)_{t \geq 0}$
for $N = 1, 2, \ldots$ is tight
as a sequence of c\`adl\`ag, real-valued processes.
\end{lemma}
\begin{proof}
The Aldous-Rebolledo criterion (Theorem~\ref{thm:aldous_rebolledo})
applied to the semimartingale representation of $\langle f, \eta^N_t\rangle$
of equation~\eqref{eqn:eta_f_mgale_decomp}, tells us
that it suffices to show that for each $T>0$,
(a) for each fixed $0\leq t\leq T$, the sequence $\{\langle f, \eta^N_t \rangle\}_{N \ge 1}$ is tight,
and (b) for any sequence of stopping times $\tau_N$ bounded by $T$,
and for each $\nu > 0$, there exist $\delta > 0$ and $N_0 > 0$ such that 
\begin{gather}
        \label{eqn:eta_projections_goal1}
    \sup_{N > N_0}
    \sup_{t \in [0, \delta]}
    \IP\left\{\left|
            \int_{\tau_N}^{\tau_N + t}
            \int_{\IR^d}
            \left\{
                \gamma(x, \eta^N_s) B^\theta_f(x, \eta^N_s)
                + f(x) F(x, \eta^N_s)
            \right\} 
            \eta^N_s(dx)
            ds
        \right|> \nu \right\}
        < \nu ,
    \\ \text{and} \qquad
        \label{eqn:eta_projections_goal2}
    \sup_{N > N_0}
    \sup_{t \in [0, \delta]}
    \IP\left\{\big|
        \langle M^{N}(f) \rangle_{\tau_N + t} 
            - \langle M^{N}(f) \rangle_{\tau_N} \big|
        > \nu
    \right\}
    < \nu.
\end{gather}
Tightness of $\langle f, \eta^N_t\rangle$ for fixed $t$
follows from Lemma~\ref{lem:eta_f_bound} and Markov's inequality,
so we focus on the remaining conditions.

The proof of Lemma~\ref{lem:eta_f_bound}
provides a uniform bound on $\gamma B^\theta_f$, but we only know
that $F$ is bounded above. However, by assumption, for each fixed value of
$m$, $\sup_{k\leq m}|F(x,k)|$ is uniformly bounded as a function of $x$.
Noting that $\smooth{F}\eta\leq \langle 1,\eta\rangle\|\rho_F\|_\infty$,
we can use Lemma~\ref{lem:eta_compact_containment} 
to choose $N_0$ and $K$ such that
if $N > N_0$, then
$$
    \IP\left\{\sup_{0\leq s\leq T}\langle 1, \eta_s^N\rangle\geq K\right\} < \nu/2,
$$
we now choose $\delta_1$ so that
$$
    \delta_1 \|f\|_\infty \sup\big\{ \sup_{x}|F(x,k)| : k\leq K\|\rho_F\|_\infty \big\}
    <
    \nu/4,
    \qquad
    \sup_{x,\eta}\gamma(x,\eta)B^\theta_f(\,\eta)\delta_1 < \nu/4,
$$
so that~\eqref{eqn:eta_projections_goal1} is satisfied with $\delta = \delta_1$.

Similarly,
\begin{align*}
    &
 \big|   \langle M^{N}(f) \rangle_{\tau_N + t} 
        - \langle M^{N}(f) \rangle_{\tau_N} \big|
    \\ & \qquad =
  \Big|  \int_{\tau_N}^{\tau_N + t}
        \frac{\theta}{N}
        \int_{\IR^d}
        \left\{
            \gamma(x, \eta^N_s)
            \int_{\IR^d} f^2(y) r(y, \eta^N_s) q_\theta(x, dy)
            +
            \mu_\theta(x, \eta^N_s) f^2(x)
        \right\}
        \eta^N_s(dx)
    ds \Big|
    \\ & \qquad =
   \Big| \int_{\tau_N}^{\tau_N + t}
        \frac{\theta}{N}
        \int_{\IR^d}
        \left\{
            \gamma(x, \eta^N_s)
            \left(
                2 f^2(x) r(x, \eta^N_s)
                +
                B^\theta_{f^2}(x, \eta^N_s)
            \right)
            -
            f^2(x) \frac{F(x, \eta^N_s)}{ \theta}
        \right\}
        \eta^N_s(dx)
    ds\Big| ,
\end{align*}
and so using the fact that $\theta/N \to \alpha < \infty$,
an argument entirely analogous to 
that for \eqref{eqn:eta_projections_goal1}
yields a $\delta_2$ for which \eqref{eqn:eta_projections_goal2} is
satsified. Taking $\delta=\min\{\delta_1,\delta_2\}$, the result follows.

\end{proof}


We collect the implications of the last two lemmas into a proposition.

\begin{proposition}[Tightness of $(\eta^{N}_t)_{t \geq 0}$]
\label{tightness in one point compactification}
    The collection of measure-valued processes 
    $\{(\eta^{N}_t)_{t \geq 0}: N \geq 1\}$
    is tight in $\mathcal{D}_{[0,\infty)}(\cmeasures)$.
\end{proposition}

\begin{proof}
    Theorem 3.9.1~in \cite{ethier/kurtz:1986}
    says that if the collection of $E$-valued processes
    satisfies a compact containment condition
    (for any $\epsilon > 0$ and $T>0$, there is a compact set
    such that the processes stay within that set up to time $T$ with probability at least $1-\epsilon$),
    then the collection is relatively compact (which is equivalent to tightness
since we are working on a Polish space)
    if and only if
    $\{(f(\eta^{N}_t))_{t \geq 0}: N \geq 1\}$ is relatively compact
    for all $f$ in a dense subset of $C_b(E)$
    under the topology of uniform convergence in compact sets.

Since $\{ \nu: \langle 1, \nu \rangle \leq K \} $ is compact in $\cmeasures$,
    Lemma~\ref{lem:eta_compact_containment}
    gives compact containment.
    Lemma \ref{lem:eta_projections_tightness}
    shows that the real-valued processes $\langle f, \eta^{N}_t \rangle$
    are relatively compact for all 
$f \in \mathcal{C}^{\infty}_{b}(\overline{\IR}^d)$.
    Since by the Stone-Weierstrass theorem,
    the algebra of finite sums and products of terms of this form
    is dense in the space of bounded continuous functions on $\cmeasures$,
    and tightness of $\langle f, \eta_t^N \rangle$ extends to sums and products of this form
    by Lemma~\ref{lem:product_tightness},
    we have relative compactness in $\mathcal{D}_{[0,\infty)}(\cmeasures)$.
\end{proof}

We wish to characterise the limit points of
$\{(\eta^{N}_t)_{t>0}\}_{N\geq 1}$ as solutions to a martingale problem 
with generator $\Pgen^{\infty}$ which we now identify. Most of the work
was done in Section~\ref{sec:heuristics}.
First, we record an equivalent formulation
of the martingale problems,
which were essentially laid out in Subsection~\ref{sec:population_heuristics}.

\begin{lemma}
    \label{def: MP definition of limit}
For $G \in \mathcal{C}^{\infty}(\IR)$ with $\|G'''\|_\infty<\infty$,
and 
$f \in \mathcal{C}_{b}^\infty(\overline{\IR}^d)$, define the 
function $G_f$ by 
$G_f(\eta):=G (\langle f, \eta \rangle)$.
Let $\Pgen^N$ be the generator given by
\begin{equation}
\begin{aligned}
    \Pgen^N G_f(\eta)
    &:=
    \theta N \bigg\langle
    \gamma(x, \eta)
        \int 
        \left(G(\langle f, \eta \rangle + f(z)/N) - G(\langle f, \eta \rangle)\right)
        r(z,\eta) q_\theta(x,dz)
    \\ & \qquad \qquad {}
    +
    \left(G(\langle f, \eta \rangle - f(x)/N) - G(\langle f, \eta \rangle)\right)
    \mu_\theta(x, \eta),
    \eta(dx) \bigg\rangle .
\end{aligned}
\end{equation}
The process $(\eta^N_t)_{t \geq 0}$ of Definition~\ref{defn:mgale_construction}
is the unique solution to the $(\Pgen^N, \eta_0)$-martingale problem, i.e.,
\revpoint{1}{30}
$$
    M_t :=
    G_f(\eta^N_t) - G_f(\eta^N_0)
    - \int_{0}^{t}\Pgen^N G_f(\eta^N_s)ds
$$
is a martingale for all such test functions (with respect to the natural $\sigma$-field). 

Furthermore, let $\Pgen^\infty$ be the generator given by
\begin{equation}
    \label{eq: Limit Generator Definition}
\begin{aligned}
\Pgen^{\infty} G_f(\eta)
    &:= G'(\langle f, \eta \rangle)
                   \big\langle
                        \gamma(x, \eta)
                            \mathcal{B}\left(
                            f(\cdot) r(\cdot, \eta)
                            \right)(x)
                    +
                    f(x) F(x, \eta),
                    \eta(dx)
                    \big\rangle \\
               &\qquad {}
               + \alpha G''(\langle f, \eta \rangle)
                  \big\langle
                    \gamma\left( x, \eta \right)
                    r\left(x,\eta \right)
                    f^2(x),
                    \eta (dx)
                    \big\rangle  .
\end{aligned}    
\end{equation}
A process $(\eta^{\infty}_t)_{t \geq 0}$ 
satisfies the martingale characterization of equations~\eqref{eqn:limiting_mgale_problem}
and~\eqref{eqn:limiting_mgale_variation}
if it is a solution 
to the $(\Pgen^{\infty}, \eta^\infty_0)$-martingale problem, i.e.,
if for all such test functions 
$$
    M_t :=
    G_f(\eta^{\infty}_t) - G_f(\eta^\infty_0)
    - \int_{0}^{t}\Pgen^{\infty}G_f(\eta^{\infty}_s)ds
$$
is a martingale (with respect to the natural $\sigma$-field). 
\end{lemma}

The converse -- that any solution to equations~\eqref{eqn:limiting_mgale_problem}
and~\eqref{eqn:limiting_mgale_variation}
is a solution to the $(\Pgen^{\infty}, \eta^\infty_0)$-martingale problem --
requires continuity, which we expect to be true, but have not proved.

\begin{lemma}[Characterisation of limit points]
    \label{lem:limit_mgale}
Suppose that $(\eta^{N}_0)_{N\geq 1}$ converges weakly to $\eta_0$ as 
$N\to\infty$.
Then any limit point of $\{(\eta^{N}_t)_{t \geq 0}\}_{N\geq 1}$ 
in $\mathcal{D}_{[0,\infty)}(\cmeasures)$ 
is a solution to the martingale problem for 
$(\Pgen^\infty, \eta_0)$.
\end{lemma}
\begin{proof}
We use Theorem 4.8.2 in \cite{ethier/kurtz:1986}.
First observe that the set of functions
$\{G_f(\eta):= G(\langle f, \eta \rangle ),~
G \in \mathcal{C}^{\infty}(\IR), \|G'''\|_\infty<\infty,~
f \in \mathcal{C}_{b}^\infty(\overline{\IR}^d)\}$
is separating 
on $\mathcal{M}_F(\overline{\IR}^d)$.
Therefore, it suffices to show that for any $t>0$ and $\tau > 0$ that
\begin{equation}
    \label{eq: Convergence Condition}
\lim_{N \to \infty}
\mathbb{E}\left[
\left(
G_f(\eta^{N}_{t+\tau})-G_f(\eta^{N}_t)
-\int_{t}^{t+\tau}\Pgen^{\infty}G_f(\eta^{N}_s)ds
\right)
\prod_{i=1}^{k} h_i(\eta^{N}_{t_i})
\right]=0
\end{equation}
for all $k\geq 0$, $0\leq t_1<t_2<\ldots,t_k \leq t < t+\tau$,
and 
bounded continuous functions 
$h_1,\ldots,h_k$ on $\cmeasures$.
Since $(\eta^N_t)_{t\geq 0}$ is Markov, the tower property gives that,
for each $N$,
\begin{equation}
    \label{eq: Prelimit MP Application}
\mathbb{E}\left[
\left(
G_f(\eta^{N}_{t+\tau})-G_f(\eta^{N}_t)
-\int_{t}^{t+\tau}\Pgen^{N}G_f(\eta^{N}_s)ds
\right)
\prod_{i=1}^{k}h_i(\eta^{N}_{t_i})
\right]=0 .
\end{equation}
Therefore, it suffices to show that 
\begin{equation}
    \label{eq: Convergence of Spatial Generator}
\lim_{N \to \infty}
\mathbb{E}\left[
\int_{t}^{t+\tau}
\left|
\Pgen^{N}G_f(\eta^{N}_s)
-\Pgen^{\infty}G_f(\eta^{N}_s)
\right|
ds
\prod_{i=1}^{k}h_i(\eta^{N}_{t_i})
\right]=0,
\end{equation}
and, again using the tower property, since the functions $h_i$ are
bounded, this will follow if
\begin{equation}
\lim_{N \to \infty}
\mathbb{E}\left[\left.
\int_{t}^{t+\tau}
\left|
\Pgen^{N}G_f(\eta^{N}_s)
-\Pgen^{\infty}G_f(\eta^{N}_s)
\right|
ds\right| {\cal F}_t\right]=0
\end{equation}
(where $\{{\cal F}_t\}_{t\geq 0}$ is the natural $\sigma$-field).

We rewrite $\Pgen^NG_f(\eta^N_s)$ using a 
Taylor series expansion up to third order for
$G\big(\langle f,\eta\rangle\pm f(y)/N\big)$ around
$G(\langle f, \eta \rangle )$.
As in Section~\ref{sec:heuristics} (except that now we are more explicit about
the error term), we find
\begin{equation} 
    \label{eq: Pre-Limit Generator Expanded}
\begin{aligned}
\Pgen^{N} G_f(\eta)
    :=& 
    G'(\langle f, \eta \rangle)\int_{\IR^d}  
    \theta\Big\{ \gamma(x,\eta) \int_{\IR^d} f(y)r(y,\eta) q_{\theta}(x,dy) 
    - f(x)\mu_\theta(x,\eta) \Big\} \eta(dx)\\
&+\frac{1}{2}\frac{\theta}{N}G''(\langle f, \eta \rangle)\int_{\IR^d} 
    \Big\{\gamma(x, \eta)\int_{\IR^d} f^2(y) r(y,\eta)q_{\theta}(x,dy)+f^2(x)
    \mu_\theta(x, \eta) \Big\} \eta(dx) \\
&+\frac{1}{6}\frac{\theta}{N^2}
    G'''(w) \gamma(x, \eta)\int_{\IR^d} \Big\{ f^3(y) r(y,\eta) 
    q_{\theta}(x,dy)
    - G'''(v)f^3(x)\mu_\theta(x, \eta) \Big\} \eta(dx)
\end{aligned}    
\end{equation}
for some $w,v \in [\langle f,\eta \rangle - \|f\|_\infty/N, 
\langle f,\eta \rangle + \|f\|_\infty/N]$.

Combining with equation~\eqref{limit of mean measure equation}, and the fact
that $\mu_\theta(x,\eta) \to r(x,\eta)\gamma(x,\eta)$ as $\theta\to\infty$, we have
pointwise convergence:
\begin{equation}
\lim_{N\to \infty} |\mathcal{P}^{N}G(\langle f, \eta \rangle) 
- \mathcal{P}^{\infty}G(\langle f, \eta \rangle)| = 0 .
\end{equation}
To conclude convergence of the expectation, 
we would like to apply the Dominated Convergence 
Theorem in~(\ref{eq: Convergence of Spatial Generator}).
Recall that $f$ and $G$ and their derivatives are bounded,
and $\gamma(x,\eta)$ is bounded independent of $\theta$.
Since $\theta/N^2 \to 0$,
rearranging as in \eqref{eqn:rewritten mean measure}
and using the convergence of \eqref{eqn:near_critical},
we deduce that we can dominate
$\left| \Pgen^{N}G_f(\eta^{N}_s) -\Pgen^{\infty}G_f(\eta^{N}_s) \right|$
by a constant multiple of $\langle 1+|F(x)|,\eta_s(dx)\rangle$.
Since $F$ is bounded above, there is a constant $K$ such that $|F|\leq K-F$ so
that, exactly as in equation~(\ref{integral of -F}), 
we can check that
\[
\IE\Big[\left.\int_t^{t+\tau}\big\langle |F(x,\eta^N_s)|,\eta^N_s(dx)\big\rangle ds
\right| {\cal F}_t\Big]<\infty,
\]
which concludes our proof.
\end{proof}

The last step in the proof of Theorem~\ref{thm:nonlocal_convergence}
is to check that
any limit point $(\eta_t)_{t\geq 0}$ of $\{(\eta^N_t)_{t\geq 0}\}_{N\geq 1}$
actually takes its values in $\measures$, that is, ``no mass has escaped to infinity''.

\begin{lemma}
\label{no mass at infty}
Under the assumptions of Theorem~\ref{thm:nonlocal_convergence},
if $(\eta_t)_{t\geq 0}$ is a limit point of $\{(\eta^N_t)_{t\geq 0}\}_{N\geq 1}$,
then for any $\delta>0$ and $T>0$,
\[
\IP\bigg[
    \sup_{0 \le t \le T} \eta_t\big(\{\|x\|>R\}\big)>\delta
\bigg]\to 0\qquad\mbox{ as }R\to\infty.
\]
\end{lemma}

\begin{proof}[Sketch]
Take $f_0(x)$ as in the statement of Theorem~\ref{thm:nonlocal_convergence},
i.e., $f_0$ is nonnegative, grows to infinity as $x \to \infty$,
has uniformly bounded first and second derivatives,
and has $\langle f_0, \eta^N_0\rangle$ uniformly bounded in $N$.
We take a sequence of nonnegative test functions $f_n$ that increase to the function $f_0$
and having uniformly bounded first and second derivatives,
so that there is a (single) $C$
from Condition~\ref{gamma_B_condition} of Assumptions~\ref{def:model_setup}
such that $\gamma(x,\eta) B_{f_n}^\theta(x,\eta) \le C(1 + f_n(x))$
for all $x$, $\eta$, and $f_n$.
Then, just as we arrived at equation~\eqref{bound on intfdeta},
\begin{align*}
    \IE\left[\left\langle f_n(x), \eta^N_t(dx) \right\rangle\right]
    &\le
    \IE\left[\left\langle f_n(x), \eta^N_0(dx) \right\rangle\right]
    +
    C \int_0^t \IE\left[\left\langle f_n(x), \eta^N_s(dx) \right\rangle \right] ds ,
\end{align*}
with the same constant for all $n$ and all $N$.
Gronwall's inequality then implies that 
$\IE[\langle f_n, \eta^N_t\rangle] \le C'$
for some $C'$ independent of $n$, $N$, and $t \in [0,T]$.
By first taking $N \to \infty$ and then $n \to \infty$,
we find that $\IE[\langle f_0, \eta_t(dx)\rangle] \le C'$ for $t\in [0,T]$.
However, this is for a single time --
we would like instead to uniformly bound
$\IE[\sup_{0 \le t \le T} \langle f_n, \eta^N_t(dx)\rangle]$.
This can be done in a similar but lengthier manner,
following the proof of Lemma~\ref{lem:eta_compact_containment}
and observing that bounds can be taken independent of $n$ and $N$.

Finally, since $f_0 \to \infty$ as $|x| \to \infty$,
an application of Markov's inequality
tells us that for any $\delta > 0$,
$$
    \IP\bigg\{ \sup_{0 \le t \le T} \eta_t(\{x : \|x\| > R\})>\delta \bigg\}
    \le \frac{\IE[\sup_{0 \le t \le T} \langle f_0, \eta_t\rangle]}{\delta \inf_{\{x:\|x\|\ge R\}} f_0(x)}
    \to 0 \qquad \text{as } R\to\infty .
$$
\end{proof}

\subsection{Convergence of some nonlocal equations to classical PDEs}
\label{subsec:nonlocal to local}

It is natural to conjecture that when the limit of the rescaled 
population process that we obtained in the previous section solves a nonlocal
PDE, if we further scale the kernels 
$\rho_r$, $\rho_\gamma$, and $\rho_F$ by setting 
$\rho^\epsilon(\cdot)=\rho(\cdot/\epsilon)/\epsilon^d$, as $\epsilon\to 0$, the 
corresponding solutions should converge to a limiting
population density that solves the corresponding ``classical'' PDE. We verify
this in two examples; in the first the nonlocal equation is a reaction-diffusion equation
with the ``nonlocality'' only appearing in the 
reaction term; in the second the nonlocal PDE is a special 
case of a nonlinear porous medium equation. These, in particular,
capture the examples that we explored in 
Section~\ref{ancestral lineages for travelling waves}.

\subsubsection{Reaction--diffusion equation limits}
\label{two-step convergence to FKPP}

In this subsection we prove Proposition~\ref{prop:nonlocal_to_local}. {\em The conditions of the
proposition are in force throughout this subsection.} 
The proof rests on a Feynman-Kac representation.
We write $(Z_t)_{t\geq 0}$ for a diffusion with generator $\DG^*$ and denote 
its transition density by $f_t(x,y)$. The first step is a regularity
result for this density.

\begin{lemma} \label{regularityForX1}
Fix $T>0$. There exists a constant $K= K(T)>0$ such that, for any 
$x, y\in \mathbb{R}^d$ and $t \in [0,T]$,
\begin{equation} \label{eq:boundDensityXt}
\int |f_t(x,z)-f_t(y,z)| dz \leq \frac{\|x-y\|}{\sqrt{t}} K.
\end{equation}
\end{lemma}
\begin{proof}
We first use the Intermediate Value Theorem to obtain the bound
\begin{align*}
\int |f_t(x,z)-f_t(y,z)| dz & \leq \int \|x-y\| \|\nabla f_t(w,z)\| dz
\end{align*}
where $\nabla$ acts on the first coordinate only and $w$ is in the line 
segment $[x,y]$ joining $x$ to $y$. 
Under our assumptions on $b$ and $C$, equation~(1.3) 
of~\cite{sheu:1991}, gives existence of constants 
$\lambda=\lambda(T)>0$ and $K$ such that,
\[ \|\nabla f_t(w,z)\|  \leq  \frac{K }{\sqrt{t}}p_{ \lambda  t}(w,z), \]
where $p_s(x,y)$ is the Brownian transition density.
Hence, 
\begin{align*}
\int |f_t(x,z)-f_t(y,z)| dz & \leq K 
\frac{\|x-y\|}{\sqrt{t}} \int p_{\lambda  t}(w,z)dz = 
K  \frac{\|x-y\|}{\sqrt{t}} .
\end{align*}
\end{proof}

\begin{lemma} \label{regularityForX2}
Fix $T>0$. Let $x,y \in \mathbb{R}^d$, $t \in [0,T]$, and denote by 
$(Z_t^y)_{t\geq 0}$ and $(Z_t^x)_{t\geq 0}$ independent copies of the 
diffusion $(Z_t)_{t\geq 0}$ starting 
from $y$ and $x$ respectively. There exists a constant $K=K(T)>0$ such that, 
\[ \IE[\|Z_t^y-Z_t^x\|] \leq K(\sqrt{t} + \|y-x\|). \]
\end{lemma}
\begin{proof}
First we write, 
\[ \IE[\|Z_t^y-Z_t^x\|]  = \int \int \|u-v\| f_t(y,u) f_t(x,v) du dv . \]
Under our regularity assumptions on $C$, $b$, using equation~(1.2)
of~\cite{sheu:1991}, there exist constants $K$, $\lambda =\lambda (T)>0$ 
for which,
\[ f_t(y,u) \leq K  p_{\lambda  t}(y,u). \]
It then follows that,
\begin{align}
  \IE[\|Z_t^y-Z_t^x\|]  \leq \int \int \|u-v\| K^2 p_{\lambda  t}(y,u) 
p_{\lambda  t}(x,v) dv du = K^2 \IE[\|B_{\lambda  t}^y-B_{\lambda  t}^x\|], 
\label{eq:ExpDif1}
\end{align}
where $(B_t^y)_{t\geq 0}$ and $(B_t^x)_{t\geq 0}$ are 
independent Brownian motions starting at $y$ and $x$ respectively. 
Using the triangle inequality, and writing $(B^0_t)_{t\geq 0}$ for a 
Brownian motion started from the origin,
\begin{equation}
\IE[\|B_{\lambda  t}^y-B_{\lambda  t}^x\|]
	\leq \|y-x\|+\IE[\|B^0_{2\lambda t}\|] 
\leq \|y-x\|+C\sqrt{t}. \label{eq:ExpDif2}
\end{equation}
Substituting (\ref{eq:ExpDif2}) in (\ref{eq:ExpDif1}) gives the result.
\end{proof}

We use the representations of the solutions to equations \eqref{localPDE} and \eqref{nonlocalPDEv1} respectively:
\begin{align}
\varphi_t(x) &= \IE_x\Big[\varphi_0(Z_t) 
    + \int_0^t \varphi_s(Z_{t-s})F(\varphi_s(Z_{t-s})) ds\Big], \label{FK:varphi} \\
\varphi^\epsilon_t(x) &= \IE_x\Big[\varphi_0(Z_t) 
    + \int_0^t \varphi^\epsilon_s(Z_{t-s})F(\rho^\epsilon_F*\varphi^\epsilon_s(Z_{t-s})) ds\Big],
\end{align}
from which
\begin{equation} 
\varphi_t(x) - \varphi^\epsilon_t(x) = 
    \IE_x\left[ \int_0^t \Big(\varphi_s(Z_{t-s})F\big(\varphi_s(Z_{t-s})\big)
    -\varphi^\epsilon_s(Z_{t-s})F\big(\rho_\epsilon*\varphi^\epsilon_s(Z_{t-s})\big)\Big) ds \right] ,
\label{FK:diferencesVarphis}
\end{equation}
where $\IE_x$ denotes expectation for $Z$ with $Z_0 = x$.
The key to our proof of Proposition~\ref{prop:nonlocal_to_local}
will be to replace $F(\varphi_s(Z_{t-s}))$ by $F(\rho^\epsilon_F*\varphi_s(Z_{t-s}))$ 
in this expression. 
We achieve this through three lemmas.

First we need a uniform bound on $\varphi$ and $\varphi^\epsilon$.  
\begin{lemma} 
\label{BoundednessVarphis}
For any $T>0$ there exists $M = M(T, \Vert \varphi_0 \Vert) >0$ such that, 
for all $0 \leq t \leq T$:
\[ \max\{ \| \varphi_t(\cdot) \|_\infty, \| \varphi^\epsilon_t(\cdot) \|_\infty\} 
< M.
 \]
\end{lemma}
\begin{proof}
Using that $\varphi_0$ and $F$ are bounded above, from the
representation~(\ref{FK:varphi}), we have
\begin{align*}
\varphi_t(x) \leq \| \varphi_0 \|_\infty + 
    K \IE\Big[\int_0^t \varphi_s(Z_{t-s}) ds\Big].
\end{align*}
In particular, 
\[ \| \varphi_t(\cdot) \|_\infty \leq \| \varphi_0 \|_\infty
+  K \int_0^t \| \varphi_s(\cdot) \|_\infty ds, \]
so, by Gronwall's inequality,
\[ \| \varphi_t(\cdot) \|_\infty \leq \| \varphi_0 \|_\infty \exp\left( K T \right). \]
Similarly,
$\| \varphi^\epsilon_t(\cdot) \|_{\infty} \leq \| \varphi_0 \|_\infty \exp\left( K T \right) $.
\end{proof}

We also need a continuity estimate for $\varphi$.
\begin{lemma} 
\label{ContinuityVarphi}
Let $T>0$. There exists a constant 
$K=K(T, \| \varphi_0\|_\infty)>0$ and 
$\delta_0=\delta_0(T, \| \varphi_0\|_\infty) >0$ such that 
for all $0 < \delta<\delta_0$ and $0 \leq t \leq T$,
\[ 
 \|x-y\|<\delta^3 \Rightarrow |\varphi_t(x)-\varphi_t(y)|< K\delta.
\]
\end{lemma}
\begin{proof}
First we need some notation. 
Fix $T>0$ and write $M$ for the corresponding constant from 
Lemma~\ref{BoundednessVarphis}.
Let $\| F \|_{M} = \sup_{m \in [0,M]} |F(m)|$.
We reserve $\widehat{K}$ for the constant on the right hand side of 
equation~(\ref{eq:boundDensityXt})
and $\widetilde{K}$ for the constant in Lemma~\ref{regularityForX2},
and write $K_{\varphi_0}$ for the Lipschitz constant of 
$\varphi_0$. Set 
\[ 
\delta_0=\min\Big( \frac{1}{\| F \|_M^{2}},\frac{1}{M e\big(2\| F \|_{M}+ \widehat{K}\big)}, 
    \frac{1}{\widetilde{K} K_{\varphi_0} + 2 \| F \|_M M},  1\Big).
\]
In what follows we take $0< \delta<\delta_0$. 

We first prove that the result holds if $t < \delta^2$. As before
let $Z_t^x$ and $Z_t^y$ be 
independent copies of the diffusion $Z_t$ starting at $x$ and $y$ respectively. 
From our representation (\ref{FK:varphi}) and 
Lemma~\ref{BoundednessVarphis}, we can write:
\begin{align*}
|\varphi_t(x)- \varphi_t(y)| 
&\leq \big|\IE_x[\varphi_0(Z_t)]-\IE_y[\varphi_0(Z_t)]\big| 
    + 2 \| F \|_M M t \\ 
& \leq \IE[|\varphi_0(Z_t^x)-\varphi_0(Z_t^y)|] 
    + 2 \| F \|_M M t  \\ 
&\leq K_{\varphi_0}  \IE[\|Z_t^x-Z_t^y\|] 
    + 2 \| F \|_M M t \\ 
& \leq \widetilde{K} K_{\varphi_0}  (\sqrt{t} + \|y-x\|) 
    + 2 \| F \|_M M t \\ 
& \leq \widetilde{K} K_{\varphi_0} (\delta + \delta^3) 
    + 2 \| F \|_M M \delta^2 
\leq (\widetilde{K} K_{\varphi_0} +1)\delta,
\end{align*}
where we have used Lemma~\ref{regularityForX2} in the fourth inequality and the definition of $\delta_0$ in the last inequality.

Suppose now that $\delta^2<t$.
We will follow the pattern in Lemma~2.2 of~\cite{penington:2017}. 
First, note that by the Feynman-Kac formula we have an alternative 
representation for $\varphi_t(x)$: for any $t'<t$,
\[ 
\varphi_t(x) = \IE_x\Big[ \varphi_{t-t'}(Z_{t'}) 
\exp\big( \int_0^{t'} F(\varphi_{t-s}(Z_s)) ds \big) \Big]. 
\]
Therefore, setting $t'=\delta^2$ and using Lemma~\ref{BoundednessVarphis}, for all $z$,
\[ 
e^{-\delta^2\| F \|_{M}}\IE_z\big[ \varphi_{t-\delta^2}(Z_{\delta^2})\big] 
\leq  
\varphi_t(z) 
\leq e^{\delta^2\| F \|_{M}} \IE_z\big[  \varphi_{t-\delta^2}(Z_{\delta^2})\big]. 
\]
We can then deduce that
\begin{align}
\nonumber
\varphi_t(x)-\varphi_t(y) 
&\leq e^{\delta^2\| F \|_{M}}\IE_x\big[ \varphi_{t-\delta^2}(Z_{\delta^2})\big] 
- e^{-\delta^2\| F \|_{M}}\IE_y\big[ \varphi_{t-\delta^2}(Z_{\delta^2})\big] 
\\ 
\nonumber
& = e^{\delta^2\| F \|_{M}}
\Big(\IE_x\big[ \varphi_{t-\delta^2}(Z_{\delta^2})\big]
- \IE_y\big[ \varphi_{t-\delta^2}(Z_{\delta^2})\big]\Big) 
\\ 
\nonumber
&\qquad \qquad 
+ \big(e^{\delta^2\| F \|_{M}} - e^{-\delta^2\| F \|_{M}}  \big)
\IE_y\big[ \varphi_{t-\delta^2}(Z_{\delta^2})\big] 
\\
\label{difference}
&\leq e^{\delta^2\| F \|_{M}} 
\Big(\IE_x\big[ \varphi_{t-\delta^2}(Z_{\delta^2})\big]- 
\IE_y\big[ \varphi_{t-\delta^2}(Z_{\delta^2})\big]) 
+ M\big(e^{\delta^2\| F \|_{M}}-e^{-\delta^2\| F \|_{M}}\big) .
\end{align}

To bound the differences of the expected values in the last equation 
note that, by using again Lemma~\ref{BoundednessVarphis},
\begin{align*}
\IE_x\big[ &\varphi_{t-\delta^2}(Z_{\delta^2})\big]
- \IE_y[ \varphi_{t-\delta^2}(Z_{\delta^2})\big] & \\ 
& = \int \varphi_{t-\delta^2}(z) (f_{\delta^2}(x,z)-f_{\delta^2}(y,z) ) dz  \\ 
&\leq M \int \big|f_{\delta^2}(x,z)-f_{\delta^2}(y,z) \big| dz \\ 
&\leq M \widehat{K} \frac{\|x-y\|}{\delta}  
\leq M \widehat{K} \delta^2, 
\end{align*}
where we have used Lemma~\ref{regularityForX1} and that $\|x-y\| < \delta^3$.
Substituting in~(\ref{difference}),
\begin{align*}
\varphi_t(x)-\varphi_t(y)
&\leq e^{\delta^2 \| F \|_{M}} 
\left(M \widehat{K}\delta^2 +M - M e^{-2\delta^2\| F \|_{M}}  \right) \\ 
& \leq e^{\delta^2 \| F \|_{M}}\left(M \widehat{K}\delta^2 
+ 2 M \delta^2\| F \|_{M}  \right) \\ 
&\leq e \left( M \widehat{K}+2 M \| F \|_M  \right) \delta^2 \leq  \delta,
\end{align*}
where the last two inequalities follow from the definition of $\delta$. 
Interchanging $x$ and $y$ yields the same bound for $\varphi_t(y)-\varphi_t(x)$, and the result 
follows.
\end{proof}

We proceed to control the difference between $F(\varphi)$ and 
$F(\rho^\epsilon_F*\varphi)$. 
Note first that since $\rho_F\in L^1$, 
\[
I(\epsilon):=\int_{\{\|y\|>\epsilon^{3/4}\}}\rho^\epsilon_F(y)dy
=\int_{\{\|y\|>\epsilon^{-1/4}\}}\rho_F(y)dy
\to 0\qquad\mbox{ as }
\epsilon\to 0.
\]

\begin{lemma} 
\label{ContinuityConvolution}
Let $T>0$. There exists a constant 
$C=C(T,\| \varphi_0 \|_\infty)>0$ such that, for all $0 \leq t \leq T$, 
for all $\epsilon$ small enough,
\begin{equation} 
\label{ContinuityConvolution1} 
\| \varphi_t(\cdot) - \rho^\epsilon_F*\varphi_t(\cdot) \|_\infty 
\leq C (I(\epsilon)+\epsilon^{1/4}). 
\end{equation}
Furthermore, there is a constant $\widetilde{C}(T, \| \varphi_0 \|_\infty) = \widetilde{C}$
such that, for all $0 \leq t \leq T$,
\begin{equation} 
\label{ContinuityConvolution2} 
\| F(\varphi_t(\cdot)) - F(\rho^\epsilon_F*\varphi_t(\cdot)) \|_\infty 
\leq \widetilde{C}\big(I(\epsilon)+\epsilon^{1/4}\big). 
\end{equation}
\end{lemma}
\begin{proof}
Let $\epsilon<\delta_0^4$, with $\delta_0$ from Lemma~\ref{ContinuityVarphi}. Then,
\begin{align*}
|\varphi_t(x) - \rho^\epsilon_F*\varphi_t(x)| & \leq 
\int_{\|x-y\| > \epsilon^{3/4} } \rho^\epsilon_F(x-y)|\varphi_t(y)-\varphi_t(x)| dy  \\ 
& + \int_{\|x-y\| \leq \epsilon^{3/4} } \rho^\epsilon_F(x-y)|\varphi_t(y)-\varphi_t(x)| dy \\ 
& \leq 2 M \int_{\|x-y\| > \epsilon^{3/4}  } \rho^\epsilon_F(x-y) dy 
+ \int_{\|x-y\| \leq \epsilon^{3/4}} \rho^\epsilon_F(x-y) K \epsilon^{1/4} dy \\
&\leq 2M I(\epsilon) + K \epsilon^{1/4},
\end{align*}
where we used 
the estimates of Lemma~\ref{BoundednessVarphis}
and Lemma~\ref{ContinuityVarphi}.
This proves (\ref{ContinuityConvolution1}). For (\ref{ContinuityConvolution2}),
let $L_M$ be the (uniform) Lipschitz constant of $F$ on $[0,M]$, with
$M$ still taken from Lemma~\ref{BoundednessVarphis}. Then,
\begin{align*}
 \| F(\varphi_t(\cdot)) - F(\rho^\epsilon_F*\varphi_t(\cdot)) \|_\infty 
&\leq L_M \| \varphi_t(\cdot)) - (\rho^\epsilon_F*\varphi_t(\cdot)) \|_\infty \\ 
& \leq L_M ( 2 M I(\epsilon) + K\epsilon^{1/4}),
\end{align*}
which proves (\ref{ContinuityConvolution2}). 
\end{proof}

\begin{proof}[Proof of Proposition \ref{prop:nonlocal_to_local}:]
Let $\epsilon$ be small enough that Lemma~\ref{ContinuityConvolution} holds. 
We use the notation $\widehat{\delta}(\epsilon)$ for the quantity on the right hand side 
of~(\ref{ContinuityConvolution2}).
Then from the representation~(\ref{FK:diferencesVarphis}) 
and Lemma~\ref{ContinuityConvolution} we can write,
\begin{align*}
| &\varphi_t(x) - \varphi^\epsilon_t(x)|  \\ 
& \leq \IE_x\left[ \int_0^t\Big| \varphi_s(Z_{t-s})
    F(\rho^\epsilon_F*\varphi_s(Z_{t-s}))
    -\varphi^\epsilon_s(Z_{t-s})
    F\big(\rho^\epsilon_F*\varphi^\epsilon_s(Z_{t-s})\big)\Big| ds \right] 
    +M t \widehat{\delta}(\epsilon)   \\ 
& \leq \IE_x\left[ \int_0^t \big|F(\rho^\epsilon_F*\varphi^\epsilon_s(Z_{t-s}))\big|\cdot
    \big|\varphi^\epsilon_s(Z_{t-s})-\varphi_s(Z_{t-s})\big| ds  \right] \\ 
    & \qquad +  \IE_x\left[\int_0^t |\varphi_s(Z_{t-s})| 
    \big|F(\rho^\epsilon_F*\varphi^\epsilon_s(Z_{t-s}))
    -F(\rho^\epsilon_F*\varphi_s(Z_{t-s}))\big| ds \right] 
    + M t \widehat{\delta}(\epsilon)  \\ 
& \leq \| F \|_M \int_0^t \| \varphi^\epsilon_s(\cdot) 
    - \varphi_s(\cdot) \|_\infty ds
      + M L_M \int_0^t \| \rho^\epsilon_F*\varphi^\epsilon_s(\cdot)
    -\rho^\epsilon_F*\varphi_s(\cdot) \|_\infty ds 
    +  M  t \widehat{\delta}(\epsilon)  \\ 
& \leq (\| F \|_M + M L_M) 
    \int_0^t \| \varphi^\epsilon_s(\cdot) - \varphi_s(\cdot) \|_\infty ds 
    + M  t \widehat{\delta}(\epsilon),
\end{align*}
where the second inequality is the triangle inequality, 
and the third is Lemma~\ref{BoundednessVarphis}. 
An application of Gronwall's inequality then yields,
\begin{align*}
\| \varphi^\epsilon_t(\cdot)-\varphi_t(\cdot)\|_\infty 
&\leq M  t \widehat{\delta}(\epsilon)  
\exp(t (\| F \|_M + M L_M)) \\ 
& \leq  M  T \widehat{\delta}(\epsilon)  \exp(T(\| F \|_M + M L_M)),
\end{align*}
	giving the result, since $\widehat{\delta}\to 0$ as $\epsilon\to 0$.
\end{proof}

\subsubsection{Porous Medium Equation}
\label{sec:pme}

In this subsection we prove Proposition~\ref{nonlocalPME to PME}.
To ease notation, we present
the proof in $d=1$ (although we retain the notation $\nabla$).
However, to recall the dependence on $\epsilon$ we write $\rho^\epsilon$ for $\rho_\gamma$.
It should be clear that it extends almost without change
to higher dimensions. 

Recall that 
we are concerned with non-negative solutions to the equation~(\ref{mollified equation}):
\begin{equation*}
	\partial_t\psi_t^\epsilon(x) =
	\Delta\left(\psi_t^\epsilon \, \rho^\epsilon *\psi_t^\epsilon\right)(x)
	+\psi_t^\epsilon(x)\left(1-\rho^\epsilon *\psi_t^\epsilon(x) 
\right) .
\end{equation*}
and 
we assume that $\rho=\zeta*\check{\zeta}$ with $\zeta$ a rapidly decreasing function
and $\check{\zeta}(x) = \zeta(-x)$.
The example we have in mind is $\zeta$ (and therefore $\rho$) being the density of
a mean zero Gaussian random variable.  
We shall prove that under the 
assumptions of Proposition~\ref{nonlocalPME to PME},
as $\epsilon\to 0$, we have
convergence to the solution to the porous medium equation with
logistic growth,
equation~\eqref{PME1}:
\begin{equation*}
	\partial_t \psi_t(x)=
	\Delta\left(\psi_t^2\right)(x)
	+\psi_t(x)\left(1-\psi_t(x)\right).
\end{equation*}

We work on the time interval $[0,T]$.
We will require a lower bound on $\int \psi_t^\epsilon(x)\log \psi_t^\epsilon(x) dx$ 
which we record as a lemma.

\begin{lemma}
\label{lower bound on ulogu}
	Suppose that there exists $\lambda\in (0,1)$ and $C<\infty$, both independent of
$\epsilon$, such that $\int \exp(\lambda |x|) \psi_0^\epsilon(x) dx<C$.
Then there exists
a constant $K<\infty$, independent of $\epsilon$, such that  
 $\int \psi_t^\epsilon(x)\log \psi_t^\epsilon (x)dx>-K$ for all $t\in [0,T]$. 
\end{lemma}
\begin{proof}
First observe that, since $x\log x$ is bounded below, 
$\int_{-1}^1\psi_t^\epsilon(x)\log\psi_t^\epsilon(x) dx$ is bounded below,
and recall that $\psi_t^\epsilon(x) \ge 0$.

Now consider 
\begin{multline}
\frac{d}{dt}\int \exp(\lambda x) \psi_t^\epsilon (x)dx
	= \int \exp(\lambda x)\Delta \big(\psi_t^\epsilon \, \rho^{\epsilon}*\psi_t^\epsilon\big)(x) dx
	\\
	+ \int \exp(\lambda x)\psi_t^\epsilon(x)\big(1-\rho^{\epsilon}*\psi_t^\epsilon(x)\big) dx
    \\
	= \int (\lambda^2-1)\exp(\lambda x) \psi_t^\epsilon(x) \smooth{\epsilon} \psi_t^\epsilon(x) dx 
	+ \int \exp(\lambda x)\psi_t^\epsilon(x) dx
    \\
    \leq \int \exp(\lambda x)\psi_t^\epsilon(x) dx,
\end{multline}
and so, by Gronwall's inequality, 
$\int \exp(\lambda x)\psi_t^\epsilon (x)dx$ is uniformly bounded on $[0,T]$.
In particular, combining with the Mean Value Theorem, we find 
\begin{equation*}
\int_x^{x+1}\psi_t^\epsilon(y)dy\leq C\exp(-\lambda x),
\end{equation*}
where the constant $C$ is independent of $x\geq1$.
A fortiori, 
\begin{equation}
\label{exp decay of u} 
\int_x^{x+1}\psi_t^\epsilon(y)\ind_{\psi_t^\epsilon(y)\leq 1}dy\leq C\exp(-\lambda x).
\end{equation}
Now the function $\psi\mapsto \ind_{0\leq \psi\leq 1}\psi|\log \psi|$ is concave, and so 
using Jensen's inequality and~(\ref{exp decay of u}),
\begin{equation*}
\int_x^{x+1}\psi_t^\epsilon(y)
|\log \psi_t^\epsilon (y)|\ind_{\psi_t^\epsilon(y)\leq 1}dy
\leq C'x \exp(-\lambda x).
\end{equation*}
Evidently a symmetric argument applies for $x\leq -1$.
Summing over $x$,
    and using that $\psi \log \psi \ge - \psi |\log \psi| \ind_{\psi \le 1}$,
we find
\begin{equation*}
\int \psi_t^\epsilon (x) \log \psi_t^\epsilon (x) dx
\geq -C''\sum_{x=1}^\infty x\exp(-\lambda x)>-K>-\infty,  
\end{equation*}
as required.
\end{proof}

\begin{proof}[Proof of Proposition~\ref{nonlocalPME to PME}]
First observe that
\begin{multline*}
	\int \psi_t^\epsilon(x)\, \rho^{\epsilon}*\psi_t^\epsilon(x) dx 
=
\int\int\int \psi_t^\epsilon(x)\psi_t^\epsilon(x-y)\zeta^{\epsilon}(y-z)
\check{\zeta}^{\epsilon}(z) dz dy dx 
\\
=
\int\int\int \psi_t^\epsilon(\tilde{x}-\tilde{z})
\psi_t^\epsilon(\tilde{x}-\tilde{y})
\zeta^{\epsilon}(\tilde{y})
\zeta^{\epsilon}(\tilde{z}) d\tilde{z} d\tilde{y} d \tilde{x}
	=\int\left(\zeta^{\epsilon}*\psi_t^\epsilon(x)\right)^2 dx,
\end{multline*}
where we have set $\tilde{x}=x-z$, $\tilde{y}=y-z$, $\tilde{z}=-z$.

Now note that
\begin{eqnarray*}
\frac{d}{dt}\int \psi_t^\epsilon (x) dx
	&=&
\int \Delta \big(\psi_t^\epsilon \, \rho^{\epsilon}*\psi_t^\epsilon\big)(x) dx 
	+\int \psi_t^\epsilon(x)\big(1-\rho^{\epsilon}*\psi_t^\epsilon(x)\big) dx
\\	
&=&
 \int \psi_t^\epsilon (x) d x
	-\int \big(\zeta^{\epsilon}*\psi_t^\epsilon(x)\big)^2  dx.
\end{eqnarray*}
Thus, Gronwall's inequality implies that $\int \psi_t^\epsilon(x) dx$ is uniformly bounded above in
$\epsilon$ and $t\in [0,T]$. Note that this also then gives a uniform
bound on the rate of change of $\int \psi_t^\epsilon (x) d x$, and since
we are working on $[0,T]$ this will be enough to give continuity in time of the
$L^1$ norm of the limit
when we pass to a convergent subsequence.

Now consider
\begin{eqnarray}
	\frac{d}{dt}\int \psi_t^\epsilon\log \psi_t^\epsilon dx &=&
	\int (1+\log \psi_t^\epsilon)\left[\Delta \big(\psi_t^\epsilon\,
	\rho^{\epsilon}*\psi_t^\epsilon\big)
	+\psi_t^\epsilon\big(1-\rho^{\epsilon}*\psi_t^\epsilon\big)\right]
dx
\nonumber
\\
&=&
	\int (1+\log \psi_t^\epsilon)\left[\nabla\Big(\psi_t^\epsilon\,
	\nabla(\rho^{\epsilon}*\psi_t^\epsilon)+
	\nabla \psi_t^\epsilon\, \rho^{\epsilon}*\psi_t^\epsilon\Big)
	+\psi_t^\epsilon\big(1-\rho^{\epsilon}*\psi_t^\epsilon\big)\right]  dx
\nonumber
	\\
&=&
\int \left[-\frac{\nabla \psi_t^\epsilon}{\psi_t^\epsilon}
\Big(\psi_t^\epsilon\, \nabla(\rho^{\epsilon}*\psi_t^\epsilon)
+\nabla \psi_t^\epsilon\, \rho^{\epsilon}*\psi_t^\epsilon\Big)+
(1+\log \psi_t^\epsilon)\psi_t^\epsilon(1-\rho^{\epsilon}*\psi_t^\epsilon)\right] d x
\nonumber
	\\
&=&
-\int \left(\nabla (\zeta^{\epsilon}*\psi_t^\epsilon)\right)^2  dx
-\int (\nabla \psi_t^\epsilon)^2\frac{\rho^{\epsilon}*\psi_t^\epsilon}{\psi_t^\epsilon}  dx
\nonumber
	\\
&&+\int
\left[\psi_t^\epsilon+\psi_t^\epsilon\log \psi_t^\epsilon
\big(1-\rho^{\epsilon}*\psi_t^\epsilon\big)
-\psi_t^\epsilon\, \rho^{\epsilon}*\psi_t^\epsilon\right] dx
\nonumber
	\\
&=&
-\int \left(\nabla (\zeta^{\epsilon}*\psi_t^\epsilon)\right)^2  dx
-\int (\nabla \psi_t^\epsilon)^2\frac{\rho^{\epsilon}*\psi_t^\epsilon}{\psi_t^\epsilon} dx
-\int (\zeta^{\epsilon}*\psi_t^\epsilon)^2  dx
\nonumber
	\\
&&+\int
\left[\psi_t^\epsilon+\psi_t^\epsilon\log \psi_t^\epsilon
\big(1-\rho^{\epsilon}*\psi_t^\epsilon\big)
\right] dx .
	\label{change in ulogu}
\end{eqnarray}
The first three terms are negative; and we already
saw that the $L^1$ norm
of $\psi_t^\epsilon$ is uniformly bounded.
Moreover, since $\psi_t^\epsilon\log \psi_t^\epsilon$ is uniformly bounded below
and $\int \rho^\epsilon(x) dx = 1$,
$$
    -\int \psi_t^\epsilon \log \psi_t^\epsilon \, \rho^{\epsilon}*\psi_t^\epsilon dx
    \leq
    C\int \rho^{\epsilon}*\psi_t^\epsilon dx
    =
    C \int \psi_t^\epsilon dx.
$$
From this and~(\ref{change in ulogu}),
we see immediately that $\int \psi_t^\epsilon \log \psi_t^\epsilon dx$
is uniformly bounded
above in $\epsilon$ and $t\in [0,T]$. 
Combining with Lemma~\ref{lower bound on ulogu},
we deduce that
we have a uniform bound on $\int \psi_s^\epsilon(x)|\log \psi_s^\epsilon(x)| d x$.
From \eqref{change in ulogu}, this in turn means that both
$\int_0^t \int(\zeta^{\epsilon}*\psi_s^\epsilon(x))^2  dx ds$ and
$\int_0^t \int \big(\nabla (\zeta^{\epsilon}*\psi_s^\epsilon(x))\big)^2  dx ds$
are uniformly bounded in
$\epsilon$ and $t\in [0,T]$.

We shall next show that $\zeta^\epsilon * \psi_t^\epsilon$ solves \eqref{PME1}
up to a remainder of order $\epsilon$.
First observe that
\begin{equation}
\label{total equation}
\int\Delta\left((\rho^{\epsilon}*\psi_t^\epsilon)\psi_t^\epsilon\right)\phi  dx
    =
    - \int \nabla(\rho^{\epsilon}*\psi_t^\epsilon)\, \psi_t^\epsilon\,\nabla\phi dx
    -\int \rho^{\epsilon}*\psi_t^\epsilon \,\nabla \psi_t^\epsilon \, \nabla\phi dx.
\end{equation}
We would like to show that this is close to $\int (\zeta^\epsilon * \psi^\epsilon_t)^2 \Delta \phi dx$.
For the first term
\begin{eqnarray}
\nonumber
\int \nabla(\rho^{\epsilon}*\psi_t^\epsilon)\, \psi_t^\epsilon\, \nabla\phi  dx
&=&
\int\int\int\nabla \psi_t^\epsilon(x-y)\zeta^{\epsilon}(y-z)
\check{\zeta}^{\epsilon}(z)
\psi_t^\epsilon(x)\nabla\phi(x) dz  dy  dx
\\
\nonumber
&=&
\int\int\int\nabla \psi_t^\epsilon(\tilde{x}-\tilde{y})
\zeta^{\epsilon}(\tilde{y})
\zeta^{\epsilon}(\tilde{z})\psi_t^\epsilon(\tilde{x}-\tilde{z})
\nabla\phi(\tilde{x}-\tilde{z}) d\tilde{z} d\tilde{y} d\tilde{x}
\\
\nonumber
&=&\int(\nabla \zeta^{\epsilon}*\psi_t^\epsilon)\,
\left(\zeta^{\epsilon}*(\psi_t^\epsilon\nabla\phi)\right)
d x\\
\nonumber
&=&
\frac{1}{2}\int\nabla ((\zeta^{\epsilon}*\psi_t^\epsilon)^2)
\,\nabla\phi  dx
\\
\label{error 1}
&&+
\int\nabla(\zeta^{\epsilon}*\psi_t^\epsilon)
\left[\zeta^{\epsilon}*(\psi_t^\epsilon\,\nabla\phi)-
\nabla\phi\,(\zeta^{\epsilon}*\psi_t^\epsilon)\right] dx,
\end{eqnarray}
where, as before, we have substituted $\tilde{x}=x-z$, $\tilde{y}=y-z$,
$\tilde{z}=-z$.
We are going to bound the square of the $L^2$-norm of the term in square brackets 
in~(\ref{error 1}) by the product of its $L^1$-norm (which is bounded
by a constant times the $L^1$-norm of $\psi_t^\epsilon$) and its $L^\infty$-norm.
To control the $L^\infty$-norm,
we use the Intermediate Value Theorem to see that
\begin{multline*}
\left|\int\left[\psi_t^\epsilon(x-y)\nabla\phi(x-y)\zeta^{\epsilon}(y)-
\nabla\phi(x)\psi_t^\epsilon(x-y)\zeta^{\epsilon}(y)
\right] dy\right|
\\
	\leq C\|\Delta\phi\|_\infty\int \psi_t^\epsilon (x-y)
	\epsilon\frac{|y|}{\epsilon}\zeta^{\epsilon}(y) d y
	\leq C\|\Delta\phi\|_\infty\epsilon \|z\zeta(z)\|_\infty\int \psi_t^\epsilon (x-y) dy.
\end{multline*}
Since $\zeta\in\mathcal{S}(\IR)$,
and (as we checked above)
$\psi_t^\epsilon$ is uniformly bounded in $L^1$ over $[0,T]$,
this expression is $\mathcal{O}(\epsilon)$ -- \revpoint{1}{31}
i.e., is bounded by a constant multiple of $\epsilon$
with a constant that does not depend on $t \in [0,T]$ or $x$.

We can now apply the Cauchy-Schwarz inequality to~(\ref{error 1}) to bound
it by the square root of
\begin{equation}
	C\epsilon\int \big(\nabla(\zeta^\epsilon*\psi_t^\epsilon(x))\big)^2 dx.
\end{equation}

Similarly, for the second term in~(\ref{total equation}),
\begin{eqnarray}
\nonumber
\int (\rho^{\epsilon}*\psi_t^\epsilon)\, \nabla \psi_t^\epsilon\, \nabla\phi dx
&=&
\int\int\int \psi_t^\epsilon(x-y)\zeta^{\epsilon}(y-z)
\check{\zeta}^{\epsilon}(z)\nabla \psi_t^\epsilon(x)\nabla\phi(x) dz dy  dx
\\
\nonumber
&=&
\int\int\int \psi_t^\epsilon(\tilde{x}-\tilde{y})\zeta^{\epsilon}(\tilde{y})
\zeta^{\epsilon}(\tilde{z})\nabla \psi_t^\epsilon(\tilde{x}-\tilde{z})
\nabla\phi(\tilde{x}-\tilde{z}) d \tilde{z} d\tilde{y} d\tilde{x}
\\
\nonumber
	&=&\int(\zeta^{\epsilon}*\psi_t^\epsilon)\,
\left(\zeta^{\epsilon}*(\nabla \psi_t^\epsilon \, \nabla\phi)\right)  dx\\
\nonumber
&=&
\frac{1}{2}\int\nabla ((\zeta^{\epsilon}*\psi_t^\epsilon)^2)
\, \nabla\phi  dx
\\
\label{error 2}
&&+
\int (\zeta^{\epsilon}*\psi_t^\epsilon)
\left[\zeta^{\epsilon}*(\nabla \psi_t^\epsilon\, \nabla\phi)-
\nabla\phi\,(\nabla \zeta^{\epsilon}*\psi_t^\epsilon)\right] dx.
\end{eqnarray}
The term in square brackets in~(\ref{error 2}) is 
\begin{multline*}
\left|\int\left[\nabla \psi_t^\epsilon(x-y)\nabla\phi(x-y)\zeta^{\epsilon}(y)
-\nabla\phi(x)\nabla \psi_t^\epsilon(x-y)\zeta^{\epsilon}(y)
\right] dy\right|
\\
=
	\left|\int\left[\psi_t^\epsilon(x-y)
	\nabla\big((\nabla\phi(x-y)-\nabla\phi(x))\zeta^{\epsilon}(y)\big)
\right] dy\right|,
\end{multline*}
and expanding $\nabla\phi$ in a Taylor series about $x$ and once again using that
$\zeta\in \mathcal{S}(\IR)$, we see that the $L^\infty$ norm of this quantity is
$\mathcal{O}(\epsilon)$.

We can now apply the Cauchy-Schwarz inequality to~(\ref{error 2}) to bound
it by the square root of
\begin{equation}
	C\epsilon\int \big(\zeta^\epsilon*\psi_t^\epsilon(x)\big)^2 dx.
\end{equation}

We now have the ingredients that we need. Recalling that 
$\int_0^t \int(\zeta^{\epsilon}*\psi_s^\epsilon(x))^2  dx ds$ and
$\int_0^t \int \big(\nabla (\zeta^{\epsilon}*\psi_s^\epsilon(x))\big)^2  dx ds$
are uniformly bounded in
$\epsilon$ and $t\in [0,T]$,
the calculations above yield both a uniform (in $\epsilon$) bound on
$\zeta^{\epsilon}*\psi_t^\epsilon$ in
$L^1\cap L^2\big([0,T]\times \IR\big)$, and (with another application of 
Cauchy-Schwarz, this time applied to the time integral, to control the error terms) that
\begin{multline} \label{eqn:weak_approx_pme}
\int \psi_t^\epsilon(x)\phi(x) dx-\int \psi_0^\epsilon (x)\phi(x) dx
	=\int_0^t\int (\zeta^{\epsilon}*\psi_s^\epsilon(x))^2\Delta\phi(x) dx
\\
	+\int_0^t \int \zeta^{\epsilon}*\psi_s^\epsilon(x)
	\left(1- \zeta^{\epsilon}*\psi_s^\epsilon(x)\right)
\phi(x) dx
	+\mathcal{O}(\sqrt{\epsilon})
\end{multline}
(for sufficiently regular $\phi$).
Since
$\int \psi_t^\epsilon(x)\phi(x) d x- \int \zeta^{\epsilon}*\psi_t^\epsilon (x)\phi(x) d x$
is order $\epsilon$,
if we replace $\psi^\epsilon$ by $\zeta^{\epsilon}*\psi^\epsilon$ on the left hand side,
then~\eqref{eqn:weak_approx_pme} says that $\zeta^{\epsilon}*\psi^\epsilon$
solves~\eqref{PME} weakly up to order $\epsilon$.
Therefore, $\zeta^{\epsilon}*\psi^\epsilon$ converges weakly to $\psi$ in $L^1$,
where $\psi$ is the (unique)
solution to equation~(\ref{PME})
and, so, therefore, does $\psi^\epsilon$. In fact, strong convergence, that is
$\int |\psi^\epsilon -\psi|\phi dx \to 0$, follows from the
uniform integrability of $\psi^\epsilon$ that we
can deduce from the uniform control of
$\int \psi^\epsilon|\log \psi^\epsilon| d x$ that we proved above.
\end{proof}

\section{Simultaneous scaling with interaction distance}
\label{subsec:one step convergence proof}

In this section we prove Theorem~\ref{thm:local_convergence},
which proves convergence in the case that the width of the interaction kernel $\rho_F$
simultaneously scales along with the parameters $\theta$ and $N$,
in the special case in which $r\equiv 1\equiv\gamma$, $q_{\theta}(x,dy)$
is isotropic with zero mean, the kernel
$\rho_F$ is Gaussian, and the scaling 
limit is a reaction-diffusion equation.

{\em To simplify notation, in this section we shall write 
$$\rho_\epsilon*\eta(x)=\rho_F^\epsilon*\eta(x)
=\langle p_{\epsilon^2}(x,y),\eta(dy)\rangle,$$
where $p_t(x,y)$ denotes the heat semigroup. 
The assumptions of Theorem~\ref{thm:local_convergence} will be in force throughout,
in particular,
\begin{equation}
    \label{conditions for one step}
	\epsilon^2\theta\to\infty, \qquad\mbox{and}\quad \frac{\theta}{N\epsilon^d}\to 0.
\end{equation}
That $N, \theta\to \infty$ and $\epsilon\to 0$ simultaneously will be 
implicit, so for example if we write $\lim_{\epsilon\to 0}$, it should be 
understood that $\theta, N\to\infty$ in such a way 
that~(\ref{conditions for one step}) is satisfied. 
Moreover, where there is no risk of confusion, except where it is helpful
for emphasis, we 
suppress dependence of $\eta$ on $N$.}

The first part of the proof mirrors that of Theorem~\ref{thm:nonlocal_convergence}:
in Subsection~\ref{bounds on rhoepsilon}
we establish bounds on the moments of $\rho_\epsilon*\eta_t(x)$ that are sufficient to
imply tightness and then apply standard results on 
convergence of Markov processes from~\cite{ethier/kurtz:1986}. 
The challenge
comes in identifying the limit points. This is much more intricate than the case in 
which we do not scale the interaction kernel, as weak convergence will no longer
be sufficient to guarantee the form of the nonlinear terms in the limiting equation.
Identification of the limit will rest on regularity inherited from continuity 
estimates for a random walk with Gaussian jumps which we prove in 
Subsection~\ref{continuity for random walk}, before
identifying the limit points in Subsection~\ref{limit in onestep case}.

Roughly speaking, the assumption that $\theta/N\epsilon^d$ is small \revpoint{1}{3}
is used in ensuring a well-defined and deterministic limit,
while the assumption on $\epsilon^2 \theta$ is used in proving continuity.
For more motivation behind these assumptions,
see the last part of Section~\ref{sec:one_step_heuristics}.

\subsection{Moment bounds for $\rho_\epsilon*\eta$}
\label{bounds on rhoepsilon}

Let us write $\mathcal{L}^\theta f(x) := \theta \int (f(y)-f(x))q_\theta(x,y) dy$ 
where $q_\theta$ is a Gaussian kernel of mean $0$ and variance $1/\theta$. 
We note that $\mathcal{L}^\theta$ is the generator of a continuous (time and 
space) random walk, which makes jumps of mean $0$ and variance $1/\theta$ at rate $\theta$.
In what follows we write $\psi_t^{\epsilon, x}(y)$ for the solution of
\begin{equation}
    \partial_t \psi_t^{\epsilon,x}
    =
    \mathcal{L}^\theta \psi_t^{\epsilon, x},
    \label{AlmostHeatEquation}
\end{equation}
with initial condition 
$\psi_0^{\epsilon,x}(y) = \rho_\epsilon(y-x) =p_{\epsilon^2}(x,y)$.

To see why $\psi_t^{\epsilon,x}$ is useful,
first note that for any time-dependent function $\phi_t(x)$
with time derivative $\dot \phi_t(x) = \partial_t \phi_t(x)$,
\begin{multline}
\label{time dep mg prob}
    \langle \phi_t(x), \eta_t(dx) \rangle
    = 
    \langle \phi_0(x), \eta_0(dx) \rangle
    +
    M_t(\phi)
    +
    \int_0^t \big\langle \mathcal{L}^\theta \phi_s(x) 
        + \dot \phi_s(x), \eta_s(dx) \big\rangle ds
    \\  {}
    +
    \int_{0}^t \big\langle \phi_s(x) F(x, \eta_s) , \eta_s(dx) \big\rangle ds ,
\end{multline}
where $M_t(\phi)$ is a martingale (with respect to the natural filtration)
with angle bracket process given by~\eqref{eqn:prelimit_martingale_variation}
with $f$ replaced by $\phi_s(\cdot)$.
So, taking $\phi_s(\cdot)=\psi_{t-s}^{\epsilon,x}(\cdot)$ for $0 \le s \le t$,
\begin{eqnarray}
    \nonumber
    \rho_\epsilon*\eta_t(x)&=&\langle\psi^{\epsilon,x}_0(y),\eta_t(dy)\rangle
    \\
    &=&\langle\psi^{\epsilon,x}_t(y),\eta_0(dy)\rangle
    +\int_0^t\big\langle\psi^{\epsilon,x}_{t-s}(y)
    F\big(\rho_\epsilon*\eta_s(y)\big), \eta_s(dy)\big\rangle ds
    +M_t(x),
    \label{expn rhoepsilon} 
\end{eqnarray}
where $M_t(x)$ has mean zero and a second moment we can easily write down.

\begin{lemma} \label{PsiBoundHS}
    Fix $t>0$,
    let $(\Pi(s))_{s \ge 0}$ be a rate one Poisson process,
    and let $T(t) = \Pi(\theta t) / \theta$.
    Then
    \[
        \psi^{\epsilon,x}_t(y)
        =
        \IE\left[ p_{\epsilon^2+T(t)}(x,y)\right],
    \]
    and, moreover,
    since under our assumptions $\theta \epsilon^2$ is bounded below,
    there is a $C$ independent of $\epsilon$ or $t$ such that
    \[
        \| \psi^{\epsilon,x}_t \|_\infty
        \leq
        \frac{C}{(\epsilon^2 + t)^{d/2}} .
    \]
\end{lemma}

\begin{proof}
    The first claim is immediate from the definition of the random walk 
with generator $\mathcal{L}^\theta$.

For the second claim, first define $\tau(t) = T(t)-t$.
Since if $\tau(t) \ge -(\epsilon^2 + t)/2$,
then $1/(\epsilon^2 + T(t)) \le 2/(\epsilon^2 + t)$,
while $\epsilon^2 + T(t) \ge \epsilon^2$ always,
partitioning over
$\{ \tau(t) \ge -(\epsilon^2 + t)/2 \}$ and its complement,
\begin{align}
   \| \psi^{\epsilon,x} \|_\infty
    &=
        \IE\left[ 
            \frac{1}{\big(2 \pi (\epsilon^2+T(t))\big)^{d/2}}
        \right] \nonumber
    \\ & \le
        \frac{C}{(\epsilon^2 + t)^{d/2}}
        + 
        \frac{C}{\epsilon^d}
        \IP\left\{
            \tau(t) < - (\epsilon^2 + t) / 2
        \right\}. \label{eqn:psi_infty_bound}
\end{align}
Now, observe that since $\IE[e^{-\Pi(\theta t)}] = \exp(-\theta t (1 - e^{-1}))$,
by Markov's inequality,
\begin{align}
    \IP\left\{
        \tau(t) < - \frac{\epsilon^2 + t}{2}
    \right\}
    &=
    \IP\left\{
        e^{-\Pi(\theta t)} > e^{-\theta (t - \epsilon^2) / 2}
    \right\}
\nonumber
\\
&\leq \frac{\IE[\exp\big(-\Pi(\theta t)\big)]}{\exp\big(-\theta(t-\epsilon^2)/2\big)}
\nonumber
    \\&=
    \frac{
        \exp(- \theta t(1 - e^{-1}))
    }{
        \exp(- \theta (t - \epsilon^2) / 2)
    }
\nonumber
    \\&=
    \exp\left\{ - \chi \theta t - \frac{\theta \epsilon^2}{2} \right\} ,
\label{bound for negative tau}
\end{align}
where $\chi = 1/2 - e^{-1} > 0$.
The second term in \eqref{eqn:psi_infty_bound} is therefore bounded by
\[
	C \left(1 + \frac{t}{\epsilon^2}\right)^{d/2} e^{-\chi \theta t }
	\frac{1}{(\epsilon^2 + t)^{d/2}}e^{-\epsilon^2 \theta/2} .
\]
Now observe that the derivative (with respect to $t$) of 
	$e^{-\chi \theta t} 
	(1 + t/\epsilon^2)^{d/2}$ 
is
\[
    \left( \frac{d}{2 \epsilon^2} - \left(1 + \frac{t}{\epsilon^2}\right) \chi \theta \right)
    \left( 1 + \frac{t}{\epsilon^2} \right)^{d/2 - 1} e^{-\chi \theta t },
\]
which is negative if $\theta(\epsilon^2 + t) > d/2 \chi$.
At the maximum, 
$(1 + t/\epsilon^2) = d /(2\chi \theta \epsilon^2)$,
and so this quantity is bounded uniformly over not only $t$ but also $\epsilon$
(since we've assumed that $\theta \epsilon^2$ is bounded below).
Therefore, 
we have the bound
\begin{equation}
	\label{small tau bound}
	\frac{1}{\epsilon^d}
        \IP\left\{
            \tau(t) < - (\epsilon^2 + t) / 2
        \right\}
\leq \frac{C}{(\epsilon^2 + t)^{d/2}} e^{-\epsilon^2\theta/2}.
\end{equation}
Substituting this into \eqref{eqn:psi_infty_bound} yields the result.
\end{proof}

\begin{lemma}
\label{bounds on moments}
Let $\{{\cal F}_t\}_{t\geq 0}$ denote the natural filtration. 
Under the assumptions of Theorem~\ref{thm:local_convergence},
for each $T\in [0,\infty)$, and $k\in\IN$, there exist constants $C=C(k,T)$ and
$\widetilde{C}=\widetilde{C}(k,T)$, independent of $\epsilon$, such that for all 
$x\in\IR^d$ and all $u,t\in [0,T]$ with $u<t$,
\begin{equation}
\label{moment bound rhoepsilon}
\IE\Big[\left.\big(\rho_\epsilon*\eta_t(x)\big)^k\right| {\cal F}_u\Big]
\leq C\langle\psi_{t-u}^{\epsilon,x}(z), \eta_u(dz)\rangle^k
+C\frac{\theta}{N\epsilon^d}
\langle\psi_{t-u}^{\epsilon,x}(z), \eta_u(dz)\rangle;
\end{equation} 
and
\begin{multline}
\label{flandoli trick for rhoepsilon}
\IE\Big[\left. \int_u^t\langle\psi_{t-s}^{\epsilon,x}(z),\eta_s(dz)\rangle^{k-1}
\big\langle\psi_{t-s}^{\epsilon,x}(z)|F(\rho_\epsilon*\eta_s(z))|,
\eta_s(dz)\big\rangle ds\right| {\cal F}_u\Big]
\\
\leq\widetilde{C}\langle\psi_{t-u}^{\epsilon,x}(z), \eta_u(dz)\rangle^k
+\widetilde{C}\frac{\theta}{N\epsilon^d}
\langle\psi_{t-u}^{\epsilon,x}(z), \eta_u(dz)\rangle;
\end{multline}
where the function $\psi^{\epsilon,x}_t(\cdot)$ was
defined in~(\ref{AlmostHeatEquation}).
In particular, 
under the assumptions of Theorem~\ref{thm:local_convergence}, the
expected values of the
quantities on the right hand side of~(\ref{moment bound rhoepsilon})
and~(\ref{flandoli trick for rhoepsilon}) are both integrable 
with respect to Lebesgue measure.  
\end{lemma}

\begin{proof}
To simplify our expressions, we shall consider the case $u=0$, but
the proof goes through unchanged for other values of $u$.

We proceed by induction. 
Taking expectations in~(\ref{expn rhoepsilon}), using that $F$ is bounded above, and 
applying Gronwall's inequality to 
$\langle\psi^{\epsilon,x}_{t-s}, \eta_s\rangle$ 
we obtain 
$\IE[\langle \psi_0^{\epsilon,x},\eta_t\rangle] \le C \IE[\langle \psi_t^{\epsilon,x},\eta_0\rangle]$,
which implies~(\ref{moment bound rhoepsilon}) in the case $k=1$.
Moreover, rearranging~(\ref{expn rhoepsilon}) we find 
\begin{equation}
-\int_0^t\big\langle\psi^{\epsilon,x}_{t-s}(y)
F(\rho_\epsilon*\eta_s(y), \eta_s(dy)\big\rangle ds
=\langle\psi^{\epsilon,x}_t(y),\eta_0(dy)\rangle
-\langle \psi^{\epsilon,x}_0(y),\eta_t(dy)\rangle +M_t(x),
\end{equation} 
and taking expectations again, 
since $\langle\psi^{\epsilon,x}_0,\eta_t\rangle>0$, and $M_0(x)=0$,
this yields 
$$\IE\Big[-\int_0^t\langle\psi^{\epsilon,x}_{t-s}(y)F(\rho_\epsilon*\eta_s(y)), 
\eta_s(dy)\rangle ds\Big|{\cal F}_0\Big]
\leq \langle\psi^{\epsilon, x}_t(y),\eta_0(dy)\rangle.$$
Since $F$ is bounded above, there exists a constant $K$ such that $|F|\leq K-F$ and so 
combined with the bound on $\IE[\langle\psi^{\epsilon, x}_0(y),\eta_t(dy)\rangle]$ just
obtained, this in turn yields 
$$\IE\Big[\int_0^t\big\langle\psi^{\epsilon,x}_{t-s}(y)
|F(\rho_\epsilon*\eta_s(y))|, 
\eta_s(dy)\big\rangle ds\Big|{\cal F}_0\Big]\leq 
\widetilde{C}\langle\psi^{\epsilon, x}_t(y),\eta_0(dy)\rangle,$$
which is~(\ref{flandoli trick for rhoepsilon})
in the case $k=1$.

Now suppose that we have established~(\ref{moment bound rhoepsilon}) 
and~(\ref{flandoli trick for rhoepsilon}) for all exponents $j<k$.
First we apply the generator $\Pgen^N$ of our scaled population process
to functions of the form $\langle f,\eta\rangle^k$.
Recalling that each jump of the process involves the birth or death of a single 
individual, and so increments $\langle f,\eta\rangle$ by $\pm f/N$ at the location of
that individual and that $r \equiv \gamma \equiv 1$, we find 
\begin{multline}
\label{pgen applied to kth moment}
\Pgen^N\Big(\langle f,\eta\rangle^k\Big)
=
\Big\langle
\int\theta N
    \sum_{j=1}^k \binom{k}{j} \frac{f(y)^j}{N^j} \langle f, \eta\rangle^{k-j}
    q_\theta(x,dy),\eta(dx)\Big\rangle
\\
+
    \Big\langle\theta N\Big(1-\frac{F(\rho_\epsilon*\eta(x))}{\theta}\Big)
    \sum_{j=1}^k \binom{k}{j} (-1)^j \frac{f(x)^j}{N^j} \langle f, \eta\rangle^{k-j}
    ,\eta(dx)\Big\rangle.
\end{multline}
Mimicking what we did above, we set $f(\cdot)=\psi^{\epsilon,x}_t(\cdot)$ and 
write
\begin{multline}
\label{kth moment for time varying function}
\IE\Big[\langle \psi^{\epsilon,x}_0,\eta_t\rangle^k\Big|{\cal F}_0\Big]
=\langle\psi^{\epsilon, x}_t(y),\eta_0(dy)\rangle^k
+\IE\Big[\int_0^t\Pgen^N\big(\langle\psi^{\epsilon,x}_{t-s}(y),\eta_s(dy)\rangle^k\big) ds
\\
-\int_0^t\langle k\dot{\psi}^{\epsilon,x}_{t-s}(y), \eta_s(dy)\rangle
\big\langle \psi^{\epsilon, x}_{t-s}(y),\eta_s(dy)\big\rangle^{k-1}ds
\Big|{\cal F}_0\Big].
\end{multline}
Since $\dot \psi^{\epsilon,x}_s = \mathcal{L}_\theta \psi^{\epsilon,x}_s$,
the $j=1$ term from $\Pgen^N(\langle \psi^{\epsilon,x}_{t-s}(y),\eta_s(dy) \rangle^k)$
combines with the last term in~\eqref{kth moment for time varying function}
to yield 
$$ \int_0^t k \langle \psi^{\epsilon,x}_{t-s},\eta\rangle^{k-1} \langle F(\rho_\epsilon*\eta_s(y)) \psi^{\epsilon,x}_{t-s}(y), \eta_s(dy) \rangle ds . $$
As for the remaining terms, using (from Lemma~\ref{PsiBoundHS}) 
that $\sup_s\|\psi^{\epsilon,x}_s(\cdot)\|_\infty=C/\epsilon^d$,
$N\epsilon^d>1$,
and our inductive
hypothesis, we find
\begin{multline*}
\IE\Big[
\Big\langle
\int_0^t\theta N\sum_{j=2}^{k}\binom{k}{j}\int
\frac{\psi_{t-s}^{\epsilon,x}(z)^j}{N^j}
\langle \psi^{\epsilon,x}_{t-s},\eta_s\rangle^{k-j}
q_\theta(y,dz),
\eta_s(dy)\Big\rangle ds 
\\
+
\Big\langle \int_0^t\theta N\sum_{j=2}^{k}\binom{k}{j}\frac{\psi_{t-s}^{\epsilon,x}(y)^j}{N^j}\langle \psi^{\epsilon,x}_{t-s},\eta_s\rangle^{k-j}
(-1)^j\Big(1-\frac{F(\rho_\epsilon*\eta_s(y))}{\theta} 
\Big),
\eta_s(dy)\Big\rangle ds\Big|{\cal F}_0\Big]
\\
\leq 
C\IE\Big[\Big\langle\int_0^t\sum_{j=2}^k\frac{\theta}{N\epsilon^d}
\Big(\frac{1}{(N\epsilon^d)^{j-2}}\Big)
	\Big\langle \psi^{\epsilon,x}_{t-s}(y)
\Big(2+\frac{|F(\rho_\epsilon*\eta_s(y))|}{\theta}\Big),\eta_s(dy)\Big\rangle
\langle\psi^{\epsilon,x}_{t-s},\eta_s\rangle^{k-j}ds\big|{\cal F}_0\Big]
\\
\leq
C'\frac{\theta}{N\epsilon^d}\sum_{j=1}^{k-1}
\langle\psi^{\epsilon, x}_t(y),\eta_0(dy)\rangle^j
\leq
C''\frac{\theta}{N\epsilon^d}
\Big(\langle\psi^{\epsilon, x}_t(y),\eta_0(dy)\rangle^k
+\langle\psi^{\epsilon, x}_t(y),\eta_0(dy)\rangle\Big).
\end{multline*}
Combining this with~(\ref{pgen applied to kth moment})
and~(\ref{kth moment for time varying function}), using once again the fact that $F$
is bounded above, we find 
\begin{multline*}
\label{kth step of induction}
\IE\Big[\langle \psi^{\epsilon,x}_0,\eta_t\rangle^k\Big|{\cal F}_0\Big]
\leq
\langle\psi^{\epsilon, x}_t(y),\eta_0(dy)\rangle^k
+\tilde{C}\IE\Big[\int_0^t
\langle\psi^{\epsilon, x}_{t-s}(y),\eta_s(dy)\rangle^k ds
\Big|{\cal F}_0\Big]
\\
+C''\frac{\theta}{N\epsilon^d}
\Big(\langle\psi^{\epsilon, x}_t(y),\eta_0(dy)\rangle^k+
\langle\psi^{\epsilon, x}_t(y),\eta_0(dy)\rangle\Big)
,
\end{multline*}
and~(\ref{moment bound rhoepsilon}) 
follows from Gronwall's inequality. Rearranging exactly as in the case $k=1$, we 
recover~(\ref{flandoli trick for rhoepsilon}) 
and the inductive step is complete.
\end{proof}

We shall also need the following consequence of the bounds that we obtained in 
Lemma~\ref{bounds on moments}:
\begin{corollary}
\label{alternative form moment bounds}
Under the assumptions of Theorem~\ref{thm:local_convergence},
for each $k\geq 1$, $T>0$, there is a $C(k,T)$ such that
\begin{equation}
\label{double integrals with respect to eta}
\IE\Big[\big\langle (\rho_\epsilon*\eta_t)^{k},\eta_t\big\rangle \Big]
<C(k,T)<\infty, \qquad\mbox{ for all }t\in [0,T].
\end{equation}
\end{corollary}
\begin{proof}[Sketch]
First observe that if $A\in (0,1)$, then 
\begin{equation}
\label{heat equation at smaller time}
p_{A\epsilon^2}(x,y)
=\frac{1}{A^{d/2}}p_{\epsilon^2}(x,y)\exp\Big(-\frac{\|x-y\|^2}{2\epsilon^2}
\big(\frac{1}{A}-1\big)\Big)\leq \frac{1}{A^{d/2}}p_{\epsilon^2}(x,y).
\end{equation}
Now consider
    \begin{align*}
\IE\big[\langle \rho_\epsilon*\eta_t(x),\eta_t(dx)\rangle\big]&=
\IE\Big[\int\int p_{\epsilon^2}(x,z)\eta_t(dz)\eta_t(dx)\Big]
\\        
&=
\IE\Big[\int\int\int p_{\epsilon^2/2}(x,y) 
p_{\epsilon^2/2}(y,z)dy\eta_t(dz)\eta_t(dx)\Big]
        \\ &=
       \IE\Big[ \int \left( p_{\epsilon^2/2} * \eta_t(y) \right)^2 dy \Big]
\\&\leq C\int \IE\big[\big(\rho_\epsilon*\eta_t(x)\big)^2\big] dx,
    \end{align*}
where we used~(\ref{heat equation at smaller time}) in the last line.
Using Lemma~\ref{bounds on moments}
and our assumptions on $\eta_0$, this quantity is finite.

To illustrate the inductive step, now consider
\begin{multline}
\label{rewrite for inductive step}
\IE\big[\langle \rho_\epsilon*\eta_t(x)^2,\eta_t(dx)\rangle\big] 
=
\IE\Big[\int\int\int p_{\epsilon^2}(x,z_1) p_{\epsilon^2}(x,z_2)
\eta_t(dz_1)\eta_t(dz_2)\eta_t(dx)\Big]
\\
=
\IE\Big[\int\cdots\int p_{\epsilon^2/2}(x,y_1) 
p_{\epsilon^2/2}(x,y_2)p_{\epsilon^2/2}(y_1,z_1)
p_{\epsilon^2/2}(y_2,z_2)
\eta_t(dz_1)\eta_t(dz_2)dy_1dy_2\eta_t(dx)\Big].
\end{multline}
We use the identity
\[
p_{\epsilon^2/2}(x,y_1) p_{\epsilon^2/2}(x,y_2)
=p_{\epsilon^2}(y_1,y_2)p_{\epsilon^2/4}\Big(x, \frac{y_1+y_2}{2}\Big)
\]
to rewrite~(\ref{rewrite for inductive step}) as
\begin{align*}
&\IE\Big[\int \int p_{\epsilon^2/2}*\eta_t(y_1)\, p_{\epsilon^2/2}*\eta_t(y_2)\,
p_{\epsilon^2/4}*\eta_t\big(\frac{y_1+y_2}{2}\big)\, 
p_{\epsilon^2}(y_1,y_2)dy_1dy_2\Big]
\\
&\leq 
\IE\Big[\int \int\Big\{
\big(p_{\epsilon^2/2}*\eta_t(y_1)\big)^3+ \big(p_{\epsilon^2/2}*\eta_t(y_2)\big)^3+
\big(p_{\epsilon^2/4}*\eta_t\big(\frac{y_1+y_2}{2}\big)\big)^3
\Big\}p_{\epsilon^2}(y_1,y_2)dy_1dy_2\Big],
\end{align*}
where we have used that for any non-negative real numbers
$\beta_1$, $\beta_2$, $\beta_3$, 
$\beta_1\beta_2\beta_3\leq \beta_1^3+\beta_2^3+\beta_3^3$.
For the first two terms in the sum we integrate with respect to $y_2$ and $y_1$
respectively to reduce to an expression of the form considered in 
Lemma~\ref{bounds on moments}. For the final term, the change of variables
$z_1=y_1+y_2$, $z_2=y_1-y_2$ in the integral similarly 
allows us to integrate out the heat kernel, and we conclude that the result
holds for $k=2$.

We can proceed in the same way for larger values of $k$, using repeatedly that
\[
p_{t_1}(x,y_1)p_{t_2}(x,y_2)=p_{\frac{t_1t_2}{t_1+t_2}}
\Big(x,\frac{t_2y_1+t_1y_2}{t_1+t_2}\Big)p_{t_1+t_2}(y_1,y_2)
\]
to write 
\[
\prod_{j=1}^k p_{\tau}(y,y_j)
=\prod_{j=2}^{k}p_{\frac{j\tau}{j-1}}\big(y_j,Y_{j-1}\big)
p_{\frac{\tau}{k}}(y, Y_k)
\]
where
\[
    Y_1=y_1, \qquad Y_j=\frac{j-1}{j}Y_{j-1}+\frac{1}{j}y_j,  \mbox{for }j\geq 2.
\]
Writing $p_{\epsilon^2}(x,z_j) = \int p_{\epsilon^2/2}(x,y_j) p_{\epsilon^2/2}(y_j,z_j) dy_j$
and using the above with $\tau = \epsilon^2/2$,
this yields
\begin{multline*}
\big\langle\big(\rho_\epsilon*\eta_t(x)\big)^k,\eta_t(dx)\big\rangle
    =\int\cdots\int \prod_{j=2}^k p_{\epsilon^2 j/2(j-1)}(y_j,Y_{j-1})
     \prod_{i=1}^k p_{\epsilon^2/2}*\eta_t(y_i)p_{\epsilon^2/2k}*\eta_t(Y_k)
dy_1\ldots dy_k
\\
\leq
\int\cdots\int 
    \prod_{j=2}^k p_{\epsilon^2 j/2(j-1)}(y_j,Y_{j-1})
    \Big\{\sum_{i=1}^k\big(p_{\epsilon^2/2}*\eta_t(y_i)\big)^{k+1}
    + \big(p_{\epsilon^2/2k}*\eta_t(Y_k)\big)^{k+1}\Big\}dy_1\ldots dy_k,
\end{multline*}
and once again we can change variables in the integrals and 
use~(\ref{heat equation at smaller time}) to bound this by a constant 
multiple of $\int\IE\big[\big(\rho_\epsilon*\eta_t(x)\big)^{k+1}\big]dx$, and
the inductive step is complete.
\end{proof}

\begin{corollary}[Tightness of $\{(\rho_\epsilon*\eta_t^N(x) dx)_{t\geq 0}\}$]
\label{lem:density_tightness} 
Under the assumptions of Theorem~\ref{thm:local_convergence},
the sequence of measure valued processes $\{ \rho_\epsilon*\eta_t^N(x) dx \}_{t \geq 0}$
	(taking values in ${\mathcal D}_{[0,T]}(\measures)$) is tight.
\end{corollary}
\begin{proof}

First observe that the proof, from Lemma~\ref{lem:eta_compact_containment}, 
that 
$\IE[\sup_{0\leq t\leq T} \langle 1,\eta_t^N\rangle]$ is 
bounded goes through unchanged, and since 
$\langle 1, \rho_\epsilon*\eta_t^N(x)dx\rangle=\langle 1, \eta_t^N\rangle$,  
compact containment follows.
 
As in the nonlocal case, it suffices to prove that for $T>0$,
and any $f\in C_b^\infty (\IR^d)$ with bounded second derivatives
and $\int|f(x)|dx < \infty$, the sequence of real-valued processes 
$\big\{\big(\int f(x)\rho_\epsilon*\eta_t^N(x)dx\big)_{t\geq 0}\big\}_{N\geq 1}$ is tight.  
Let us temporarily write $X_f^N(t)$ for $\int f(x)\rho_\epsilon*\eta_t^N(x)dx$
and set
\[
w'\big(X_f^N,\delta,T\big)= \inf_{\{t_i\}}\max_i\sup_{s,t\in [t_{i-1},t_i)}
\big| X_f^N(t)-X_f^N(s)\big|,
\]
where $\{t_i\}$ ranges over all partitions of the form 
$0=t_0<t_1<\cdots <t_{n-1}<T\leq t_n$ with 
$\min_{1\leq i\leq n}(t_i-t_{i-1})>\delta$ and $n\geq 1$.
Using Corollary~3.7.4 of~\cite{ethier/kurtz:1986}, 
to prove tightness
of the sequence of real-valued processes $X_f^N$
it suffices to check compact containment of 
    the sequence $\{\int f(x) \rho_\epsilon*\eta_t^N(x)dx\}_{N\geq 1}$ at any rational time $t$ and
that for every $\nu>0$ and $T>0$, there exists $\delta>0$ such that
\[
\limsup_{N\to\infty}\IP\big[w'\big(X_f^N,\delta,T\big)>\nu\big]<\nu.
\]
Evidently this will follow 
if we can show that this condition is satisfied when we replace the 
minimum over all partitions with mesh at least $\delta$ in the 
definition of $w'$, by the partition into
intervals of length exactly $\delta$.

We have
\begin{multline}
\label{tightness estimate}
\left|
\langle \rho_{\epsilon} * f, \eta_t^N \rangle 
- \langle \rho_{\epsilon} * f, \eta_s^N \rangle            
\right|
\leq
\left|
\int_s^t\Big\langle\theta\int\big(\rho_\epsilon*f(y)-\rho_\epsilon*f(x)\big)
q_{\theta}(x,dy),\eta_u^N(dx)\Big\rangle du
\right|
\\
+\int_s^t\Big\langle |F\big(\rho_\epsilon*\eta_u^N(x)\big)|\rho_\epsilon*|f|(x), 
\eta_u^N(dx)\Big\rangle du +2\sup_{0\leq u\leq T}|\widehat{M}^N(f)_u|,
\end{multline}
where $\widehat{M}^N(f)$ is the martingale of~(\ref{eqn:eta_f_mgale_decomp})
with the test function $f$ replaced by $\rho_\epsilon*f$. 
We control each of the three terms on the right hand side separately.

By the Intermediate Value Theorem, using $T_t$ to denote the heat semigroup,
there exists $s\in (0,1/\theta)$ such that
\begin{multline*}
\left|
\theta\int\big(\rho_\epsilon*f(y)-\rho_\epsilon*f(x)\big)
q_{\theta}(x,dy) \right|
=
\left|\theta\Big(T_{\epsilon^2+1/\theta}f(x)-T_{\epsilon^2}f(x)\Big)\right|
\\
=
\left|\partial_sT_{\epsilon^2+s}f(x)\right|=\left|T_{\epsilon^2+s}\Delta f(x)\right|
\leq\|\Delta f\|_\infty.
\end{multline*}
The first term in~\eqref{tightness estimate} is therefore bounded by
$$
    \| \Delta f \|_\infty |t-s| \sup_{0 \leq u \leq T } \langle 1, \eta_u^N \rangle .
$$

We follow the approach of Lemma~\ref{lem:eta_compact_containment}.
Consulting~\eqref{eqn:prelimit_martingale_variation},
the angle bracket process of $\widehat{M}^N(f)$ satisfies
$ \IE[\langle\widehat{M}^N_f\rangle_T] \le C (\theta/N) \int_0^T \IE[\langle 1, \eta_s \rangle] ds \le C' \theta/N$
for some constants $C$ and $C'$.
Now, using the Burkholder-Davis-Gundy inequality
and the same fact as before from \citet{barlow/jacka/yor:1986}, \revpoint{1}{32}
$\IE[\sup_{0 \le u \le T} |\widehat{M}^N(f)_u|^2] \le C'' \IE[\langle\widehat{M}^N(f)\rangle_T]$,
and so using Markov's inequality,
\begin{equation}
\label{martingale term to zero}
    \limsup_{N\to\infty}\IP\Big[
        2\sup_{0\leq u\leq T}|\widehat{M}^N(f)_u|
        >\frac{\nu}{3}
    \Big]
    \leq
    \limsup_{N\to\infty} \frac{36}{\nu^2} C'' \IE\big[\langle\widehat{M}^N(f)\rangle_T\big]
    \leq
    \limsup_{N\to\infty}\frac{36}{\nu^2}\frac{C' C''\theta}{N}
    = 0 .
\end{equation}

Now consider
\begin{multline}
	\label{square of F-integral}
\IE\Big[
	\Big(\int_s^t\big\langle\rho_\epsilon*|f|(x) \big|F\big(\rho_\epsilon*\eta_u^N(x)\big)\big|,
\eta_u^N(dx)\big\rangle du\Big)^2\Big]
\\
=
2\IE\Big[
	\int_s^t\big\langle\rho_\epsilon*|f|(x) \big|F\big(\rho_\epsilon*\eta_u^N(x)\big)\big|,
\eta_u^N(dx)\big\rangle 
	\int_u^t\big\langle \rho_\epsilon*|f|(x)\big|F\big(\rho_\epsilon*\eta_r^N(x)\big)\big|,
\eta_r^N(dx)\big\rangle dr
du\Big].
\end{multline}
Since $F$ is polynomial, we use the approach
of Corollary~\ref{alternative form moment bounds}, the tower 
property, and 
Lemma~\ref{bounds on moments}, to bound this
in terms of sums of terms of the form
\[
\IE\Big[\int_s^t(t-u)\int\rho_\epsilon*|f|(x)\rho_\epsilon*\eta_u^N(x)^jdx
\int\rho_\epsilon*|f|(y)\rho_\epsilon*\eta_u^N(y)^kdydu\Big].
\]
Now observe that, again using Lemma~\ref{bounds on moments}, 
since for nonnegative $a$ and $b$, $a^j b^k \le a^{j+k} + b^{j+k}$,
\begin{multline*}
	\IE\left[	\int\int\rho_\epsilon*|f|(x)\rho_\epsilon*\eta_u^N(x)^j
	\rho_\epsilon*|f|(y)\rho_\epsilon*\eta_u^N(y)^k dx dy\right]
	\\
	\leq
	\IE\left[\int\int\|f\|_\infty\rho_\epsilon*\eta_u^N(x)^{j+k}
	\rho_\epsilon*|f|(y)
	dx dy
+
	\int\int
	\rho_\epsilon*|f|(x)
	\|f\|_\infty
	\rho_\epsilon*\eta_u^N(y)^{j+k}
	dx dy \right]
	\\ \leq C\int |f|(x)dx.
\end{multline*}
Thus the quantity~(\ref{square of F-integral}) 
is bounded by $C(t-s)^2$ for a new constant $C$ which we can 
take to be independent of $s$, $t$ and $\epsilon$. Markov's inequality 
then gives
\[
\IP\Big[\|f\|_\infty\int_s^t
\big\langle \big|F\big(\rho_\epsilon*\eta_u^N(x)\big)\big|,
\eta_u^N(dx)\big\rangle du \geq \frac{\nu}{3}\Big]
\leq C\frac{(t-s)^2}{\nu^2}.
\]
A union bound gives that
\begin{equation}
\label{union bound}
\IP\Big[\max_i\|f\|_\infty\int_{t_{i-1}}^{t_i}
\big\langle \big|F\big(\rho_\epsilon*\eta_u^N(x)\big)\big|,
\eta_u^N(dx)\big\rangle du \geq \frac{\nu}{3}\Big]
\leq C\frac{T\delta}{\nu^2}.
\end{equation}
Now using Markov's inequality, we can choose $K$ so that
\[
\IP\Big[\|\Delta f\|_\infty\, \sup_{0\leq t\leq T} \langle 1,\eta_t^N\rangle 
>K\Big]<\frac{\nu}{3}, 
\]
and so choosing $\delta$ so that $K\delta<\nu/3$ in this 
expression and $C\delta<\nu^3/3T$ in~(\ref{union bound}),
combining with~(\ref{martingale term to zero}), the result follows.
\end{proof}

\subsection{Continuity estimates for $\rho_\epsilon*\eta$}
\label{continuity for random walk}

To identify the limit point of any convergent subsequence
of $\{\rho_\epsilon*\eta^N(x)\}$, 
we will require some control on the spatial continuity of the 
functions $\rho_\epsilon*\eta^N(x)$. This will be inherited from the regularity
of the transition density of the 
Gaussian random walk with generator ${\cal L}^\theta$, which in turn follows
from its representation
as that of a Brownian motion evaluated at the random time $T(t)$ defined
in Lemma~\ref{PsiBoundHS}.
Our approach will be to 
approximate $\psi_t^{\epsilon,x}(\cdot)$ by $p_{\epsilon^2+t}(x, \cdot)$,
and to control the error that this introduces we 
need to control $T(t)-t$. 

\begin{lemma}
    \label{lem:poisson_ld}
    In the notation of Lemma~\ref{PsiBoundHS}, for any $A>1$,
    \begin{align*}
        \IP\left\{
            T(t) - t > A(\epsilon^2 + t)
        \right\}
    \le
	    \exp\left(-\frac{\theta A}{4}\big(\epsilon^2+t\big)
	    \right) .
    \end{align*}
\end{lemma}

\begin{proof}
    This is just a Chernoff bound. With $\Pi$ a rate one Poisson process as
	in Lemma~\ref{PsiBoundHS}, for any $A>1$,
	\begin{align*}
        \IP\left\{
            T(t) - t > A(\epsilon^2 + t)
        \right\}
&=
        \IP\Big\{
            \Pi(\theta t)
            >
		\theta\Big( t + A(\epsilon^2 + t)\Big)
        \Big\} \\
&\leq\frac{\IE\left[\exp\left(\alpha\Pi(\theta t)\right)\right]}
		{\exp\left(\alpha\theta\big(t+A(\epsilon^2+t)\big)\right)}
		\\
		&=\exp\left(\theta t\big(e^\alpha-1\big)-\alpha\theta\big(t+A(\epsilon^2+t)\big)\right)
		\\
		&\leq
		\exp\left(\theta t\big(e^\alpha-\alpha-1-\frac{A\alpha}{2}\big)
		-\frac{A\alpha}{2}\theta(\epsilon^2+t)\right).
    \end{align*}
	Now set $\alpha=1/2$. Since $A>1$, $e^\alpha-\alpha-1-A\alpha/2<0$ and the 
	result follows.
\end{proof}

As advertised, we wish to control the difference between
$\psi_t^{\epsilon,x}(y)$
and $p_{\epsilon^2 + t}(x, y)$.

\begin{lemma}
    \label{Lemma:BoundPsiHS2}
    In the notation of Lemma~\ref{PsiBoundHS},
    there exists a $C < \infty$ such that
	\begin{equation}
\label{heat kernel estimate}
        \left|
            \psi_t^{\epsilon, x}(y)
            -
            p_{\epsilon^2 + t}(x, y)
        \right|
\le
        \frac{C}{(\epsilon^2 \theta)^{1/2}}
        p_{6(\epsilon^2+t)}(x, y)
        +
	    \frac{C}{(\epsilon^2 + t)^{d/2}}
	    \exp(- \epsilon^2 \theta / 2)
	    .
	\end{equation}
\end{lemma}

\begin{proof}
    Still using the notation of Lemma~\ref{PsiBoundHS},
 we partition into three events according to the value of $\tau(t)$.
    Let $A_1 = \{ \tau(t) < - (\epsilon^2 + t)/2 \}$,
    $A_2 = \{ \tau(t) > 2(\epsilon^2 + t) \}$,
    and $A_3$ the remaining event, $\{ - (\epsilon^2 + t)/2 \le \tau(t) \le 2(\epsilon^2 + t) \}$.
    Then,
    \begin{align*}
        \left|
            \psi_t^{\epsilon, x}(y)
            -
            p_{\epsilon^2 + t}(x, y)
        \right|
        &=
        \left|
        \IE\left[
            p_{\epsilon^2 + t + \tau(t)}(x, y)
            -
            p_{\epsilon^2 + t}(x, y)
        \right]
        \right|
        \\&\le
        \IE\left[
            (1_{A_1} + 1_{A_2} + 1_{A_3})
            \left|
            p_{\epsilon^2 + t + \tau(t)}(x, y)
            -
            p_{\epsilon^2 + t}(x, y)
        \right|
        \right] .
    \end{align*}
    For the first term, note that if $a < b$ then
    \begin{align*}
        |p_a(x, y) - p_b(x, y)| 
        &=
        \frac{1}{(2\pi)^{d/2}}
        \left|
            \frac{1}{a^{d/2}}
            e^{-\|x - y\|^2 / 2a}
            -
            \frac{1}{b^{d/2}}
            e^{-\|x - y\|^2 / 2b}
        \right|
        \\ &=
        \frac{1}{(2\pi a^2)^{d/2}}
            e^{-\|x - y\|^2 / 2b}
        \left|
            e^{-\|x - y\|^2 \left(\frac{1}{2a} - \frac{1}{2b}\right)}
            -
            \left(\frac{a}{b}\right)^{d/2}
        \right|
        \\ &\le
        C \left(\frac{b}{a}\right)^{d/2}
        p_b(x, y) ,
    \end{align*}
    where the inequality follows because both terms under the absolute value 
are less than 1.
Since, on the event $A_1$, $\tau(t)<0$, we can  
apply this with $a = \epsilon^2 + t + \tau(t)$ and $b = \epsilon^2 + t$,
and, using the bound~(\ref{bound for negative tau}),
    \begin{align*}
        \IE\left[
            1_{A_1} |p_a(x, y) - p_b(x, y)| 
        \right]
        & \le
            C \left(\frac{\epsilon^2 + t}{\epsilon^2}\right)^{d/2}
            p_{\epsilon^2 + t}(x, y) 
            \IP\left\{ \tau(t) < - \frac{\epsilon^2 + t}{2} \right\}
        \\ & \le
            C \frac{1}{\epsilon^d}
            \IP\left\{ \tau(t) < - \frac{\epsilon^2 + t}{2} \right\}
        \\ & \le
	    \frac{C}{(\epsilon^2+t)^{d/2}} \exp\left(-\frac{\theta \epsilon^2}{2} \right) .
    \end{align*}

    For the third term, we will first collect some facts.
    Observe that on the event $A_3$,
    $\epsilon^2 + t + \tau(t)$ is between
    $(\epsilon^2 + t)/ 2$ and $3(\epsilon^2 + t)$,
    and for any $s$ in this interval,
    \begin{align}
        p_{2s}(y)
        &\le \nonumber
        \left( \frac{
            6 (\epsilon^2 + t)
        }{
            \epsilon^2 + t
        } \right)^{d/2}
        p_{6(\epsilon^2 + t)}(x, y)
        \\ &= \label{eqn:p_bounded_on_interval}
        6^{d/2}
        p_{6(\epsilon^2 + t)}(x, y) .
    \end{align}
    Moreover, since $u e^{-u} \le e^{-1}$ for all $u \ge 0$,
    \begin{align}
        \frac{\|x - y\|^2}{s} p_s(x, y)
        &= \nonumber
        \frac{4}{(2 \pi s)^{d/2}}
        e^{- \frac{ \|x-y\|^2 }{ 4s }}
        \frac{\|x-y\|^2}{4s}
        e^{- \frac{ \|x-y\|^2 }{ 4s }}
        \\ &\le \label{eqn:p_deriv_term2}
        C p_{2s}(x, y) .
    \end{align}
    Now, by the Intermediate Value Theorem,
    \begin{align}
\label{deduction from IVT}
        \left|
            p_{\epsilon^2 + t + \tau(t)}(x, y)
            -
            p_{\epsilon^2 + t}(x, y)
        \right|
        &=
        \left| \tau(t) \right|
        \left| \frac{\partial p_s(x, y)}{\partial s} \right|
    \end{align}
    for some $s$ between $\epsilon^2 + t + \tau(t)$ and $\epsilon^2 + t$.
   Since 
    \begin{align*}
        \partial_ s p_s(x, y)
        &=
        \partial_s
        \left(
            \frac{1}{(2 \pi s)^{d/2}}
            \exp\left( - \frac{\|x - y\|^2}{2 s} \right)
        \right)
        \\ &=
        - \frac{d}{2s} p_s(x, y) + \frac{\|x - y\|^2}{2 s^2} p_s(x, y),
    \end{align*}
    applying the inequality~\eqref{eqn:p_deriv_term2},
    using the fact that $p_s(x, y) \le 2^{d/2} p_{2s}(x,y)$,
    and then \eqref{eqn:p_bounded_on_interval},
    we have that for any $s \in ((\epsilon^2 + t)/ 2, 3(\epsilon^2 + t))$,
    \begin{align*}
        \left| \frac{\partial}{\partial s} p_s(x, y) \right|
        \le
        \frac{C}{s} p_{2s}(x, y) 
        \le
        \frac{C}{\epsilon^2 + t} p_{6(\epsilon^2 + t)}(x, y) .
    \end{align*}
    Therefore, recalling that $\IE[\tau(t)^2] = t / \theta$, substituting 
into~(\ref{deduction from IVT}),
    \begin{align*}
        \IE\left[
            1_{A_3}
            \left|
                p_{\epsilon^2 + t + \tau(t)}(x, y)
                -
                p_{\epsilon^2 + t}(x, y)
            \right|
        \right]
        &\le
            \frac{C}{\epsilon^2 + t}
                p_{6(\epsilon^2 + t)}(x, y)
            \IE\left[|\tau(t)|\right]
        \\ &\le
            \frac{C}{\epsilon^2 + t}
                p_{6(\epsilon^2 + t)}(x, y)
            \IE\left[\tau(t)^2\right]^{1/2}
        \\ &=
            \left( \frac{Ct}{\theta(\epsilon^2 + t)^2} \right)^{1/2}
                p_{6(\epsilon^2 + t)}(x, y)
        \\ &\le
            \frac{C}{\sqrt{\theta \epsilon^2}}
                p_{6(\epsilon^2 + t)}(x, y) ,
    \end{align*}
    where the last inequality follows from $2 \epsilon^2 t \le (\epsilon^2 + t)^2$.

    Finally, on the event $A_2 = \{\tau(t) > 2(\epsilon^2 + t)\}$,
    we simply use
    \begin{align*}
        \left|
            p_{\epsilon^2 + t + \tau(t)}(x, y)
            -
            p_{\epsilon^2 + t}(x, y)
        \right|
        \le
        \frac{C}{(\epsilon^2 + t)^{d/2}} ,
    \end{align*}
    so that
    \begin{align*}
        \IE\left[
        1_{A_2}
        \left|
            p_{\epsilon^2 + t + \tau(t)}(x, y)
            -
            p_{\epsilon^2 + t}(x, y)
        \right|
        \right]
        \le
        \frac{C}{(\epsilon^2 + t)^{d/2}}
        \IP\left\{
            \tau(t) > 2 (\epsilon^2 + t)
        \right\} ,
    \end{align*}
    and apply Lemma~\ref{lem:poisson_ld}
    with $A=2$.
\end{proof}

The last result will be useful when combined with the next bound for the heat kernel.

\begin{lemma}
    \label{Lemma:ContinuityHS}
Let $s>0$, and $x, y, z\in \IR^d$. The following estimate holds:
\[
    |p_s(x,z) - p_s(y,z)|
    \leq
    \frac{C\|x-y\|}{\sqrt{s}} \left(p_{2s}(x,z) + p_{2s}(y,z)\right) ,
\]
where the constant $C$ does not depend on $x,y,z$ or $s$.
\end{lemma}

\begin{proof}

    Expanding the difference of two squares,
    \begin{align*}
        e^{- \frac{\|y-z\|^2}{2 s}}
        -
        e^{- \frac{\|x-z\|^2}{2 s}}
        &=
        \left(
            e^{-\frac{\|y-z\|^2}{4 s}}
            -
            e^{-\frac{\|x-z\|^2}{4 s}}
        \right)
        \left(
            e^{-\frac{\|y-z\|^2}{4 s}}
            +
            e^{-\frac{\|x-z\|^2}{4 s}}
        \right) .
    \end{align*}
Now, thinking of the first term in brackets as a function of a single variable $x$ on
the line segment $[y,z]$ connecting $y$ to $z$, we can apply the
Intermediate Value Theorem and take the modulus to bound this expression by 
    \[
        \|y-x\|
        \left( \frac{2\|w-z\|}{4s} \exp\left(-\frac{\|w-z\|^2}{4s}\right)\right)
        (4 \pi s)^{d/2}
        \left( p_{2s}(y,z)+p_{2s}(x,z) \right)
    \]
    for some $w\in [y,z]$.
    Using the fact that $x e^{-x^2}$ is uniformly bounded,
    we can bound the first bracket in the last equation by $C/\sqrt{s}$, and the result follows.
\end{proof}

We now have the ingredients that we need to write down a continuity 
estimate for $\rho_\epsilon*\eta$.
We fix $\delta>0$ and 
suppose that $s>\delta$.
Let us write
\[
\widehat{\epsilon}(\delta,\epsilon,\theta):=
\frac{1}{(\epsilon^2+\delta)^{d/2}}
e^{-\epsilon^2\theta/2}, 
\]
and note that under the assumption that $\epsilon^2\theta\to \infty$, for 
each fixed $\delta>0$,  
$\lim_{\epsilon\to 0, \theta\to\infty}
\widehat{\epsilon}(\delta,\epsilon,\theta)=0$. 
Using the semimartingale decomposition~(\ref{expn rhoepsilon}),
and Lemma~\ref{Lemma:BoundPsiHS2},
we have
\begin{align}
& |\rho_\epsilon*\eta_s(y)-\rho_\epsilon*\eta_s(w)| \nonumber 
=
    |\langle p_{\epsilon^2}(y,z)-p_{\epsilon^2}(w,z),\eta_s(dz) \rangle | 
\nonumber
\\ &\qquad
\leq 
\big\langle
    |p_{\epsilon^2+s}(y,z)-p_{\epsilon^2+s}(w,z)|, 
\eta_0(dz) \big\rangle 
\nonumber
\\ &\qquad \qquad {}
+
\int_0^{s-\delta}
    \big\langle |p_{s-r+\epsilon^2}(y,z)-p_{s-r+\epsilon^2}(w,z)||
    |F(\rho_\epsilon*\eta_r(z)|, 
\eta_r(dz) \big\rangle  dr
\nonumber
\\ &\qquad \qquad {}
+
\Big\langle
    \frac{C}{(\theta\epsilon^2)^{1/2}}
    \Big(p_{6(\epsilon^2+s)}(y,z)+p_{6(\epsilon^2+s)}(w,z)\Big)
	+ C\widehat{\epsilon}(\delta,\epsilon,\theta)
,\eta_0(dz) \Big\rangle
\nonumber
\\ &\qquad \qquad {}
+
\int_0^{s-\delta}
\big\langle \Big\{
\frac{C}{(\epsilon^2\theta)^{1/2}}
|p_{6(s-r+\epsilon^2)}(y,z)+
p_{6(s-r+\epsilon^2)}(w,z)| +\widehat{\epsilon}(\delta,\epsilon,\theta)\Big\}
|F(\rho_\epsilon*\eta_r(z)|, 
\eta_r(dz) \big\rangle  dr
\nonumber
    \\ &\qquad \qquad {}
+ \int_{s-\delta}^s \big\langle 
|\psi_{s-r}^{\epsilon,y}(z) + \psi_{s-r}^{\epsilon, w}(z) |
|F(\rho_\epsilon*\eta_r(z)|, 
\eta_r(dz) \big\rangle  dr
\nonumber
\\ & \qquad \qquad {}
+ |M_s(y)|+|M_s(w)| 
\nonumber
\\ & \qquad 
\leq
\Big\langle \frac{\|y-w\|}{\sqrt{s+\epsilon^2}}
    \big(p_{2(s+\epsilon^2)}(y,z)+p_{2(s+\epsilon^2)}(w,z)\big),\eta_0(dz)\Big\rangle
\nonumber
\\ & \qquad \qquad  {}
+
\int_0^{s-\delta}\Big\langle\frac{\|y-w\|}{\sqrt{s-r+\epsilon^2}}
    \big(p_{2(s-t+\epsilon^2)}(y,z)+p_{2(s-r+\epsilon^2)}(w,z)\big)
    |F(\rho_\epsilon*\eta_r(z)|, 
\eta_r(dz)\Big\rangle dr
\nonumber
\\ & \qquad \qquad  {}
+
\Big\langle \frac{C}{(\theta\epsilon^2)^{1/2}}\Big( 
	p_{6(\epsilon^2+s)}(y,z)+p_{6(\epsilon^2+s)}(w,z)\Big)+
	C \widehat{\epsilon}(\delta,\epsilon,\theta) , \eta_0(dz)
\Big\rangle
\nonumber
\\ & \qquad \qquad  {}
+
\int_0^{s-\delta}
\big\langle \Big\{
\frac{C}{(\epsilon^2\theta)^{1/2}}
(p_{6(s-r+\epsilon^2)}(y,z)+ p_{6(s-r+\epsilon^2)}(w,z))
+\widehat{\epsilon}(\delta,\epsilon,\theta)\Big\}
|F(\rho_\epsilon*\eta_r(z)|, 
\eta_r(dz) \big\rangle  dr
\nonumber
    \\ &\qquad \qquad {}
+ \int_{s-\delta}^s \big\langle 
|\psi_{s-r}^{\epsilon,y}(z) + \psi_{s-r}^{\epsilon, w}(z) |
|F(\rho_\epsilon*\eta_r(z)|, 
\eta_r(dz) \big\rangle  dr
\nonumber
\\ & \qquad \qquad  {}
+ |M_s(y)|+|M_s(w)|. 
\label{difference in rhos}
\end{align} 
Although this expression is lengthy, we have successfully isolated the
terms involving $\|y-w\|$, which will control the regularity as we pass to the
limit. Asymptotically, we 
don't expect the martingale terms to contribute, since their 
quadratic variation scales with $\theta/(N\epsilon^d)$; 
under the assumption that $\epsilon^2\theta\to\infty$,
for any fixed $\delta>0$, 
the terms arising from 
approximating the transition 
density $\psi_{s-r}^{\epsilon, \cdot}(\cdot)$ of the Gaussian 
walk by $p_{s-r+\epsilon^2}(\cdot,\cdot)$ at times 
with $s-r>\delta$ will tend to zero; and  
the moment bounds of Lemma~\ref{bounds on moments} 
will allow us to control the integral over $[s-\delta, s]$.
There is some technical work to be done to rigorously identify the
limit points of $\rho_\epsilon*\eta^N$, but 
it really amounts to 
applying the tower property and our moment bounds from 
Lemma~\ref{bounds on moments} and 
Corollary~\ref{alternative form moment bounds}.

\subsection{Identification of the limit}
\label{limit in onestep case}

We now turn to the identification of the limit points of the sequence 
$\big\{\big(\rho_\epsilon*\eta_t^N(x)dx\big)_{t\geq 0}\big\}_{N\geq 1}$.
We would like to show that any limit point solves~\eqref{target equation in local convergence}
in the limit, i.e.,
\begin{equation}\label{eqn:target PDE weak form}
    \langle f(x), \varphi(t,x)dx \rangle
    =
    \int_0^t
        \Big\langle
        \frac{1}{2}\Delta f(x) + f(x) F(\varphi(s,x))
        ,
        \varphi(s,x) dx
        \Big\rangle
    ds .
\end{equation}
Since 
$\langle f, \rho_\epsilon*\eta_t^N(x)dx\rangle=
\langle \rho_\epsilon*f(x), \eta_t^N(dx)\rangle$,
and the limit is deterministic,
this will follow if we can show that each of the terms in the semimartingale
decomposition~(\ref{eqn:eta_martingale}),
with the test function $f$ replaced by $\rho_\epsilon*f(\cdot)$,
converges to the corresponding term in~\eqref{eqn:target PDE weak form}.

The linear term is straightforward.
Write $T_{\cdot}$ for the heat semigroup, so that
$\rho_\epsilon*f(x) = T_{\epsilon^2}f(x)$.
By a Taylor expansion,
\begin{align*}
\int_0^t \big\langle \mathcal{L}^\theta T_{\epsilon^2} f, 
	\eta_s^N(dx)\big\rangle ds 
&= \int_0^t \big\langle \frac{1}{2} \Delta T_{\epsilon^2} f(x), 
	\eta_s^N(dx) \big\rangle 
	+ \mathcal{O}\Big(\frac{1}{\theta}\Big) 
\\ &= \int_0^t \big\langle \frac{1}{2} \Delta f(x), \rho_\epsilon* \eta_s^N(x) dx 
	\big\rangle ds + \mathcal{O}\Big(\frac{1}{\theta}\Big).
\end{align*}
Thus, from weak convergence we can deduce that under our scaling,
for any (weakly) convergent subsequence $\{\rho_\epsilon*\eta^N(x)dx\}_{N\geq 1}$,
\[ 
\int_0^t \big\langle \mathcal{L}^\theta T_{\epsilon^2} f, 
\eta_s^N(dx)\big\rangle ds 
\rightarrow \int_0^t \big\langle \frac{1}{2} \Delta f(x), 
\varphi (s,x) dx \big\rangle ds. 
\]

The nonlinear term in the semimartingale decomposition is more intricate. 
It takes the form
\[ 
\IE\Big[
\int_0^t\big\langle T_{\epsilon^2} f(y) F\big(\rho_\epsilon *\eta_s(y)\big), 
\eta_s(dy) \big\rangle ds\Big]
\]
and we should like to show that this converges to 
	\[
		\int_0^t\int f(y)F(\varphi(s,y))\varphi(s,y)dyds.
	\]
	We proceed in stages. First we should like to transfer the 
	heat semigroup from $T_{\epsilon^2}f$ onto $\eta_s$. 
	Since $f$ is smooth, this will follow easily if we can show that
\begin{align*}
\IE\Big[
\int_0^t\big\langle T_{\epsilon^2} f(y) F\big(\rho_\epsilon *\eta_s(y)\big), 
\eta_s(dy) \big\rangle ds\Big]
	\sim 
\IE\Big[
\int_0^t\big\langle T_{\epsilon^2} f(y) F\big(\rho_\epsilon * \eta_s(y)\big), 
\rho_\epsilon * \eta_s(y) dy \big\rangle ds\Big].
\end{align*}
This is the content of Proposition~\ref{CuadraticApprox}. 

\begin{proposition} 
\label{CuadraticApprox}
Under the conditions of Theorem~\ref{thm:local_convergence},
\begin{equation}
\label{replacing nonlinear term}
\lim_{\epsilon\to 0}\IE\Big[\Big|\int_0^t\big\langle T_{\epsilon^2} f(y) 
F(\rho_\epsilon *\eta_s(y)), \eta_s(y) \big\rangle 
- \big\langle T_{\epsilon^2} f(y) 
F(\rho_\epsilon * \eta_s(y)), \rho_\epsilon * \eta_s(y) dy 
\big\rangle ds \Big|\Big]
=0.
\end{equation}
\end{proposition}

\begin{proof}
In fact we are going to fix $\delta>0$, with $t>\delta$, and show that
the expression on the left hand side of~\eqref{replacing nonlinear term} is
less than a constant times $\delta$, with a constant independent of $\delta$,
$N$, and $\epsilon$. Since $\delta$ is arbitrary, the result will follow.

We first note that,
\begin{align*}
    &
\langle T_{\epsilon^2} f(y) F(\rho_\epsilon * \eta_s(y)), 
\rho_\epsilon * \eta_s(y) dy \rangle 
-\langle T_{\epsilon^2} f(y) F(\rho_\epsilon *\eta_s(dy)), \eta_s(dy) \rangle 
\\ 
& = 
\langle \int T_{\epsilon^2} f(y) 
F(\rho_\epsilon * \eta_s(y))\rho_\epsilon(y-w) dy, \eta_s(dw)\rangle - \langle 
T_{\epsilon^2} f(w) F(\rho_\epsilon*\eta_s(w)), \eta_s(dw) \rangle 
\\ 
&= 
\langle \int \big\{ T_{\epsilon^2} f(y) F(\rho_\epsilon*\eta_s(y))-T_{\epsilon^2}f(w) 
F(\rho_\epsilon*\eta_s(w))\big\} 
\rho_\epsilon(w-y)dy, \eta_s(dw) \rangle.
\end{align*}
Let us denote the integral against $dy$ in the last expression by $I$, that is 
\[ I :=  \int \{ T_{\epsilon^2} f(y) 
F(\rho_\epsilon*\eta_s(y))-T_{\epsilon^2}f(w) 
F(\rho_\epsilon*\eta_s(w))\} \rho_\epsilon(w-y)dy, \]
and note that $|I|$ is bounded by
\begin{align}
& \int 
\Big\{ |F(\rho_\epsilon*\eta_s(y))-F(\rho_\epsilon*\eta_s(w))|T_{\epsilon^2}f(y)+
F(\rho_\epsilon*\eta_s(w))|T_{\epsilon^2}f(y)-T_{\epsilon^2}f(w))|\Big\} 
\rho_{\epsilon}(w-y)dy
\nonumber \\ &\leq \int \| f\|_{\infty} 
\big|F(\rho_\epsilon*\eta_s(y))-F(\rho_\epsilon*\eta_s(w))\big| \rho_\epsilon(w-y)dy + 
C \epsilon\|f' \|_\infty |F(\rho_\epsilon*\eta_s(w))|, \label{FirstBoundCT}
\end{align}
where we have used that 
\[
\int |T_{\epsilon^2}f(y)-T_{\epsilon^2}f(w)| p_{\epsilon^2}(w,y)dy
\leq \|f'\|_\infty\int|y-w|p_{\epsilon^2}(w,y)dy.
\]
Now recall that $F$ is a polynomial of degree $n$, and so 
there exist real numbers $b_k$ such
that $F(a)-F(b)= (a-b)\sum_{k=1}^{n-1} b_k a^k b^{n-1-k}$ and so 
\[
|F(\rho_\epsilon*\eta_s(y))-F(\rho_\epsilon*\eta_s(w))|
\leq 
|\rho_\epsilon*\eta_s(y)-\rho_\epsilon*\eta_s(w)|
\sum_{k=1}^{n-1}|b_k|\left(\rho_\epsilon*\eta_s(y)^{n-1}+\rho_\epsilon*\eta_s(w)^{n-1}\right).
\]

Combining the above, we have reduced the problem to showing that for any $k \ge 0$,
\begin{equation}
\label{target expression}
	\lim_{\epsilon\to 0}
	\IE\Big[	\int_0^t\big\langle \int |\rho_\epsilon*\eta_s(y)-\rho_\epsilon*\eta_s(w)|
\left(\rho_\epsilon*\eta_s(y))^k+\rho_\epsilon*\eta_s(w)^k\right)
	p_{\epsilon^2}(w,y)dy, \eta_s(dw)\big\rangle ds\Big]
=0 .
\end{equation}
We are going to use the estimate~(\ref{difference in rhos}).
First note that by Lemma~\ref{bounds on moments} (with $u=0$),
the contribution to~(\ref{target expression}) from the integral over
the time interval $[0,\delta]$ is bounded by a constant multiple of $\delta$, \revpoint{1}{33}
with a constant that depends only on $\eta_0$.
We focus instead
on the interval $(\delta, t]$.

The first term in~(\ref{difference in rhos}) gives
\begin{multline*}
\int_\delta^t\IE\Big[\Big\langle\int
	\big\langle\frac{\|y-w\|}{\sqrt{s+\epsilon^2}}
\big(p_{2(s+\epsilon^2)}(y,z)+p_{2(s+\epsilon^2)}(w,z)\big),\eta_0(dz)\big\rangle
\\
\big(\rho_\epsilon*\eta_s(y)^k+\rho_\epsilon*\eta_s(w)^k\big)
p_{\epsilon^2}(w,y)dy, \eta_s(dw)\Big\rangle\Big] ds.
\end{multline*}
We ``borrow'' from the exponential term to see that 
$\|y-w\|p_{\epsilon^2}(w,y)\leq C\epsilon p_{2\epsilon^2}(w,y)$
and so bound
this by
\begin{multline}
C\int_\delta^t\frac{\epsilon}{\sqrt{s+\epsilon^2}}
\IE\Big[\Big\langle\int\big\langle
\big(p_{2(s+\epsilon^2)}(y,z)+p_{2(s+\epsilon^2)}(w,z)\big),\eta_0(dz)
\big\rangle
\\
\big(\rho_\epsilon*\eta_s(y)^k+\rho_\epsilon*\eta_s(w)^k\big)
p_{2\epsilon^2}(w,y)dy, \eta_s(dw)\Big\rangle ds\Big] .
\label{contribution from initial condition terms}
\end{multline}
The four terms in the product are taken separately, according to the
combinations of $w$ and $y$ appearing. First,
\[
\int_\delta^t\frac{\epsilon}{\sqrt{s+\epsilon^2}}
\IE\Big[\Big\langle\int\big\langle
p_{2(s+\epsilon^2)}(y,z)   
,\eta_0(dz)
\big\rangle
\rho_\epsilon*\eta_s(y)^k 
p_{2\epsilon^2}(w,y)dy, \eta_s(dw)\Big\rangle  \Big] ds
\]
can be rewritten as
\begin{multline*}
\int_\delta^t\frac{\epsilon}{\sqrt{s+\epsilon^2}}
\IE\Big[\int\int\big\langle
p_{2(s+\epsilon^2)}(y,z)  
,\eta_0(dz)
\big\rangle
\rho_\epsilon*\eta_s(y)^k 
p_{\epsilon^2}(x,y) \rho_\epsilon*\eta_s(x) dy dx ds \Big]
\\
\leq
\int_\delta^t\frac{\epsilon}{\sqrt{s+\epsilon^2}}
\IE\Big[\int\int\big\langle
p_{2(s+\epsilon^2)}(y,z)  
,\eta_0(dz)
\big\rangle
\big(\rho_\epsilon*\eta_s(y)^{k+1} +\rho_\epsilon*\eta_s(x)^{k+1}\big)
p_{\epsilon^2}(x,y) dy dx \Big] ds,
\end{multline*}
and using Lemma~\ref{bounds on moments} and the tower property, and integrating
with respect to $s$, under our assumptions on $\eta_0$,
this is 
bounded by
\begin{multline*}
C\epsilon \int_\delta^t\frac{1}{\sqrt{s+\epsilon^2}}
\IE\Big[\int\int\big\langle
p_{2(s+\epsilon^2)}(y,z)  
,\eta_0(dz)
\big\rangle
\\
\big(\rho_\epsilon*\eta_0(y)+\rho_\epsilon*\eta_0(y)^{k+1} 
+\rho_\epsilon*\eta_0(x)
+\rho_\epsilon*\eta_0(x)^{k+1}\big)
p_{\epsilon^2}(x,y) dy dx \Big] ds 
\\
\le C'\epsilon\int_\delta^t\frac{1}{(s+\epsilon^2)^{d/2}}ds .
\end{multline*}
For fixed $\delta$, this bound tends to zero as $\epsilon\to 0$.
The term involving 
$ 
\big\langle
p_{2(s+\epsilon^2)}(w,z)   
,\eta_0(dz)
\big\rangle
\rho_\epsilon*\eta_s(w)^k
$ 
is handled similarly.

On the other hand
\begin{multline*}
\big\langle \int \langle p_{2(s+\epsilon^2)}(y,z), \eta_0(dz)\rangle
\rho_\epsilon*\eta_s(w)^kp_{2\epsilon^2}(w,y)dy, \eta_s(dw)\big\rangle
\\
\leq \frac{C}{(s+\epsilon^2)^{d/2}}\langle 1,\eta_0\rangle\big\langle
\rho_\epsilon*\eta_s(w)^k,\eta_s(dw)\big\rangle,
\end{multline*}
and since $\langle 1,\eta_0\rangle$ is uniformly bounded we apply 
Corollary~\ref{alternative form moment bounds} to obtain a bound on
the contribution to~(\ref{contribution from initial condition terms})
from this term of the same form as the others.

Now consider the contribution to the left hand side of~(\ref{target expression})
from the second term in~(\ref{difference in rhos}). Since $F$ is a polynomial,
it is bounded by a sum of terms of the form
\begin{multline*}
\int_\delta^t\int_0^{s-\delta}\Big\langle\int\frac{\|y-w\|}{\sqrt{s-r+\epsilon^2}}
\big\langle\big(p_{2(s-r+\epsilon^2)}(y,z)+ p_{2(s-r+\epsilon^2)}(w,z)\big)
\rho_\epsilon*\eta_r(z)^j,\eta_r(dz)\big\rangle
\\
\rho_\epsilon*\eta_s(y)^kp_{\epsilon^2}(y,w)dy,\eta_s(dw)
\Big\rangle dr ds
\\
\leq C\epsilon
\int_\delta^t\int_0^{s-\delta}
	\frac{1}{\sqrt{s-r+\epsilon^2}}
	\Big\langle\int
\big\langle
\big(p_{2(s-r+\epsilon^2)}(y,z)+ p_{2(s-r+\epsilon^2)}(w,z)\big)
\rho_\epsilon*\eta_r(z)^j,\eta_r(dz)\big\rangle 
\\
\rho_\epsilon*\eta_s(y)^kp_{2\epsilon^2}(y,w)dy,\eta_s(dw)
\Big\rangle dr ds,
\end{multline*}
where as usual we have ``borrowed'' from the exponential term in
$p_\epsilon^2(y,w)$ to replace $\|y-w\|$ by a constant times $\epsilon$.

Once again, our approach is to rearrange terms so that we can apply
Lemma~\ref{bounds on moments}
or 
Corollary~\ref{alternative form moment bounds} to obtain a bound on
the contribution to~(\ref{target expression})
from these terms of the form $C\epsilon$ (where $C$ may depend on $\delta$
but not $\epsilon$).

For example, using the Chapman-Kolmogorov equation to rewrite
\[
\int\Big\langle 
\big\langle p_{2(s-r+\epsilon^2)}(y,z)\rho_\epsilon*\eta_r(z)^j, \eta_r(dz)
\big\rangle
\rho_\epsilon*\eta_s(y)^k p_{2\epsilon^2}(y,w), \eta_s(dw)\Big\rangle dy
\]
as 
\[
\int\int
\big\langle p_{2(s-r+\epsilon^2)}(y,z)\rho_\epsilon*\eta_r(z)^j, \eta_r(dz)
\big\rangle
\rho_\epsilon*\eta_s(y)^k p_{\epsilon^2}(y,x)\rho_\epsilon*\eta_s(x) dx dy,
\]
and using Lemma~\ref{bounds on moments} and the tower property, we are led to 
control terms of the form
\[
\IE\Big[\int\big\langle p_{2(s-r+\epsilon^2)}(y,z)\rho_\epsilon*\eta_r(z)^j, 
\eta_r(dz)\big\rangle \rho_\epsilon*\eta_r(y)^{k+1} dy\Big].
\]
This, in turn, is at most
\begin{multline*}
\IE\Big[\big\langle \rho_\epsilon*\eta_r(z)^{j+k+1}, \eta_r(dz)\rangle\Big]
+\IE\Big[\int\int p_{2(s-r)}(y,x)\rho_\epsilon*\eta_r(x)
	\rho_\epsilon*\eta_r(y)^{j+k+1} dy dx\Big]
\\
\leq
\IE\Big[\big\langle \rho_\epsilon*\eta_r(z)^{j+k+1}, \eta_r(dz)\rangle\Big]
+2\int\IE\Big[\rho_\epsilon*\eta_r(x)^{j+k+2}\Big]dx,
\end{multline*} 
which is bounded by Lemma~\ref{bounds on moments}.

We now turn to the contribution arising from the martingale terms 
in~(\ref{difference in rhos}):
\[
\IE\Big[
\int_\delta^t\Big\langle\int 
\big(|M_s(y)|+|M_s(w)|\big)\big(\rho_\epsilon*\eta_s(y)^k+
\rho_\epsilon*\eta_s(w)^k\big)p_{\epsilon^2}(w,y)dy, \eta_s(dw)\Big\rangle ds
\Big] .
\]
Since $\psi_{t-s}^{\epsilon,x}(y)=\IE[p_{T(t-s)+\epsilon^2}(x,y)]$, 
rearranging~(\ref{expn rhoepsilon})
we see that we can pull a convolution with $p_{\epsilon^2/2}$ out of
our expressions for $M_s(y)$ and $M_s(w)$ and so all the manipulations 
that we used to control terms above will still be valid.
To deal with the two terms in the product involving $|M_s(y)|$, we write 
the first as $\int |M_s(y)|\rho_\epsilon*\eta_s(y)^{k+1}dy$ and then use
H\"older's inequality, Lemma~\ref{bounds on moments}, and the fact
that $\IE\big[|M_s(y)|^2\big]$ is ${\mathcal O}\big(\theta/(N\epsilon^d)\big)$
to see that the contribution from this term tends to zero in the limit.
For the second, we use the idea of the proof of 
Corollary~\ref{alternative form moment bounds} to reduce to a form to which
we can apply H\"older's inequality.

Control of the terms arising from approximating $\psi^{\epsilon,x}$
by the heat kernel follows in an entirely analogous way.

Combining the above, we see that given $\delta>0$, 
\[
\lim_{\epsilon\to 0}\IE\Big[\Big|\int_0^t\big\langle T_{\epsilon^2} f(y) 
F(\rho_\epsilon *\eta_s(y)), \eta_s(y) \big\rangle ds 
- \int_0^t\big\langle T_{\epsilon^2} f(y) 
F(\rho_\epsilon * \eta_s(y)), \rho_\epsilon * \eta_s(y) dy 
\big\rangle ds\Big|\Big]
<C\delta,
\]
where the constant $C$ is independent of $\delta$.
Since $\delta$ was arbitrary, the proof is complete.
\end{proof}


Since $f$ is smooth, $T_\epsilon^2f-f$ is ${\mathcal O}(\epsilon)$, 
with an application of the triangle inequality, 
\[
\lim_{\epsilon\to 0}\IE\Big[\Big|\int_0^t\big\langle T_{\epsilon^2} f(y) 
F(\rho_\epsilon *\eta_s(y)), \eta_s(y) \big\rangle ds 
- \int_0^t\big\langle f(y) 
F(\rho_\epsilon * \eta_s(y)), \rho_\epsilon * \eta_s(y) dy 
\big\rangle ds\Big|\Big]
<C\delta,
\]
now follows immediately.
Thus to complete the characterisation of the limit, 
it remains to show that if we
take a convergent subsequence 
$\big\{\big(\rho_\epsilon*\eta_t^N(dx)\big)_{t\geq 0}\big\}$ 
converging to a limit point $\big(\varphi(t,x)dx\big)_{t\geq 0}$, then
\[ 
\int_0^t
\int f(x) \rho_\epsilon * \eta_s^N(x) F(\rho_\epsilon * \eta_s^N(x)) dx ds 
\rightarrow \int_0^t\int f(x) \varphi(s,x) F(\varphi(s,x)) dx ds. 
\]
Since $F$ is a polynomial, we consider powers of $\rho_\epsilon*\eta$.
To illustrate the approach, we first prove that
\begin{equation} 
\label{eq:QuadConvergence} 
\int_0^t\int f(x) 
\rho_\epsilon * \eta_s^N(x)^2 dx ds
\rightarrow \int_0^t \int f(x) 
\varphi(s,x)^2 dx ds. 
\end{equation}
The convergence of higher powers will follow in an entirely analogous manner, but
with more complex expressions.

The approach is standard. We fix $\tau>0$ and, in keeping with
our notation $\rho_\epsilon$, in this subsection, use $\rho_\tau$ to 
denote the symmetric Gaussian kernel with variance parameter $\tau^2$.
Our strategy is to show that, up to an error that tends
to zero as $\tau \to 0$, 
\begin{equation}
\label{approx2}
\int_0^t\int f(z) \rho_\epsilon*\eta_s(z)^2 dz ds
\sim 
\int_0^t\int \int f(z) (\rho_\epsilon * \eta_s)(z) 
\rho_\tau (z-y) (\rho_\epsilon * \eta_s)(y) dz dy ds.
\end{equation}
Analogously, also up to an error that vanishes as $\tau \to 0$,
\begin{equation}
\label{approx3}
\int_0^t\int f(z) \varphi(s,z)^2 dz ds
\sim 
\int_0^t\int \int f(z) \varphi(s,z) 
\rho_\tau (z-y) \varphi(s,y) dz dy ds.
\end{equation}
On the other hand,
weak convergence of $\rho_\epsilon * \eta$ 
(plus continuity of the mapping $(z,y) \rightarrow f(z) \rho_\tau(z-y)$) 
gives that 
\begin{equation}
\int_0^t\int \int f(z) (\rho_\epsilon * \eta_s)(z) 
\rho_\tau (z-y) (\rho_\epsilon * \eta_s)(y) dz dy ds
\to \int_0^t \int \int f(z) \varphi(s,z) 
\rho_\tau (z-y) \varphi(s,y) dz dy ds.
\end{equation}
Since $\tau$ is arbitrary, the 
convergence~(\ref{eq:QuadConvergence})
will follow.

\begin{proposition}
	\label{quadratic convergence}
Under the conditions of Theorem~\ref{thm:local_convergence}, we have that
along any convergent subsequence,
\begin{multline}
\label{approx4}
\limsup_{\epsilon \to 0}
\IE\Big[\Big|
\int_0^t\int f(y) \rho_\epsilon*\eta_s(y)^2 dy ds
\\
-
\int_0^t\int \int f(z) (\rho_\epsilon * \eta_s)(z) 
\rho_\tau (z-y) (\rho_\epsilon * \eta_s)(y) dz dy ds\Big|\Big]\leq C\tau,
\end{multline}
where $C$ is independent of $\tau$.
\end{proposition}
\begin{proof}
First note,
\begin{align}
&\int_0^t 
\mathbb{E}\left[\big| \langle f(y), (\rho_\epsilon*\eta_s(y))^2) dy \rangle 
- \int \int f(y) (\rho_\epsilon * \eta_s)(z) \rho_\tau (z-y) 
(\rho_\epsilon * \eta_s)(y) dz dy \big|\right] ds 
\nonumber 
\\ 
&\leq  
\Vert f \Vert_\infty \int  \int_0^t  
\mathbb{E}\left[ \int \left\{ \left|(\rho_\epsilon * \eta_s)(y)
-(\rho_\epsilon * \eta_s)(z) \right|   \rho_\tau(z-y) dz \right\} 
(\rho_\epsilon *\eta_s)(y) \right] ds  dy. 
\label{b1}
\end{align}
Now proceed exactly as in the proof of 
Proposition~\ref{CuadraticApprox}.
The only distinction is that $|p_{\epsilon^2}(y,z)-p_{\epsilon^2}(w,z)|$
is replaced by $|p_{\tau}(y,z)-p_{\tau}(w,z)|$ and 
the estimate $\|y-w\|p_{\tau^2}(y,w)\leq C\tau p_{2\tau^2}(y,w)$
replaces the corresponding statement with $\epsilon^2$ replacing 
$\tau^2$ in our previous argument.
\end{proof}

The extension of Proposition~\ref{quadratic convergence}
to higher moments is straightforward, if notationally messy. For 
fixed (but arbitrary) $\tau$, one shows that
\begin{multline*}
	\limsup_{\epsilon\to 0}
	\IE\Big[\Big|
	\int_0^t\int f(y)\rho_\epsilon*\eta_s^N(y)^k dy ds
\\
	- \int_0^t\int\cdots\int f(y_1)
	\rho_\epsilon*\eta_s^N(y_1) 
\prod_{i=2}^{k}\rho_\tau (y_i-y_{i-1})\rho_\epsilon*\eta_s^N(y_i)
dy_{k}\ldots dy_1 ds\Big| \Big]
	\leq C\tau,
\end{multline*}
as well as a corresponding statement with $\rho_\epsilon*\eta_s^N(x)$ replaced
by $\varphi(s,x)$ and then use weak convergence to see that, up to an error 
of order $\tau$, any limit point of
the sequence $\{\rho_\epsilon*\eta^N(x)dx\}$ solves (the weak form 
of) equation~(\ref{target equation in local convergence}).
Since $\tau$ was arbitrary, the proof of Theorem~\ref{thm:local_convergence}
is complete.

\section{Proofs of results for the lookdown process and ancestral lineages}
    \label{sec:lookdown_proofs}

Now we turn to results about the lookdown process,
first establishing the basic connection between the population process $\eta^N$
and the lookdown process $\lp^N$,
Proposition~\ref{thm:mmt_application},
and then in the next section, convergence of the lookdown process itself.

\begin{proof}[Proof of Proposition~\ref{thm:mmt_application}:]
    This proposition is the content of the Markov Mapping Theorem,
    reproduced from \citet{etheridge/kurtz:2019} as Theorem~\ref{thm:mmt},
    applied to our situation.
    The function $\gamma$ of that theorem is what we have called $\kappa$ above,
    and the kernel $\alpha$ of that theorem
    is the transition function that assigns levels uniformly on $[0, N]$ (in the first case)
    or as a Poisson process with Lesbegue intensity (in the limiting case).
    We need a continuous $\psi^N(\lp) \ge 1$ such that $|A^N f(\lp)| \le c_f \psi^N(\lp)$
    for all $f$ in the domain of $A^N$ (and similarly a function $\psi$ for $A$).
    We also need that applying the lookdown generator to a function and averaging over levels
    is equivalent to applying the population process generator to the function
    whose dependence on levels has been averaged out,
    a condition which we precisely state, and verify,
    in Lemmas~\ref{lem:averaged_generators} and~\ref{lem:averaged_generators_limit}
    of the Appendix.

    For finite $N$,
    taking $f(\lp)$ of the form \eqref{eqn:test_functions},
    we can use $\psi^N(\lp) = \langle C (1 + u |F(x, \eta)|), \lp(dx, du) \rangle$
    for an appropriate constant $C$.
    For the scaling limit, recall that the test functions $f$ are of the form
    $f(\lp) = \prod_{(x, u) \in \lp} g(x, u)$
    with $g(x, u) = 1$ for $u \ge u_0$,
    and consulting~\eqref{eqn:limiting_lookdown_generator},
    we see that most terms in $Af(\lp)$ can be bounded 
    as above by constant multiples of $\langle 1, \eta \rangle$.
    However, the term involving $F$ is, as usual, more troublesome.
    Since $0 \le f(\lp) / g(x, u) \le 1$ for any $(x, u) \in \lp$,
    \begin{align*} 
        \big| f(\lp) \sum_{(x, u) \in \lp} F(x, \eta) u \frac{\partial_u g(x, u)}{g(x, u)} \big|
        & \le
        \| \partial_u g \|_\infty \sum_{(x, u) \in \lp} | F(x, \eta) u \ind_{u \le u_0} | \\
        & \le
        \| \partial_u g \|_\infty e^{u_0} \sum_{(x, u) \in \lp} | F(x, \eta) u e^{-u} | .
    \end{align*}
    The first line would be just what we want,
    except that $\psi(\lp)$ cannot depend on $f$, and hence neither on $u_0$.
    So, the second line provides us with the required bound:
    we absorb $\| \partial_u g \|_\infty e^{u_0}$ into $c_f$
    and take $\psi(\lp) = 1 + \langle 1 + F(x, \eta) u e^{-u}, \lp(dx, du) \rangle$.
\end{proof}

\subsection{Tightness of the Lookdown Process}
    \label{sec:lookdown_convergence}

Now we turn to the main theorem on convergence of the lookdown process, 
Theorem~\ref{thm:lookdown_convergence}, whose proof
follows a similar pattern to that of convergence for the population processes
in Section~\ref{sec:population_density_proof}.

We first give a description of the lookdown process $\lp^N$ in terms of the lines of descent
introduced in Section~\ref{sec: individual lines of descent}.
Each line of descent
gives birth to lines at higher levels at rate $2 (N - u) c_\theta(x, \eta)$,
and each such new line chooses a level uniformly from $[u, N]$,
a spatial location $y$ from the kernel
\begin{equation} \label{eqn:qm_defn}
    q^m(x, dy, \eta)
    =
    r(y, \eta) q(x, dy) / \int_{\IR^d} r(z, \eta) q(x, dz),
\end{equation}
and the two lines swap spatial locations with probability 1/2;
the level of each line of descent evolves according to equation~\eqref{eqn:u_evolution}.

It is evident from the description of the process
(or, by differentiating in Definition~\ref{defn:lookdown_mgale})
that
\begin{align} \label{eqn:f_xi_mgale}
\begin{split}
    \langle f, \lp_t^N \rangle
    &=
    \langle f, \lp_0^N \rangle
    + M^f_t
\\ & \qquad {}
    +
    \int_0^t
    \bigg\langle
        c_\theta(x, \eta^N_s)
        \int_u^N \int_{\IR^d}
        \left(
            f(y, u_1) + f(x, u_1) + f(y, u) - f(x, u)
        \right)
        q^m(x, dy, \eta^N_s) du_1 
\\ & \qquad \qquad \qquad \qquad {}
        +
        \left(
            c_\theta(x, \eta^N_s) u^2
            - b_\theta(x, \eta^N_s) u
        \right)
        \frac{d}{du} f(x, u)
    ,
    \lp_s^N(dx, du) \bigg \rangle
    ds ,
\end{split}
\end{align}
where $M^f$ is a martingale
with angle bracket process
\begin{align} \label{eqn:f_xi_qv}
    \begin{split}
    \left\langle M^f \right\rangle_t
    &=
    \int_0^t
    \bigg \langle c_\theta(x, \eta_s^N) \int_u^N \int_{\IR^d}
        \big[ f(y, u_1)^2
\\&\qquad \qquad {}
        + \left( f(x, u_1) + f(y, u) - f(x, u) \right)^2 \big]
    du_1 q^m(x, dy, \eta^N_s)
    ,
    \lp_s^N(dx, du) \bigg \rangle
    ds .
\end{split}\end{align}

\begin{remark}
In addition to tightness of the measure-valued processes $\lp^N$, \revpoint{1}{34}
the bounds used in the proofs below also imply
tightness of the number of lines of descent and the number of births below a fixed level,
and of the motion of individual lines of descent.
In other words, the limiting ``line of descent'' construction
of Section~\ref{sec: individual lines of descent} holds.
\end{remark}

\begin{proof}[Proof of Theorem~\ref{thm:lookdown_convergence}]
    As in Section~\ref{sec:population_density_proof},
    the theorem will follow from tightness and characterization of the limit points.
    This time, the processes $\lp^N$ take values in $\lpmeasures$,
    the space of locally finite measures on space $\times$ levels.
    (They will in fact be point measures, including the limit,
    but that is a consequence of this theorem.)
    Again, tightness follows from a compact containment condition,
    tightness of one-dimensional distributions,
    and an application of \citet{ethier/kurtz:1986} Theorem~3.9.1.

    Lines of descent can escape to infinite level in finite time, and so
    we endow $\lpmeasures$ with the vague topology ``in the level coordinate'',
    induced by test functions on $\overline{\IR^d} \times [0,\infty)$
    of the form $g(x) h(u)$, where $g \in C_b(\overline{\IR^d})$ is bounded and continuous
    and $h \in C_c([0, \infty))$ is compactly supported
    (following, e.g., \citet{etheridge/kurtz:2019}, Condition 2.1).
    In several places below we require a dense subset
    of $C_b(\lpmeasures)$, the bounded, continuous functions on $\lpmeasures$.
    The functions $\lp \mapsto \exp(-\langle f, \lp \rangle)$
    for nonnegative, compactly supported $f : \overline{\IR^d} \times [0,\infty)$
    do not form not a dense subset of $C_b(\lpmeasures)$,
    but they do separate points and vanish nowhere, 
    since for any $\lp_1$ and $\lp_2$ there is an $f$
    with $\langle f, \lp_1\rangle \neq \langle f, \lp_2\rangle$,
    and a $g$ such that $\langle g, \lp_1\rangle \neq 0$.
	Therefore,
    by the Stone-Weierstrass theorem, the algebra they generate is dense
    in $C_b(\lpmeasures)$ with respect to uniform convergence on compact subsets.
    Topologized in this way, the space $\lpmeasures$ is completely metrizable,
    and we may choose a countable set of bounded, nonnegative $f_k$,
    each supported on $\IR^d \times [0,u_k]$ for some $u_k < \infty$,
    such that a subset $K \subset \lpmeasures$ is relatively compact if and only if
    $\sup_{\lp \in K} \langle f_k, \lp \rangle < \infty$ for each $k$.
    (To see this, use Theorem A.2.3 of~\citet{kallenberg1997foundations}.)
    Below, Lemma~\ref{lem:compact containment lookdown} proves exactly this,
    and therefore compact containment.
    Here we have compactified $\IR^d$ for convenience
    (since it turned out to be straightforward to show that mass does not escape to infinity in space);
    however, we need to use the vague topology ``in the level direction''
    because \emph{levels} may escape to infinity in finite time in the limit.

    In order to apply \citet{ethier/kurtz:1986} Theorem~3.9.1
    we require that $\{(F(\lp^N_t))_{t \ge 0}\}_N$
    is tight as a sequence of real-valued c\`adl\`ag processes, for all $F$ in a subset
    of $C_b(\lpmeasures)$ 
    that is dense with respect to uniform convergence on compact subsets.
    Lemma~\ref{lem:test_fn_tightness} shows that
    $\{\langle f, \lp^N_t \rangle\}_N$ is a tight sequence
    for any $f: \overline{\IR^d} \times [0,\infty) \to \IR$ with compact support
    in the level direction,
    and hence $\{e^{-\langle f, \lp^N_t \rangle}\}_N$ is tight as well.
    Since as above the algebra generated by the functions $\lp \mapsto \exp(-\langle f, \lp \rangle)$
    is dense in $C_b(\lpmeasures)$,
    it suffices to show that tightness for the processes $(\exp(-\langle f, \lp^N_t\rangle))_{t \ge 0}$
    extends to finite sums and products (and constant multiples) of these processes,
    which is shown in Lemma~\ref{lem:product_tightness}.
    The fact that martingale properties are preserved under passage to the limit
    is straightforward, and can be proved in a way analogous to Lemma~\ref{lem:limit_mgale};
    we omit the proof.
    Finally, we must show that the limiting lookdown process $\lp$
    projects to the limiting process $\eta$, i.e., a solution of the martingale
    problem in Theorem~\ref{thm:nonlocal_convergence}.
    Let $N_k \to \infty$ be a sequence along which $\lp^{N_k}$ converges.
    By Theorem~\ref{thm:nonlocal_convergence},
    there is a subsequence $N_{k(j)}$ along which the projected population processes
    $\eta^{N_{k(j)}}$ converge, and the limit solves the martingale problem.
    Thus any limit point of $\lp^N$ projects to a population process $\eta$
    solving the martingale problem of Theorem~\ref{thm:nonlocal_convergence}.
\end{proof}

What we need for compact containment will come from the following Lemma.
The generality is unimportant --
for concreteness one may take $h(u) = e^{-u}$.

\begin{lemma}
        \label{lem:h_containment}
    Let $h$ be a positive, continuous, nonincreasing, differentiable function on $[0, \infty)$
    such that
    $\int_0^\infty \int_u^\infty h(v) dv du$,
    $\int_0^\infty u^2 |h'(u)| du$,
    and $\int_0^\infty h(u)^2 du$
    are all finite.
    Suppose that Assumptions~\ref{def:model_setup} hold,
    and that $\theta/N \to \alpha$ and $\lp_0^N \to \lp_0$ weakly as $N \to \infty$,
    where each $\lp_0^N$ is conditionally uniform given $\eta_0^N$
    in the sense of \eqref{eqn:conditionally_uniform}
    and $\lp_0$ is conditionally Poisson given $\eta_0$
    in the sense of \eqref{eqn:conditionally_poisson}.
    Then for any $T$ there exists a constant $K(T)$ such that for all $M > 0$,
    $$ 
    \limsup_{N \to \infty} \IP\left\{ \sup_{0 \le t \le T} \langle h, \lp^N_t \rangle > M \right\}
    <
    \frac{K(T)}{M} .
    $$
\end{lemma}

We postpone the proof of this Lemma until we have shown how it yields compact containment.
First, we show that this implies compact containment of the processes $(\langle f, \lp^N_t \rangle)_{0 \le t \le T}$
for arbitrary compactly supported $f$.

\begin{lemma}
    \label{lem:compact containment hn}
    Suppose $f \in C(\bar \IR^d \times [0, \infty))$ and there is a $u_f$ such that
    if $u \ge u_f$ then $\sup_x f(x, u) = 0$.
    Under the assumptions of Lemma~\ref{lem:h_containment},
    for any $T$ there exists a constant $K(f,T)$ such that for all $M > 0$,
    $$ 
    \limsup_{N \to \infty} \IP\left\{ \sup_{0 \le t \le T} \langle f, \lp^N_t \rangle > M \right\}
    <
    \frac{K(f,T)}{M} .
    $$
\end{lemma}

\begin{proof}[Proof of Lemma~\ref{lem:compact containment hn}:]
    Let $h$ be as in Lemma~\ref{lem:h_containment},
    so there is a $c_f < \infty$ such that $f(x, u) \le c_f h(u)$ for all $x$ and $u$.
    Therefore, $\langle f, \xi \rangle \le c_f \langle h, \xi \rangle$,
    and so by Lemma~\ref{lem:h_containment},
    \begin{align*}
        \limsup_{N \to \infty} \IP\left\{ \sup_{0 \le t \le T} \langle f, \lp^N_t \rangle > M \right\}
        \le
        \limsup_{N \to \infty} \IP\left\{ \sup_{0 \le t \le T} \langle h, \lp^N_t \rangle > M/c_f \right\}
        <
        \frac{K(T) c_f}{M} .
    \end{align*}
\end{proof}

\begin{lemma}[Compact containment for $\lp$] \label{lem:compact containment lookdown}
    Let $f_1, f_2, \ldots$ be a sequence of functions each satisfying the conditions
    of Lemma~\ref{lem:compact containment hn}.
    Under the assumptions of Lemma~\ref{lem:h_containment},
    for any $T$ and $\delta > 0$ there exists a sequence
    $(C_1, C_2, \ldots)$ of finite constants such that
    \begin{align} \label{eqn:compact eqn lookdown}
    \limsup_{N \to \infty}
    \IP\left\{
        \sup_{0 \le t \le T} \langle f_k, \lp^N_t \rangle > C_k
        \text{ for some } k \ge 1
    \right\}
    < \delta .
    \end{align}
\end{lemma}

In other words, the processes $\lp^N$ stay in 
the set
$$
    \left\{
        \lp \in \lpmeasures \;:\;
        \langle f_k, \lp \rangle \le C_k
        \text{ for all } k \ge 1
    \right\} ,
$$
for all $0 \le t \le T$
with uniformly high probability,
a set which (as discussed in the proof of Theorem~\ref{thm:lookdown_convergence})
is relatively compact for an appropriate choice of $\{f_k\}_{k \ge 1}$.

\begin{proof}[Proof of Lemma~\ref{lem:compact containment lookdown}:]
    By a union bound,
    $$ 
    \IP\left\{
        \sup_{0 \le t \le T} \langle f_k, \lp^N_t \rangle > C_k
        \text{ for some } k \ge 1
    \right\}
    \le
    \sum_{k \ge 1}
    \IP\left\{
        \sup_{0 \le t \le T} \langle f_k, \lp^N_t \rangle > C_k
    \right\} ,
    $$
    so \eqref{eqn:compact eqn lookdown} follows
    by taking $C_k = 2^{k-1} K(f_k, T) / \delta$ and using Lemma~\ref{lem:compact containment hn}.
\end{proof}

Finally, we prove the key lemma.

\begin{proof}[Proof of Lemma~\ref{lem:h_containment}:]
    Applied to $f(x, u) = h(u)$, the martingale representation~\eqref{eqn:f_xi_mgale} is
    \begin{align*}
    \begin{split}
        \langle h, \lp_t^N \rangle
        &=
        \langle h, \lp_0^N \rangle
        + M^{h}_t 
        \\ & \qquad {}
        +
        \int_0^t
        \big\langle
            2 c_\theta(x, \eta^N_s) \int_u^N h(v) dv
        ,
        \lp_s^N(dx, du) \big \rangle
        ds
        \\ & \qquad \qquad {}
        +
        \int_0^t
        \big\langle
            \left(
                c_\theta(x, \eta^N_s) u^2
                - b_\theta(x, \eta^N_s) u
            \right)
            h'(u)
        ,
        \lp_s^N(dx, du) \big \rangle
        ds ,
    \end{split}
    \end{align*}
    where $M^{h}_t$ is a martingale
    with angle bracket process
    \begin{align*}
        \left\langle M^{h} \right\rangle_t
        &=
        \int_0^t
        \langle 2 c_\theta(x, \eta_s^N) \int_u^N h(v)^2 dv
        ,
        \lp_s^N(dx, du) \big \rangle
        ds .
    \end{align*}
    Now, note that $0 \le c_\theta(x, \eta_x^N) \le C_a < \infty$
    and $b_\theta(x, \eta_s^N) \le C_b < \infty$,
    and we have assumed that $h'(u) \le 0$ (since $h$ is nonincreasing),
    so we may bound
    \begin{align} \label{eqn:hn_bound}
    \begin{split} 
        \langle h, \lp_t^N \rangle
        &\le
        \langle h, \lp_0^N \rangle
        + M^{h}_t
        \\ & \qquad {}
        +
        \int_0^t
        \big\langle
            2 C_a \int_u^\infty h(v) dv
            + \left( C_a u^2 + C_b u \right) |h'(u)|,
        \lp_s^N(dx, du) \big \rangle
        ds .
    \end{split}
    \end{align}
    Now, since $\lp_t^N$ is conditionally uniform given $\eta_t^N$
    in the sense of~\eqref{eqn:conditionally_uniform},
    we know that for compactly supported $f$,
    $\IE[\langle f, \lp^N_t \rangle] = \IE[\langle \widetilde{f}_N, \eta^N_t \rangle]$,
    where $\widetilde{f}_N(x) = \int_0^N f(x,u) du$.
    By our assumptions on $h$,
    we know that
    $$
        \int_0^\infty \left( 2 C_a \int_u^\infty h(v) dv
        + \left( C_a u^2 + C_b u \right) |h'(u)| \right) du
        < C
    $$
    for some $C < \infty$, and so (by dominated convergence)
    \begin{align*}
    \begin{split}
        \IE\left[ \langle h, \lp_t^N \rangle \right]
        &\le
        \IE\left[ \langle h, \lp_0^N \rangle \right]
        +
        C
        \int_0^t
        \IE\bigg[
            \big\langle
            1
            ,
            \eta_s^N \big \rangle
        \bigg]
        ds ,
    \end{split}
    \end{align*}
    which we know by Lemma~\ref{lem:eta_f_bound}
    is bounded by $C_0 e^{C_1 t}$ for some other constants $C_0$ and $C_1$.

    Now consider the maximum. By \eqref{eqn:hn_bound},
    using that the integrand is nonnegative,
    \begin{align*}
        \sup_{0 \le t \le T} \langle h, \lp_t^N \rangle
        &\le
        \langle h, \lp_0^N \rangle
        + \sup_{0 \le t \le T} M^{h}_t 
        \\ & \qquad {}
        + \int_0^T 
        \big \langle
            2 C_a \int_u^\infty h(v) dv
            + \left( C_a u^2 + C_b u \right) |h'(u)|
        ,
        \lp_s^N(dx, du) \big \rangle
        ds .
    \end{align*}

    As in the proof of Lemma~\ref{lem:eta_compact_containment},
    the Burkholder-Davis-Gundy inequality,
    \citet{barlow/jacka/yor:1986},
    and the fact that $\sqrt{x} \le 1 + x$ for $x \ge 0$
    tells us that there is a $C'$ such that
    \begin{align*}
    \IE\left[ \sup_{0 \le t \le T} M^{h}_t \right]
    & \le
        C'\left( 1 + \IE\left[ \langle M^{h} \rangle_T \right] \right)
    \\ & \le
        C' \left( 1 + \int_0^T \IE\left[
        \langle 2 c_\theta(x, \eta_s^N) \int_u^\infty h(v)^2 dv
        ,
        \lp_s^N(dx, du) \big \rangle
        \right] ds \right) 
    \\ & \le
        C' \left( 1 + 2 C_a \int_0^\infty h(v)^2 dv
        \int_0^T \IE\left[
            \langle 1 , \lp_s^N(dx, du) \big \rangle
        \right] ds \right)  
    \\ & \le
        C_2 e^{C_1 T} ,
    \end{align*}
    for a constant $C_2$ which is finite
    by our assumption that $\int_0^\infty h(v)^2 dv < \infty$.

    Therefore, 
    \begin{align*}
        \IE\left[ \sup_{0 \le t \le T} \langle h, \lp_t^N \rangle \right]
        &\le
        \IE\left[ \langle h, \lp_0^N \rangle \right]
        + (C_2 + C_0 / C_1) e^{C_1 T},
    \end{align*}
    and so
    \begin{align*}
        \IP\left\{ \sup_{0 \le t \le T} \langle h, \lp_t^N \rangle > K \right\}
        &\le
        \frac{
            \IE\left[ \langle h, \lp_0^N \rangle \right]
            + (C_2 + C_0 / C_1) e^{C_1 T}
        }{ K } .
    \end{align*}
\end{proof}

\begin{lemma}
    \label{lem:test_fn_tightness}
    Let $f$ be a bounded, continuous real-valued function on $\IR^d \times [0, \infty)$
    with uniformly bounded first and second derivatives
    for which there exists a $u_0$ such that if $u > u_0$ then $f(x, u) = 0$.
    Then, the sequence of real-valued processes $(\langle f, \lp_t^N \rangle)_{t \ge 0}$ for $N \ge 1$
    is tight in ${\cal D}_{[0,\infty)}(\IR)$.
\end{lemma}

\begin{proof}[Proof of Lemma \ref{lem:test_fn_tightness}:]
    Again, we use the Aldous-Rebolledo criterion.
    Tightness of $\langle f, \lp_t \rangle$ for a fixed $t$
    follows from Lemma~\ref{lem:compact containment hn},
    so we need only prove conditions analogous to
    \eqref{eqn:eta_projections_goal1}
    and \eqref{eqn:eta_projections_goal2}
    applied to the martingale representation
    of equations~\eqref{eqn:f_xi_mgale} and~\eqref{eqn:f_xi_qv}.
    Rewriting~\eqref{eqn:f_xi_mgale} with $c_\theta = c_\theta(x, \eta_s)$,
    \begin{align*}
        \langle f, \lp_t \rangle
        &=
        \langle f, \lp_0 \rangle
        + M^f_t
        + \int_0^t \Big\langle
            c_\theta \int_u^N \int (f(y, u_1) + f(x, u_1)) q^m(x,dy,\eta) du_1
        \\ {} &\qquad
            + c_\theta (N - u) \int_0^t (f(y, u) - f(x, u)) q^m(x,dy,\eta)
            + (c_\theta u^2 - b_\theta u) \frac{d}{du} f,
            \lp_s
        \Big\rangle ds .
    \end{align*}
    The bounds analogous to \eqref{eqn:eta_projections_goal1} and \eqref{eqn:eta_projections_goal2}
    follow as in the proof of Lemma~\ref{lem:eta_projections_tightness}:
    for instance, observe that using that $c_\theta \le C_a$ for some $C_a$,
    the predictable part of this semimartingale decomposition is bounded by
    \begin{align*}
        \Big\langle
            2 C_a \|f\|_\infty u_f
            + (1 - u/N) \gamma B^\theta_f
            + (C_a u^2 - b_\theta u) \frac{d}{du} f,
            \lp_s
        \Big\rangle ,
    \end{align*}
    the last term of which is bounded by
    \begin{align*}
        \langle
            C_a u_f^2 + \sup_x |b_\theta(x, \eta_s)| u_f \|\frac{d}{du} f\|_\infty
        \rangle ,
    \end{align*}
    which can be bounded as we did for \eqref{eqn:eta_projections_goal1}.
\end{proof}

\subsection{Motion of ancestral lineages}
\label{sec:lineages_proof}

In this section we prove Theorem~\ref{thm:lineages}.
The argument follows directly from the discussion in Section \ref{sec:limiting_lines_of_descent}.

\begin{proof}[Proof of Theorem~\ref{thm:lineages}:]
For brevity, in the proof we write $\gamma(x)$ or $\gamma$ for $\gamma(x,\eta)$.

Here we have taken the high-density, deterministic limit
(so, $\theta, N \to \infty$ and $\theta/N \to 0$).
We first proceed informally,
as if the limiting process has a density $\varphi_t(x)$ at location $x$ and time $t$ (which it may not),
and follow this with an integration against test functions to make the argument rigorous.
Let $Y$ denote the spatial motion followed by a single line of descent.
Above equation~\eqref{eqn:limiting_generator},
we showed that $Y$ is a diffusion with generator at time $s$
$$
    \mathcal{L}^Y_s g(x) = \gamma(x,\eta_s) ( \DG(r(\cdot,\eta_s) g(\cdot))(x) - g(x) \DG r(x,\eta_s) ) .
$$
The diffusion is time-inhomogeneous if the density is not constant in time.
Let $\varphi_t(x)$ be the limiting density,
which is a weak solution to \eqref{general deterministic limit},
$\partial_t \varphi_t = r \DG^*[ \varphi_t \gamma ] + \varphi_t F$.
Formally, the intensity of individuals at $y$ at time $t$
that are descended from individuals that were at $x$ at time $s$
(with $s < t$) is
\begin{equation} \label{eqn:formal_intensity}
    \varphi_s(x) \IE_{s,x} \left[
        \exp\left(
            \int_s^t (F + \gamma \DG r)(Y_u) du
        \right)
        \ind_{Y_t = y}
    \right]
    dy ,
\end{equation}
where the subscript $s, x$ in the expectation indicates that $Y_s = x$.
To see why this should be true, 
suppose that an ancestor at time $s$ has level $v$. Conditional on its 
spatial motion $\{Y_u\}_{s\leq u\leq t}$, its level at time $t$ will
be $v \exp(-\int_s^t(F+\gamma\DG r)(Y_u)du)$. This will be less than a given level 
$\lambda$ if $v < \lambda \exp(\int_s^t(F+\gamma\DG r)(Y_u)du)$. 
The intensity of levels at $y$ that are descended from individuals at
$x$ can therefore be obtained as the limit as $\lambda\to\infty$ of 
$1/\lambda$ times the number of levels at $x$ at time $s$ with
$u<\lambda \exp(\int_s^t(F+\gamma\DG r)(Y_u)du)$ and for which
the corresponding individual is at $y$ at time $t$, which is 
precisely the quantity in~\eqref{eqn:formal_intensity}. 

By our construction in Section~\ref{sec:limiting_lines_of_descent},
when we integrate~\eqref{eqn:formal_intensity}
with respect to $x$ we recover $\varphi_t(y)dy$. 
Consider an individual sampled at location $y$ at time $t$,
and write $p(t,s,y,x)$ for the probability density
that their ancestor at time $s$ was at $x$.
As a consequence of~\eqref{eqn:formal_intensity},
still formally,
\begin{equation}
\label{eqn:ptsyx}
    p(t,s,y,x)
    =
    \frac{\varphi_s(x)}{\varphi_t(y)}
    \IE_{s,x}\left[
        \exp\left( \int_s^t (F + \gamma\DG r)(Y_u) du \right)
        \ind_{Y_t=y}
    \right]
\qquad \text{ for } s < t.
\end{equation}

To make~(\ref{eqn:ptsyx}) meaningful, we multiply by suitable test functions
$f$ and $g$ and integrate.
\begin{align*}
&\int \int f(y) \varphi_t(y) p(t,s,y,x) g(x) dy dx \\
&\qquad =
\int g(x) \varphi_s(x)
\IE_{x,s}\left[
    \exp\left(
        \int_s^t(F+\gamma\DG r)(Y_u)du
    \right)f(Y_t)
\right] dx .
\end{align*}

Writing $\hat{T}_{t,s}$ for the time-inhomogeneous semigroup
corresponding to the motion of ancestral lineages backwards in time
(that is, $\hat{T}_{t,s} f(y) = \int p(t,s,x,y) f(x) dy$),
we can write this as 
\begin{align} \label{eqn:integrated_semigroups}
    \int f(y)\varphi_t(y)\hat{T}_{t,s}g(y)dy
    =
    \int g(x) \varphi_s(x)
        \IE_{s,x} \left[
            \exp\left(
                \int_s^t(F+\gamma\DG r)(Y_u)du
            \right)f(Y_t)
        \right]dx.
\end{align}
Next, we will differentiate this equation with respect to $t$.
There are two terms in the product on the left-hand side that depend on $t$,
so if we
use that $\partial_t \varphi_t = r \DG^*[ \varphi_t \gamma ] + \varphi_t F$
(in a weak sense),
and write $\Lgen_u$ for the generator of $\hat{T}_{t,s}$ at time $t=u$
so that $\partial_t \hat{T}_{t,s} g(y) \Big|_{t=s} = \Lgen_s g(y)$,
then
\begin{multline*}
    \qquad \qquad
    \frac{d}{dt} 
    \int f(y)\varphi_t(y)\hat{T}_{t,s}g(y)dy
    \Big|_{t=s}
    \\ {}
    =
    \int f(y) \left\{
        \varphi_s(y) \Lgen_s g(y)
        + \left[
            r(y) \DG^*(\gamma \varphi_s)(y) + \varphi_s(y) F(y)
        \right] g(y)
    \right\} dy .
\end{multline*}
As for the right-hand side, since $Y_s=x$ under $\IE_{x,s}$,
\begin{multline*}
    \frac{d}{dt}
        \IE_{x,s}\left[
            \exp\left(
                \int_s^t(F+\gamma\DG r)(Y_u)du
            \right)f(Y_t)
        \right]
    \Bigg|_{t=s}
    =
        \left[F(x)+\gamma(x) \DG r(x)\right] f(x) + \Lgen^Y_s f(x) .
\end{multline*}
Therefore, the derivative of~\eqref{eqn:integrated_semigroups}
(with respect to $t$, evaluated at $t=s$) is
\begin{align*}
&
    \int f(y) \left\{
        \varphi_s(y) \Lgen_s g(y)
        + \left(
            r(y) \DG^*(\gamma \varphi_s)(y) + \varphi_s(y) F(y)
        \right) g(y)
    \right\} dy \\
&\qquad =
    \int g(x) \varphi_s(x) \left(
        \Lgen^Y_s f(x) + \left[F(x)+\gamma(x) \DG r(x)\right] f(x)
    \right) dx \\
&\qquad =
    \int f(x) \left(
        (\Lgen^Y_s)^* (\varphi_s g)(x) + \left[F(x) + \gamma(x) \DG r(x)\right] \varphi_s(x) g(x)
    \right) dx ,
\end{align*}
where $(\Lgen^Y_s)^*$ is the adjoint of $\Lgen^Y_s$ .
Since $f$ was arbitrary,
\begin{eqnarray*}
\Lgen_s g
    &=&
    \frac{1}{\varphi_s} \left[
        (\Lgen^Y_s)^* (\varphi_s g)
        + \gamma \varphi_s g \DG(r)
        - r g \DG^*(\gamma \varphi_s)
    \right] .
\end{eqnarray*}
(Note that the $\varphi_s F g$ terms have cancelled.)
Since the adjoint of ${\mathcal L}^Y_s$ is
\begin{align*}
    ({\mathcal L}^Y_s)^* f
    &=
    r \DG^* (\gamma f) - \gamma f \DG r ,
\end{align*}
we can rewrite the generator of a lineage as 
\begin{eqnarray*}
\Lgen_s g
    &=&
    \frac{r}{\varphi_s} \left[
        \DG^* (\gamma \varphi_s g)
        - g \DG^*(\gamma \varphi_s)
    \right] .
\end{eqnarray*}
This is equation \eqref{eqn:lineage_generator}.

To simplify to equation \eqref{eqn:lineage_generator2},
first define
$
    \DD f(x) = \sum_{ij} \covq_{ij} \partial_{ij} f(x),
$
and so the adjoint of $\DD$ is
$$
    \DD^* f(x)
    =
    \sum_{ij} \partial_{ij} (\covq_{ij} f(x)) .
$$
Note that $\DD^*$ satisfies the following identity:
\begin{align*}
    \DD^*(fg)
    &=
    \sum_{ij} \left\{
        g \partial_{ij} (\covq_{ij} f)
        + 2 f \partial_{i} (\covq_{ij}) \partial_j(g)
        + 2 \covq_{ij} \partial_{i} (f) \partial_j(g)
        + \covq_{ij} f \partial_{ij} g
    \right\} \\
    &=
    g \DD^* f
    + 2 f \vec{c} \cdot \grad g
    + 2 (\covq \grad f) \cdot \grad g
    + f \DD g ,
\end{align*}
where $\vec{c}_j = \sum_i \partial_i \covq_{ij}$.
So, with $f = \gamma \varphi_s$,
\begin{eqnarray*}
\Lgen_s g
    &=&
    \frac{r}{\varphi_s} \left[
        \frac{1}{2} \DD^*(\gamma \varphi_s g) - \grad \cdot (\gamma \varphi_s g \meanq)
        - \frac{1}{2} g \DD^*(\gamma \varphi_s) + g \grad \cdot (\gamma \varphi_s \meanq)
    \right] \\
    &=&
    \frac{r}{\varphi_s} \left[
        \frac{1}{2} \gamma \varphi_s \DD g
        + \gamma \varphi_s \vec{c} \cdot \grad g
        + (\covq \grad (\gamma \varphi_s)) \cdot \grad g
        - \gamma \varphi_s \meanq \cdot \grad g
    \right] \\
    &=&
    r \gamma \left[
        \frac{1}{2} \DD g
        + \vec{c} \cdot \grad g
        + (\covq \grad \log(\gamma \varphi_s)) \cdot \grad g
        - \meanq \cdot \grad g
    \right] ,
\end{eqnarray*}
which is equation \eqref{eqn:lineage_generator2}.
\end{proof}

\begin{proof}[Proof of Corollary~\ref{cor:stationary_dist}:]
For the moment, we will write $r(x)$ for $r(x,\eta)$
and $\gamma(x)$ for $\gamma(x,\eta)$.
First note that
since in this case the semigroup does not depend on time, 
we can write $\Lgen = \Lgen_s$, and
$$
    \Lgen f = \frac{\sigma^2}{2} r \gamma \left(
        \Delta f
        + \grad ( 2 \log(\gamma \varphi) - 2h/\sigma^2 ) \cdot \grad f
    \right) .
$$
Now, observe that
$$
    \int_{\IR^d} e^{H(x)} f(x) (\Delta + \grad H(x) \cdot \grad) g(x) dx
    =
    - \int_{\IR^d} e^{H(x)} \left\{
        \grad f(x) \cdot \grad g(x)
    \right\} dx ,
$$
so that by choosing $H(x) = 2 \log(\gamma(x) \varphi(x)) - 2h(x)/\sigma^2$
and
$$ \pi(x)
    = \frac{e^{H(x)} }{\sigma^2 r(x) \gamma(x)/2}
    = \frac{\gamma(x) \varphi(x)^2 e^{-2h(x)/\sigma^2}}{\sigma^2 r(x)/2} ,
$$
we have that
$$
    \int_{\IR^d} \pi(x) f(x) \Lgen g(x) dx
    =
    - \int_{\IR^d} e^{H(x)} \grad f(x) \cdot \grad g(x) dx .
$$
Since this Dirichlet form is symmetric in $f$ and $g$,
the process $Y$ is reversible with respect to $\pi$
(and the factor of $\sigma^2/2$ is constant).
\end{proof}

\section*{Acknowledgements}
Many thanks to Matthias Birkner, Matthias Winkel, and an anonymous reviewer
for detailed comments and corrections.
Thanks go to Gilia Patterson for identifying the ``clumping'' phenomenon,
and to Marcin Bownick and David Levin for useful discussions.
AME thanks everyone in MAPS at Universit\'e Paris Cit\'e for their hospitality
during the period in which much of this research took place.
AME and PLR also thank the Kavli Institute for Theoretical Physics
for their hospitality and birdwatching opportunities;
this research was therefore supported in part by the NSF under grant \#PHY-1748958
and by the Gordon and Betty Moore Foundation grant \#2919.02, both to KITP.
PLR was supported by the NIH NHGRI (grant \#HG011395),
IL by the ANID/Doctorado en el extranjero doctoral scholarship, grant \#2018-72190055,
and TTHL by the EPSRC Centre for Doctoral Training in Mathematics of Random Systems: Analysis, Modelling and Simulation (EP/S023925/1)
the Deutsche Forschungsgemeinschaft under Germany's Excellence Strategy, EXC-2047/1-390685813,
the Rhodes Trust and St.~John's College, Oxford.

\newpage
\appendix

\section{Markov Mapping Theorem}
\label{apx:mmt}

The following appears as Theorem A.2 in \citet{etheridge/kurtz:2019},
specialized slightly here to the case that the processes 
are c\`adl\`ag and have no fixed points of discontinuity.
For an $S_0$-valued, measurable process $Y$, $\hat{\mathcal{F}}^Y_t$
denotes the completion of the $\sigma$-algebra generated by
$Y(0)$ and $\{\int_0^r h(Y(s)) ds, r \le t, h \in B(S_0)\}$.
Also, let $D_S[0,\infty)$ denote the space of c\`adl\`ag, $S$-valued functions
with the Skorohod topology, and $M_S[0,\infty)$ the space of Borel measurable functions
from $[0,\infty)$ to $S$,
topologized by convergence in Lesbegue measure.
For other definitions see \citet{etheridge/kurtz:2019}.

\begin{theorem}[Markov Mapping Theorem] \label{thm:mmt}
    Let $(S,d)$ and $(S_0,d_0)$ be complete, separable metric spaces.
    Let $A \subset C_b(S) \times C(S)$ and $\psi \in C(S)$, $\psi \ge 1$.
    Suppose that for each $f \in \mathcal{D}(A)$ there exists $c_f$
    such that
    $$ |Af(x)| \le c_f \psi(x), \qquad x \in A, $$
    and define $A_0 f(x) = Af(x) / \psi(x)$.

    Suppose that $A_0$ is a countably determined pre-generator,
    and suppose that $\mathcal{D}(A) = \mathcal{D}(A_0)$
    is closed under multiplication and is separating.
    Let $\gamma : S \to S_0$ be Borel measurable,
    and let $\alpha$ be a transition function from $S_0$ into $S$
    ($y \in S_0 \to \alpha(y,\cdot)\in\mathcal{P}(S)$ is Borel measurable)
    satisfying $\int h\circ\gamma(x)\alpha(y,dx) = h(y)$ for $y \in S_0$ and $h \in B(S_0)$,
    that is, $\alpha(y,\gamma^{-1}(y)) = 1$.
    Assume that $\tilde \psi(y) \equiv \int_S \psi(z) \alpha(y,dz) < \infty$ for each $y \in S_0$
    and define
    \begin{align} \label{eqn:C_defn}
        C = \{
        \int_S f(z) \alpha(\cdot,dz),
        \int_S Af(z) \alpha(\cdot,dz)
        \; : \; f \in \mathcal{D}(A)
    \} . \end{align}
    Let $\mu_0 \in \mathcal{P}(S_0)$ and define $\nu_0 = \int \alpha(y,\cdot) \mu_0(dy)$.
    \begin{itemize}
        \item[(a)]
            If $\tilde Y$ satisfies $\int_0^t \IE[\tilde \psi(\tilde Y(s))]ds < \infty$
            for all $t \ge 0$ and $\tilde Y$ is a solution of the martingale problem for $(C,\mu_0)$,
            then there exists a solution $X$ of the martingale problem for $(A,\nu_0)$
            such that $\tilde Y$ has the same distribution on $M_{S_0}[0,\infty)$
            as $Y = \gamma \circ X$.
            If $Y$ and $\tilde Y$ are c\`adl\`ag, then $Y$ and $\tilde Y$ have the same distribution
            on $D_{S_0}[0,\infty)$.
        \item[(b)]
            For $t \ge 0$,
            $$ \IP\{ X(t) \in \Gamma \;|\; \hat{\mathcal{F}}^Y_t \}
            = \alpha(Y(t),\Gamma), \qquad \text{for } \Gamma \in \mathcal{B}(S).
            $$
        \item[(c)]
            If, in addition, uniqueness holds for the martingale problem for $(A,\nu_0)$,
            then uniqueness holds for the $M_{S_0}[0,\infty)$-martingale problem for $(C,\mu_0)$.
            If $\tilde Y$ has sample paths in $D_{S_0}[0,\infty)$ then uniqueness holds for the
            $D_{S_0}[0,\infty)$-martingale problem for $(C,\mu_0)$.
        \item[(d)]
            If uniqueness holds for the martingale problem for $(A,\nu_0)$
            then $Y$ is a Markov process.
    \end{itemize}
\end{theorem}

In our application, we have taken $S$ to be the space of locally finite counting measures
on $\IR^d \times [0,N)$ or on $\IR^d \times [0,\infty)$,
and $S_0$ the space of finite measures on $\bar \IR^d$.
Then, $A$ corresponds to the generator for the lookdown process (i.e., either $A^N$ or $A$),
and $C$ corresponds to the generator for the spatial population process
(i.e., either $\Pgen^N$ or $\Pgen$).
The ``$\gamma$'' of the theorem is our spatial projection operator
that we have called $\kappa^N$ or $\kappa$,
and the ``$\alpha$'' of the theorem will be named $\Gamma^N$ or $\Gamma$ below.
Finally, ``$X$'' of the theorem is our lookdown process, $\lp$,
and ``$Y$'' is our spatial process, $\eta$.

\subsection{Lookdown Generators}

In this section we verify one of the conditions of the Markov Mapping Theorem,
namely, that ``integrating out levels'' in the generator of the lookdown process
we obtain the generator of the projected process.
In the notation of the theorem,
we are verifying that $C$ defined in~\eqref{eqn:C_defn} is in fact $\Pgen^N$
(if defined with $A^N$) or $\Pgen^\infty$ (if defined with $A$).
We will work with test functions of the form
\begin{equation} \label{eqn:f_defn}
    f(\lp) = \prod_{(x, u) \in \lp} g(x, u) = \exp\left(\langle \log g, \lp \rangle \right) ,
\end{equation}
where $0 \le g \le 1$ and $g(x,u) = 1$ for all $u \ge u_g$ for some $u_g < \infty$.
Furthermore, recall that $\kappa^N(\lp)(\cdot) = \lp(\cdot \times [0, N)) / N$
is the ``spatial projection operator'',
and define the transition function $\Gamma^N : \measures \to \mathcal{M}(\IR^d \times [0,N))$
so that for $\eta \in \measures$, if $\hat g_N(x) = \int_0^N g(x, u) du / N$, then
\begin{align*}
    F^N_g(\eta)
    &:=
    \int f(\lp) \Gamma^N(\eta, d\lp) \\
    &=
    \exp\left(
        N \left\langle
            \log \frac{1}{N} \int_0^N g(x, u) du, \eta(dx)
        \right\rangle
    \right) \\
    &=
    \exp\left(
        N \left\langle \log \hat g_N(x), \eta(dx) \right\rangle
    \right) ,
\end{align*}
i.e., $\Gamma^N$ assigns independent labels on $[0, N]$ to each of the points in $\eta$.
It follows from Lemma~\ref{def: MP definition of limit} that for test functions of this form
the generator of $\eta^N_t$ is
\begin{align} \label{eqn:pgen_defn_again}
    \begin{split}
    \Pgen^N F^N_g(\eta)
    =
    F^N_g(\eta)
        N \theta & \bigg\langle
        \gamma(x, \eta) \int r(z, \eta) \left( \hat g_N(z) - 1 \right) q_\theta(x, dz)
    \\ &\qquad \qquad {}
        +
        \mu_\theta(x, \eta) \left( \frac{1}{\hat g_N(x)} - 1 \right)
        ,
        \eta(dx)
    \bigg\rangle .
    \end{split}
\end{align}
(Note that $f$ here differs from the $f$ used in Lemma~\ref{def: MP definition of limit}
so as to agree with standard usage in the literature on lookdown processes.)
The generator of $\lp^N_t$ is $A^N$, defined in equation~\eqref{eqn:lookdown_generator}.

\begin{lemma}
    \label{lem:averaged_generators}
    For all finite counting measures $\eta$ on $\IR^d$,
    if $f$ is of the form~\eqref{eqn:f_defn}, then
    \begin{align}
        \int A^N f(\lp) \Gamma^N(\eta, d\lp)
        &=
        \Pgen^N F^N_g(\eta) .
    \end{align}
\end{lemma}

For the limiting process,
recall that $\kappa(\lp)(\cdot) = \lim_{u \to \infty} \lp(\cdot \times [0, u)) / u$
is the ``spatial projection operator'',
and define the probability kernel $\Gamma : \measures \to \flpmeasures$
so that for $\eta \in \measures$,
defining $\widetilde{g}(x) = \int_0^\infty (g(x, u) - 1) du$,
\begin{align*}
    F_g(\eta)
    &:=
    \int f(\lp) \Gamma(\eta, d\lp) \\
    &=
    \exp\left(
        \big \langle
            \int_0^\infty (g(x, u) - 1) du,
            \eta(dx)
        \big \rangle
    \right) \\
    &=
    e^{ \langle \widetilde{g}(x), \eta(dx) \rangle } .
\end{align*}
i.e., $\Gamma(\eta, \cdot)$ is the distribution of a conditionally Poisson process
with intensity a product of $\eta$ and Lebesgue measure.
It again follows from Lemma~\ref{def: MP definition of limit} that for test functions of this form
the generator of $\eta_t$ is
\begin{align} \label{eqn:pgen_infty}
    \Pgen^\infty F_g(\eta)
    &=
    F_g(\eta)
    \left\langle
        \gamma(x, \eta)
        \DG( \widetilde{g}(\cdot) r(\cdot) )(x)
        + F(x, \eta) \widetilde{g}(x)
        + \alpha \gamma(x, \eta) r(x, \eta) \widetilde{g}^2(x)
        ,
        \eta(dx)
    \right\rangle .
\end{align}
The generator of $\lp_t$ is $A$, defined in equation~\eqref{eqn:limiting_lookdown_generator}.

\begin{lemma}
    \label{lem:averaged_generators_limit}
    For all $\eta \in \measures$,
    if $f$ is of the form~\eqref{eqn:f_defn}, then
    \begin{align}
        \int A f(\lp) \Gamma(\eta, d\lp)
        &=
        \Pgen^\infty F_g(\eta) .
    \end{align}
\end{lemma}

\begin{proof}[Proof of Lemma~\ref{lem:averaged_generators}:]
First, break the generator $A^N$ into three parts, 
\begin{align*}
    A^N_1f(\lp)
    &=
    f(\lp) \sum_{(x,u)\in\lp} 2 c_\theta(x,\eta)
    \int_u^{N}\Bigg(
        \frac{1}{2}\frac{g(x,v_1)}{g(x,u)}
        \int_{\IR^d} (g(y,u)-g(x,u))q^m_{\theta}(x,dy,\eta)
    \Bigg)dv_1 ,
    \\
    A^N_2f(\lp)
    &=
    f(\lp) \sum_{(x,u)\in\lp} 2 c_\theta(x,\eta)
    \int_u^{N}\Bigg(
        \frac{1}{2}\int_{\IR^d}
            \left(\frac{g(y,v_1) + g(x,v_1)}{2} - 1\right)
        q^m_{\theta}(x,dy,\eta)
    \Bigg)dv_1 ,
    \\
    A^N_3f(\lp)
    &=
    f(\lp) \sum_{(x,u)\in\lp}\,
    \left( c_\theta(x,\eta) u^2 - b_{\theta}(x,\eta)u \right) \frac{\partial_u g(x,u)}{g(x,u)},
\end{align*}
where $q^m$ was defined in equation~\eqref{eqn:qm_defn},
so that 
$$ A^Nf(\lp) = A^N_1f(\lp) + A^N_2f(\lp) + A^N_3f(\lp). $$

We now integrate each piece against $\Gamma^N$.
First note that by the product form of $f$,
$$
    \int f(\lp) \sum_{(x,u)\in\lp} \frac{\ell(x,u)}{g(x,u)} \Gamma^N(\eta,d\lp)
    =
    F_g^N(\eta) \bigg\langle
        \frac{1}{N \hat g_N(x)} \int_0^N \ell(x,u) du,
        N \eta(dx)
    \bigg\rangle .
$$
Therefore,
\footnotesize
\begin{align*}
&
    \int A^N_1f(\lp) \Gamma^N(\eta,d\lp) 
\\ &\qquad =
    F_g^N(\eta) \bigg\langle \frac{c_\theta(x,\eta)}{\hat g_N(x)} \int_0^N \Bigg\{
        \int_u^N \Bigg(
            \frac{1}{2} g(x,v_1) \int_{\IR^d}(g(y,u)-g(x,u))q^m_{\theta}(x,dy,\eta)
        \Bigg) dv_1
    \Bigg\} du, \eta(dx) \bigg \rangle
\\ &\qquad =
    F_g^N(\eta) \bigg\langle \frac{c_\theta(x,\eta)}{\hat g_N(x)} \Bigg\{
    \int_{\IR^d} \left\{
        \int_{0}^{N} \int_u^N g(x,v_1) (g(y,u)-g(x,u)) dv_1 du
    \right\} q^m_{\theta}(x,dy,\eta)
    , \eta(dx) \bigg \rangle .
\end{align*}
\normalsize
For the second generator, we have
\footnotesize
\begin{align*}
&
    \int A^N_2 f(\lp) \Gamma^N(\eta, d\lp)
\\ &\qquad =
    F_g^N(\eta) \bigg\langle \frac{c_\theta(x,\eta)}{\hat g_N(x)} \int_0^N g(x, u) \Bigg\{
        \int_u^N\Bigg(\int_{\IR^d}
        \left(\frac{g(y,v_1)+g(x,v_1)}{2}-1\right) q^m_{\theta}(x,dy,\eta)
        \Bigg)dv_1 \Bigg\}
    , \eta(dx) \bigg\rangle
\\ &\qquad =
    F_g^N(\eta) \bigg\langle \frac{c_\theta(x,\eta)}{\hat g_N(x)}
    \int_{\IR^d}\Bigg\{ \int_{0}^{N}
    \int_u^N g(x,u) \left(
        g(y,v_1)+g(x,v_1)-2
    \right) dv_1 du\Bigg\} q^m_{\theta}(x,dy,\eta)
    , \eta(dx) \bigg\rangle .
\end{align*}  
\normalsize
For the third generator we have that
\begin{align*}
&
    \int A^N_3f(\xi)\Gamma^N(\eta,d\lp)
\\ & \qquad =
    F_g^N(\eta) \bigg\langle \frac{1}{\hat g_N(x)}
    \int_0^N
        \left( c_\theta(x,\eta) u^2 - b_\theta(x,\eta) u \right) \partial_u g(x,u)
    du
    , \eta(dx) \bigg\rangle 
\\ & \qquad =
    F_g^N(\eta) \bigg\langle \frac{1}{\hat g_N(x)}
    \int_0^N
        \left( b_\theta(x,\eta) - 2 c_\theta(x,\eta) u \right) (g(x,u) - 1)
    du
    , \eta(dx) \bigg\rangle .
\end{align*}

Note that $2 \int_0^N \int_u^N g(x,v_1) g(y,u) dv_1 du = N^2 \hat g_N(x) \hat g_N(y)$,
and so
\begin{align*}
&
    \int_{0}^{N} \int_u^N g(x,v_1) (g(y,u)-g(x,u)) dv_1 du
    + \int_{0}^{N} \int_u^N g(x,u) \left(g(y,v_1)+g(x,v_1)-2\right)dv_1 du
\\ &\qquad = 
    N^2 \hat{g_N}(x)(\hat{g_N}(y) - 2) + 2\int_{0}^{N} u g(x,u) du.
\end{align*}

Combining the last equations,
and using the fact that
$N c_\theta(x,\eta) - b_\theta(x,\eta) = \theta \mu_\theta(x,\eta)$,
we have 
\begin{align*}
&
    \int \Bigg( A^N_1 f(\xi) + A^N_2 f(\xi) + A^N_3 f(\xi) \Bigg) \Gamma^N(\eta,d\lp)
\\ &\qquad =
    F_g^N(\eta) \bigg\langle
        c_\theta(x,\eta) N^2 \int_{\IR^d} (\hat g_N(y) - 2) q^m(x,dy,\eta)
        +
        \frac{1}{\hat g_N(x)} c_\theta(x,\eta) \int_0^N 2u du
\\ &\qquad \qquad \qquad {}
        +
        \frac{1}{\hat g_N(x)} b_\theta(x,\eta) \int_0^N (g(x,u) - 1) du
    , \eta(dx) \bigg\rangle 
\\ &\qquad =
    F_g^N(\eta) \bigg\langle
        c_\theta(x,\eta) N^2 \int_{\IR^d} (\hat g_N(y) - 1) q^m(x,dy,\eta)
        +
        N^2 c_\theta(x,\eta) \left(\frac{1}{\hat g_N(x)} - 1\right)
\\ &\qquad \qquad \qquad {}
        +
        N b_\theta(x,\eta) \left(1 - \frac{1}{\hat g_N(x)}\right)
    , \eta(dx) \bigg\rangle 
\\ &\qquad =
    F_g^N(\eta) N \bigg\langle
        N c_\theta(x,\eta) \int_{\IR^d} (\hat g_N(y) - 1) q^m(x,dy,\eta)
        +
        \theta \mu_\theta(x,\eta) \left(\frac{1}{\hat g_N(x)} - 1\right)
    , \eta(dx) \bigg\rangle  .
\end{align*}
This matches equation~\eqref{eqn:pgen_defn_again}, as desired,
because $N c_\theta(x,\eta) q_m(x,dy,\eta) = \theta \gamma(x,\eta) q_\theta(x,dy)$.
\end{proof}

Before proving Lemma~\ref{lem:averaged_generators_limit},
we recall an important equality for conditionally
Poisson point processes (\cite{kurtz/rodrigues:2011} Lemma A.3).

\begin{lemma} \label{lem:poisson_eqn}
If $\lp = \sum_{i}\delta_{Z_i}$ is a Poisson random measure with mean measure $\nu$, 
then for $\ell \in L^{1}(\nu)$ and $g\geq0$ with $\log g \in L^{1}(\nu)$,
\begin{equation}
\IE\left[\sum_{j} \ell(Z_j) \prod_{i}g(Z_i)\right] = \int \ell g d\nu e^{\int (g-1) d \nu}.
\end{equation}
\end{lemma}

\begin{proof}[Proof of Lemma~\ref{lem:averaged_generators_limit}:]
By Lemma~\ref{lem:poisson_eqn},
$$
    \int f(\lp) \sum_{(x,u)\in\lp} \frac{\ell(x,u)}{g(x,u)} \Gamma(\eta, \lp)
    =
    F_g(\eta) \bigg\langle \int_0^\infty \ell(x,u) du , \eta(dx) \bigg\rangle .
$$
Comparing this to the definition of $A$
(equation~\eqref{eqn:limiting_lookdown_generator}),
we see that
\begin{align*}
    \int Af(\lp) \Gamma(\eta,d\lp)
    &=
    F_g(\eta) \bigg\langle
        \int_0^\infty
        (\ell_1(x,u) + \ell_2(x,u) + \ell_3(x,u))
        du, \eta(dx)
    \bigg\rangle ,
\end{align*}
where
\begin{align*}
    \ell_1(x,u)
    &=
    \gamma(x,\eta)\left(
        \DG(g(\cdot,u) r(\cdot,\eta))(x) - g(x,u) \DG r(x,\eta)
    \right) \\
    &=
    \gamma(x,\eta)\left(
        \DG((g(\cdot,u)-1) r(\cdot,\eta))(x) - (g(x,u)-1) \DG r(x,\eta)
    \right)
\end{align*}
and
\begin{align*}
    \ell_2(x,u)
    =
    2 g(x,u) \alpha \gamma(x,\eta) r(x,\eta) \int_u^\infty(g(x,v) - 1)dv
\end{align*}
and
\begin{align*}
    \ell_3(x,u)
    =
    \left( \alpha \gamma(x,\eta)r(x,\eta)u^2 - \{\gamma(x,\eta)\DG r(x,\eta) + F(x,\eta)\} u \right)
    \partial_u g(x,u) .
\end{align*}

First note that since $\DG$ acts on space, it commutes with the integral over levels, and so
\begin{align*}
    \int_0^\infty \ell_1(x,u) du
    &=
    \gamma(x,\eta)\left(
        \DG(\widetilde{g}(\cdot) r(\cdot,\eta))(x) - \widetilde{g}(x) \DG r(x,\eta)
    \right) ,
\end{align*}
since $\widetilde{g}(x) = \int_0^\infty (g(x,u) - 1) du$.
Next,
\begin{align*}
    \int_0^\infty \ell_2(x,u) du
    &=
    \alpha \gamma(x,\eta) r(x,\eta) 2\int_0^\infty g(x,u) \int_u^\infty (g(x,v)-1) dv du .
\end{align*}
Finally, integrating by parts,
\begin{align*}
    \int_0^\infty \ell_3(x,u) du
    &=
    - \alpha \gamma(x,\eta) r(x,\eta) \int_0^\infty 2u(g(x,u) - 1) du
    \\ & \qquad {}
    + \{ \gamma(x,\eta) \DG r(x,\eta) + F(x,\eta) \} \widetilde{g}(x)
\end{align*}
Now, note that
\begin{align*}
&
    \int_0^\infty g(x,u) \int_u^\infty (g(x,v)-1) dv du - \int_0^\infty u(g(x,u)-1)du
\\ & \qquad =
    \int_0^\infty g(x,u) \int_u^\infty (g(x,v)-1) dv du - \int_0^\infty \int_v^\infty (g(x,u) -1) du dv
\\ & \qquad =
    \int_0^\infty (g(x,u)-1) \int_u^\infty (g(x,v)-1) dv du
\\ & \qquad =
    \widetilde{g}(x)^2/2 .
\end{align*}

Adding these together, we get that
\begin{align*}
    &
    \int_0^\infty (\ell_1(x,u) + \ell_2(x,u) + \ell_3(x,u)) du
    \\ &\qquad=
    \gamma(x,\eta) \DG(\widetilde{g}(\cdot)r(\cdot,\eta))(x)
    + F(x,\eta) \widetilde{g}(x) 
    + \alpha \gamma(x,\eta) r(x,\eta) \widetilde{g}(x)^2 ,
\end{align*}
which agrees with \eqref{eqn:pgen_infty}, as desired.
\end{proof}

\section{Technical Lemmas}

\subsection{Constraints on kernel widths}

\begin{lemma}
    \label{lem:gamma_bound}
    Suppose the first three conditions of Assumptions~\ref{def:model_setup} hold,
    and furthermore the kernels $\rho_r = p_{\epsilon_r^2}$ and $\rho_\gamma = p_{\epsilon_\gamma^2}$
    are each Gaussian with standard deviations $\epsilon_r$ and $\epsilon_\gamma$ respectively.
    Let $\lambda = \sup_x \sup_{y : \|y\| = 1} y^T \covq(x) y$ be the largest eigenvalue of $\covq(x)$
    across all $x$.
    If $\epsilon_{r}^2 + \frac{2\lambda}{\theta} < \epsilon_{\gamma}^2$, then 
    there is a $C < \infty$ such that
    for all $x\in\IR^d$, $\eta\in\measures$,
    \begin{equation}
        \label{eqn:first_moment_rho}
        \left|\theta \int_{\IR^d}
            ( \smooth{r} \eta(y)-\smooth{r} \eta(x) )
        q_\theta(x,dy)\right|
        \leq C \smooth{\gamma} \eta(x)
    \end{equation}
    and
    \begin{equation}
        \label{eqn:second_moment_rho}
        \theta \int_{\IR^d}
                \left( \smooth{r} \eta(y) - \smooth{r} \eta(x) \right)^2
        q_\theta(x,dy)
        \leq C \left(\smooth{\gamma}\eta(x) \right)^2 .
    \end{equation}
\end{lemma}

Note that the right hand side of each is the average density
over a wider region (since $\epsilon_\gamma > \epsilon_r$).
The key assumption here 
is that the spatial scale over which local density affects birth rate
is larger than the scale over which it affects establishment.
In the simple case of $\meanq = 0$ and $\covq = \sigma^2 I$,
the condition is simply that $\epsilon_r^2 + 2\sigma^2 / \theta < \epsilon_\gamma^2$.
This gives a yet more concrete situation in which 
Condition~\ref{control through gamma} of Lemma~\ref{lem:conditions on r} holds.

\begin{proof}[Proof of Lemma \ref{lem:gamma_bound}:]

    First we prove \eqref{eqn:first_moment_rho}.
    Recall that $\smooth{r} \eta(x) = \int p_{\epsilon^2_r}(x - w) \eta(dw)$,
    where $p_t$ is the density of a Gaussian with mean 0 and variance $t$,
    so that applying Fubini, \eqref{eqn:first_moment_rho} is
    \begin{align*}
        \left|
        \int_{\IR^d}
        \theta
            \int_{\IR^d}
                ( p_{\epsilon_r^2}(y - w) - p_{\epsilon_r^2}(x - w) )
            q_\theta(x, dy)
        \eta(dw)
        \right| .
    \end{align*}
    Write $p_{s, x}(\cdot)$ for the density of a Gaussian
    with mean $s \meanq(x)$ and covariance $\epsilon_r^2 I + s \covq(x)$,
    so that $\int p_{\epsilon_r^2}(y - w) q_\theta(x, dy) = p_{1/\theta, x}(w-x)$.
    It therefore suffices to show that for all $x$ and $w \in \IR^d$,
    there exists $K$ such that
    \begin{align*}
        \left|
        \theta
            \int_{\IR^d}
                ( p_{\epsilon_r^2}(y-w) - p_{\epsilon_r^2}(x-w) )
            q_\theta(x, dy)
        \right|
        &=
        \theta \left|
            p_{1/\theta, x}(w-x) - p_{0, x}(w-x)
        \right|
        \le
        K p_{\epsilon^2_\gamma}(w-x) .
    \end{align*}
    However,
    $\theta( p_{1/\theta, x}(z) - p_{0, x}(z) )
    = \partial_s p_{s, x}(z)$ for some $0 \le s \le 1/\theta$.
    Write $\Gamma(s,x) = \epsilon_r^2 I + s \covq(x)$,
    so that
    \begin{align*}
        p_{s, x}(z)
        &=
        \frac{1}{\left(2 \pi |\Gamma(s,x)|\right)^{d/2}}
        \exp\left(
            -\frac{1}{2} (z - s\meanq(x))^T \Gamma(s,x)^{-1} (z - s\meanq(x))
        \right) ,
    \end{align*}
    and note that if $\lambda_i$ are the eigenvalues of $\covq(x)$
    then $|\Gamma(s, x)| = \prod_i (\epsilon_r^2 + s \lambda_i)$,
    and
    $\partial_s |\Gamma(s, x)| = \sum_i \lambda_i |\Gamma(s, x)| / (\epsilon_r^2 + s \lambda_i)$.
    Therefore,
    \begin{align*}
        \partial_s p_{s, x}(z)
        &=
        \bigg(
            \meanq(x)^T \Gamma(s,x)^{-1} (z - s \meanq(x))
            +
            (z - s \meanq(x))^T
            \Gamma(s,x)^{-1} \covq(x) \Gamma(s,x)^{-1}
            (z - s \meanq(x))
        \\ & \qquad \qquad \qquad {}
            -
            \sum_i \frac{\lambda_i}{\epsilon_r^2 + s \lambda_i}
        \bigg)
        p_{s, x}(z) ,
    \end{align*}
    where $z^T$ is the transpose of $z$.
    This implies that
    \begin{align*}
        \frac{
        \theta
            \int_{\IR^d}
                ( p_{\epsilon_r^2}(y-w) - p_{\epsilon_r^2}(x-w) )
            q_\theta(x, dy)
        }{
            p_{\epsilon^2_\gamma}(x-w) .
        }
        =
        h(x-w) e^{k(x-w)},
    \end{align*}
    where $h(z)$ and $k(z)$ are quadratic polynomials in $z$
    whose coefficients depend on $s$ and $x$ but are uniformly bounded,
    and
    \begin{align*}
        k(z)
        =
        \frac{1}{2\epsilon_\gamma^2} \|z\|^2
        -\frac{1}{2} (z - s\meanq(x))^T \Gamma(s,x)^{-1} (z - s\meanq(x)) ,
    \end{align*}
    Since $\inf_z z^T \Gamma(s,x)^{-1} z / \|z\|^2 = 1 / (s \lambda(x) + \epsilon_r^2)$,
    where $\lambda(x) = \sup z^T \covq(x) z / \|z\|^2$ is the largest eigenvalue of $\covq(x)$,
    this is negative for all $z$ outside a bounded region, and so
    equation~\eqref{eqn:first_moment_rho} follows from the assumption
    that $\epsilon_r^2 + 2 \sup_x \lambda(x)/\theta < \epsilon_\gamma^2$.
    (Note that we do not yet need the factor of 2.)

    Next we prove equation~\eqref{eqn:second_moment_rho}, in a similar way.
    Again applying Fubini,
    \begin{align*}
        &
        \theta \int_{\IR^d}
                \left( \smooth{r} \eta(y) - \smooth{r} \eta(x) \right)^2
        q_\theta(x,dy)
        \\ &\qquad
        =
        \int_{\IR^d} \int_{\IR^d}
        \theta \int_{\IR^d}
            ( p_{\epsilon_r^2}(y - w) - p_{\epsilon_r^2}(x - w) )
            ( p_{\epsilon_r^2}(y - v) - p_{\epsilon_r^2}(x - v) )
        q_\theta(x, dy)
        \eta(dv) \eta(dw) ,
    \end{align*}
    and so as before, equation~\eqref{eqn:second_moment_rho} will follow if
    the integrand is bounded by $K p_\gamma(x-w) p_\gamma(x - v)$.
    Now, let $Y_1$, $Y_2$, and $Z$ be independent $d$-dimensional Gaussians with mean zero,
    where $Y_1$ and $Y_2$ have covariance $\epsilon^2_r I$,
    and $Z$ has covariance $\covq(x)$.
    Write $p_{s,t,x}(\cdot, \cdot)$ for the joint density of
    $Y_1 + \sqrt{s} Z + s \meanq(x)$ and $Y_2 + \sqrt{t} Z + t \meanq(x)$.
    Then, observe that
    \begin{align*}
        &
        \theta \int_{\IR^d}
            ( p_{\epsilon_r^2}(y - w) - p_{\epsilon_r^2}(x - w) )
            ( p_{\epsilon_r^2}(y - v) - p_{\epsilon_r^2}(x - v) )
        q_\theta(x, dy)
        \\ &\qquad
        =
        \theta \left(
            p_{1/\theta, 1/\theta, x}(x-w, x-v)
            - p_{0, 1/\theta, x}(x-w, x-v)
        \right. \\ &\qquad \qquad \left. {}
            - p_{1/\theta, 0, x}(x-w, x-v)
            + p_{0, 0, x}(x-w, x-v)
        \right) 
        \\ &\qquad 
        =
            \partial_s p_{s, 1/\theta, x}(x-w, x-v)
            - \partial_t p_{0, t, x}(x-w, x-v) ,
    \end{align*}
    for some $0 \le s, t \le 1/\theta$.
    As before,
    \begin{align*}
        \frac{
            \theta \int_{\IR^d}
                ( p_{\epsilon_r^2}(y - w) - p_{\epsilon_r^2}(x - w) )
                ( p_{\epsilon_r^2}(y - v) - p_{\epsilon_r^2}(x - v) )
            q_\theta(x, dy)
        }{
            p_{\epsilon_\gamma^2}(x - w)
            p_{\epsilon_\gamma^2}(x - v)
        }
        =
        h(x-w, x-v) e^{ k(x-w, x-v) },
    \end{align*}
    where $h(z_1, z_2)$ is a polynomial with uniformly bounded coefficients and
    \begin{align*}
        k(z_1, z_2)
        =
        (\|z_1\|^2 + \|z_2\|^2)/(2 \epsilon_\gamma^2)
            - \frac{1}{2} [z_1, z_2]^T \Gamma(s, t, x)^{-1} [z_1, z_2] ,
    \end{align*}
    where $[z_1, z_2]$ is the $\IR^{2d}$ vector formed by concatenating $z_1$ and $z_2$,
    and $\Gamma(s, t, x)$ is the block matrix
    \begin{align*}
        \Gamma(s, t, x)
        =
        \left[
        \begin{array}{cc}
            \epsilon_r^2 I + s \covq(x) & \sqrt{st} \covq(x) \\
            \sqrt{st} \covq(x) & \epsilon_r^2 I + t \covq(x) \\
        \end{array}
        \right] .
    \end{align*}
    If $\covq(x) u = a u$ for some $a \in \IR$,
    then $[u \sqrt{s}, u \sqrt{t}]$ is an eigenvector of $\Gamma(s, t, x)$
    with eigenvalue $\epsilon^2_r + (s+t) a$,
    and $[u \sqrt{t}, - u \sqrt{s}]$ is an eigenvector of $\Gamma(s, t, x)$
    with eigenvalue 0.
    This implies the largest eigenvalue of $\Gamma(s, t, x)$
    is equal to $\epsilon^2_r + (s+t) \lambda(x)$,
    where $\lambda(x)$ is again the largest eigenvalue of $\covq(x)$.
    Therefore, if $s + t \le 2 / \theta$,
    \begin{align*}
        &
        (\|z_1\|^2 + \|z_2\|^2) / \epsilon^2_\gamma
            - [z_1, z_2]^T \Gamma(s, t, x)^{-1} [z_1, z_2]
        \\ &\qquad \le
        (\|z_1\|^2 + \|z_2\|^2) \left(
            \frac{1}{\epsilon^2_\gamma}
            - \frac{1}{\epsilon^2_r +  2\lambda(x)/\theta}
        \right) ,
    \end{align*}
    which is negative by assumption.
    Therefore, there is a $K$ such that
    \begin{align*}
        \frac{ \left|
            \partial_s p_{s, 1/\theta, x}(x-w, x-v)
            - \partial_t p_{0, t, x}(x-w, x-v)
        \right| }{
            p_{\epsilon^2_\gamma}(x - w)
            p_{\epsilon^2_\gamma}(x - v)
        }
        \le K
    \end{align*}
    for all $\theta > 1$ and all $x$, $v$, and $w \in \IR^d$,
    proving equation~\eqref{eqn:second_moment_rho} and hence the lemma.
\end{proof}

\subsection{Tightness of processes}
\label{apx:aldous rebolledo}

Here we record, for completeness, the fact used above
that tightness for a family of processes, if determined by the Aldous-Rebolledo criterion,
extends to sums and products of those processes.
We first record for reference
one version of the Aldous-Rebolledo criteria for tightness of a sequence real-valued processes
(as it appears in Theorem~1.17 of \citet{etheridge2000introduction};
see also Theorem~4.13 of \citet{jacod/shiryaev:2013}):

\begin{theorem}[\citet{rebolledo:1980}]
    \label{thm:aldous_rebolledo}
    Let $\{Y^{(n)}\}_{n \ge 1}$ be a sequence of real-valued processes
    with c\`adl\`ag paths.
    Suppose that the following conditions are satisfied.
    \begin{enumerate}
        \item For each fixed $t \in [0,T]$, $\{Y_t^{(n)}\}_{n \ge 1}$  is tight.
        \item Given a sequence of stopping times $\tau_n$, bounded by $T$,
            for each $\epsilon > 0$ there exists $\delta > 0$ and $n_0$ such that
            \begin{align*}
                \sup_{n \ge n_0}
                \sup_{\theta \in [0, \min(\delta,T-\tau_n)]}
                \IP\left\{ \left|
                    Y^{(n)}_{\tau_n + \theta} - Y^{(n)}_{\tau_n}
                \right| > \epsilon
                \right\}
                \le \epsilon .
            \end{align*}
    \end{enumerate}
    Then the sequence $\{(Y^{(n)}_t)_{t=0}^T\}_{n \ge 1}$ is tight.
\end{theorem}

\begin{lemma}
    \label{lem:product_tightness}
    Let $\{X^{(n)}\}_{n \ge 1}$ and $\{Y^{(n)}\}_{n \ge 1}$ be sequences of jointly defined
    real-valued processes with c\`adl\`ag paths
    satisfying the conditions of Theorem~\ref{thm:aldous_rebolledo}.
    Then $\{X^{(n)}Y^{(n)}\}_{n \ge 1}$ and $\{X^{(n)} + Y^{(n)}\}_{n \ge 1}$ also
    satisfy the conditions of Theorem~\ref{thm:aldous_rebolledo}.
\end{lemma}

By ``jointly defined'' we mean that $X^{(n)}$ and $Y^{(n)}$ are defined on the same
probability space, so that the products and sums make sense.

\begin{proof}[Proof of Lemma~\ref{lem:product_tightness}:]
    The proof for $X^{(n)} + Y^{(n)}$ is similar to but more straightforward than
    for $X^{(n)} Y^{(n)}$, so on only prove the Lemma for the latter.

    First, note that for any $\epsilon > 0$, by tightness of $(X^{(n)}_t)_{n\ge0}$ and $(Y^{(n)}_t)_{n\ge0}$
    there is a $K$ such that
    $\IP\{ X^{(n)}_t > \sqrt{K} \}$ and $\IP\{ Y^{(n)}_t > \sqrt{K} \}$
    are both less than $\epsilon/2$,
    and hence
    \begin{align*}
        \IP\{
            X^{(n)}_t Y^{(n)}_t > K
        \}
        \le
        \IP\{
            X^{(n)}_t > \sqrt{K}
        \}
        +
        \IP\{
            Y^{(n)}_t > \sqrt{K}
        \} 
        \le \epsilon.
    \end{align*}
    Therefore, $(X^{(n)}_t Y^{(n)}_t)_{n \ge 0}$ is tight.

    Next, note that for $0 \le \tau_n \le T$,
    \begin{align*}
        &
        \sup_{0 \le \theta \le \min(\delta, T-\tau_n)} 
        \left|
            X^{(n)}_{\tau_n + \theta} Y^{(n)}_{\tau_n + \theta} 
            -
            X^{(n)}_{\tau_n} Y^{(n)}_{\tau_n} 
        \right|
        \\ &\qquad 
        \le
        \sup_{0 \le \theta \le \min(\delta, T-\tau_n)} 
        \left| X^{(n)}_{\tau_n + \theta} \right|
        \left|
            Y^{(n)}_{\tau_n + \theta} 
            -
            Y^{(n)}_{\tau_n} 
        \right|
        +
        \left|
            X^{(n)}_{\tau_n + \theta}
            -
            X^{(n)}_{\tau_n}
        \right|
        \left| Y^{(n)}_{\tau_n} \right|
        \\ &\qquad 
        \le
        \sup_{0 \le t \le T}
            \left| X^{(n)}_t \right|
        \sup_{0 \le \theta \le \min(\delta, T-\tau_n)} 
        \left|
            Y^{(n)}_{\tau_n + \theta} 
            -
            Y^{(n)}_{\tau_n} 
        \right|
        +
        \sup_{0 \le \theta \le \min(\delta, T-\tau_n)} 
        \left|
            X^{(n)}_{\tau_n + \theta}
            -
            X^{(n)}_{\tau_n}
        \right|
        \sup_{0 \le t \le T}
        \left| Y^{(n)}_t \right| ,
    \end{align*}
    so that for any $C$,
    \begin{align} \label{eqn:xy_bound}
        \begin{split}
        &
        \IP\left\{
        \sup_{0 \le \theta \le \min(\delta, T-\tau_n)} 
        \left|
            X^{(n)}_{\tau_n + \theta} Y^{(n)}_{\tau_n + \theta} 
            -
            X^{(n)}_{\tau_n} Y^{(n)}_{\tau_n} 
        \right|
        > \epsilon
        \right\}
    \\ &\qquad \le
        \IP\left\{
            \sup_{0 \le t \le T}
                \left| X^{(n)}_t \right|
            > C
        \right\}
    +
        \IP\left\{
        \sup_{0 \le \theta \le \min(\delta, T-\tau_n)} 
        \left|
            Y^{(n)}_{\tau_n + \theta} 
            -
            Y^{(n)}_{\tau_n} 
        \right|
            > \epsilon/C
        \right\}
    \\ &\qquad {} +
        \IP\left\{
        \sup_{0 \le \theta \le \min(\delta, T-\tau_n)} 
        \left|
            X^{(n)}_{\tau_n + \theta} 
            -
            X^{(n)}_{\tau_n} 
        \right|
            > \epsilon/C
        \right\}
    +
        \IP\left\{
            \sup_{0 \le t \le T}
                \left| Y^{(n)}_t \right|
            > C
        \right\}
        \end{split}
    \end{align}
    Now, since $\max_{0 \le t \le T} X^{(n)}_t$ is tight (and likewise for $Y$)
    (see, e.g., Remark 3.7.3 in \citet{ethier/kurtz:1986}),
    we may choose a $C \ge 4$ for which
    $$
    \IP\left\{
        \sup_{0 \le t \le T}
            \left| X^{(n)}_t \right|
        > C
    \right\} \le \frac{\epsilon}{4} .
    $$
    Similarly, by assumption we can choose a $\delta$ for which
    $$
        \IP\left\{
        \sup_{0 \le \theta \le \min(\delta, T-\tau_n)} 
        \left|
            X^{(n)}_{\tau_n + \theta} 
            -
            X^{(n)}_{\tau_n} 
        \right|
            > \epsilon/C
        \right\}
        \le \frac{\epsilon}{C} .
    $$
    If we choose $C$ and $\delta$ that do this for both $X^{(n)}$ and $Y^{(n)}$,
    then each of the terms in equation~\eqref{eqn:xy_bound} are bounded by $\epsilon/4$,
    and condition (2) is satisfied for the product process.
\end{proof}


\end{document}